\documentclass[reqno]{amsart}

\usepackage{amsmath}
\usepackage{amssymb}

\def\enddemo{\end{proof}}
\newtheorem{theorem}{Theorem}[section]
\newtheorem{proposition}{Proposition}[section]
\newtheorem{definition}{Definition}[section]
\newtheorem{corollary}{Corollary}[section]
\newtheorem{lemma}{Lemma}[section]
\newtheorem{example}{Example}[section]
\newtheorem{remark}{Remark}[section]
\begin{document}

\title[A stochastic fixed-point theorem]
{A fixed-point theorem for local operators with applications to stochastic equations}

\author{Arcady Ponosov}
\thanks{Arcady Ponosov, Faculty of Sciences and Technology,
Norwegian University of Life Sciences, P. O. Box 5003, NO-1432 {\AA}s,
Norway, E-mail: arkadi@nmbu.no.
}


%\maketitle

\begin{abstract}
We study weak and strong solutions of nonlinear non-compact operator equations in abstract spaces of adapted random points. The main result of the paper is similar to Schauder's fixed-point theorem for compact operators. The illustrative examples  explain how this analysis can be applied to stochastic differential equations.
\end{abstract}

\maketitle

\section{Introduction}\label{sec-1}

Nonlinear operators studied in this paper possess \textit{the property of locality}. A mapping between two spaces consisting of continuous functions has this property if the value of the output function at a given point is completely determined by the values of the input function in an arbitrarily small neighbourhood of this point. Typical examples are superposition operators, differentiations and their combinations. However, this definition is no longer valid for spaces of measurable functions, where elements are equivalence classes and not individual functions. In the latter case, the formal definition of locality was suggested by I. V. Shragin in the paper \cite{Shragin}. It says that if two equivalence classes coincide on a set $A$, then their images must coincide on the same set. If we replace the equivalent classes with their representatives, then we have to add the expression "almost everywhere" to this definition.

It can be shown that Shragin's definition covers the above mentioned examples. On the other hand, it also covers stochastic integrals if they can be defined as limits of finite sums and, by this, stochastic operators, which do not directly contain global characteristics of stochastic processes, like expectation, covariance, distributions etc.
%The property of locality of an operator $h$ says then that the transformed bunch $h\mathcal T =\{(hx)(\omega ,\cdot )\}$ is only dependent on the trajectories from
%$\mathcal T$ and independent of the complementary bunch  $\{x(\omega ,\cdot )\}$ ($\omega\notin A)$.

This paper develops a fixed-point theory for general local operators defined on spaces of adapted random points in abstract separable Banach spaces. The main result of the paper states, roughly speaking,  that if a local and continuous operator, defined on a special set of adapted random points, maps this set into its tight subset, then it has at least one \textit{weak fixed point}, a random point on an expanded probability space. This fixed-point theorem was first announced in the author's paper \cite{Pon-1}, although without a proof. The main objective of this report is to provide a detailed proof of this result.

Note that this theory is not about a simple special case of local operators given by superpositions $x(\cdot)\mapsto F(\cdot, x(\cdot))$ that are generated by random, a.s. continuous maps $F(\omega, \cdot)$. Combining the theory of compact operators with the technique of measurable selections it is not difficult to develop a fixed-point theory for such operators, but they do not cover most interesting stochastic differential equations. The local operators considered in this report do cover stochastic integrals and equations, and this is demonstrated in a number of examples.

The paper is organized as follows. In Section \ref{sec-Prelim} a simplified version of the main fixed-point theorem, which can be directly used in applications, is formulated. Here we replace abstract Banach spaces by two examples of functional spaces.
The general case is considered in Section \ref{sec-inf-dim-fixed-point}, while Sections \ref{sec-3} and \ref{sec-extensions-operators} contain necessary definitions and auxiliary results, the proofs of which are moved to Appendix B.
An overview of the terminology and the notation can be found in Appendix A. Appendix D contains illustrative examples, some of which are based on the propositions collected and proven in Appendix C.

\section{Local operators and a simplified version of the fixed-point theorem}\label{sec-Prelim}

Let
\begin{equation}\label{eq-probability-space}
\mathcal S=(\Omega , {\mathcal F}, P)
\end{equation}
 be a complete probability space.  The expectation, the integral w.r.t. the measure $P$, is denoted by $E$, and the abbreviation a.s. means almost surely (i.e. almost everywhere) with respect to $P$.

We will use the following notation:
%\begin{itemize}
$I_A$ is the indicator of a set $A$;
$R^n$ is the $n$-dimensional Euclidean space with some norm $|.|$; $X$, $Y$ are separable Banach spaces with the norms $\|.\|_X$ and $\|.\|_Y$, respectively; we also put $B_r=\{x\in X: \ \|x\|_{X}\le r\}$.

  For any separable metric space $M$ the set ${\mathcal P}(M, \mathcal S)$ consists of all equivalence classes $[x]$ of $\mathcal F$-measurable functions $x:\Omega
\to M$ (also referred to as random points in $M$). Convergence in probability defines a metrizable topology on ${\mathcal P}(M, \mathcal S)$. In this paper we will use the metric
$$d_M(x,y)=E\min\{\rho(x,y); 1\},$$ where $\rho$ is the distance on $M$. If $M$ is complete, then ${\mathcal P}(M, \mathcal S)$ is complete as well. Any convergent in probability sequence contains an a.s. convergent subsequence. In particular, this implies that the topology on ${\mathcal P}(M, \mathcal S)$ does not depend on the choice of any equivalent distance on $M$. If $M=X$ is a separable Banach space $X$, then the set ${\mathcal P}(X, \mathcal S)$ is a linear metric, but not locally convex, space.

Below we usually disregard the difference between the equivalence classes $[x]$ and their particular representatives $x$ writing (somewhat unprecisely) $x\in {\mathcal P}(M,\mathcal S)$ instead of $[x]\in {\mathcal P}(M,\mathcal S)$.
We will also sometimes write ${\mathcal P}(M)$ instead of ${\mathcal P}(M,\mathcal S)$ if the probability space $\mathcal S$ is fixed and if it does not cause misunderstandings.

Notice that if $V$ is a closed resp. convex subset of $X$, then ${\mathcal P}(V, \mathcal S)$ is a closed resp. convex subset of ${\mathcal P}(X, \mathcal S)$. Bounded subsets $\mathcal A$ of the space ${\mathcal P}(X, \mathcal S)$ can be described as follows: for any $\varepsilon>0$ there is a ball $B_r$ in $X$ such that
$P\{x\notin B_r\}<\varepsilon$
for any $x\in \mathcal A$.

 It is assumed in the definition below that two equivalence classes $[x], [y]\in {\mathcal A}$ coincide on a set $A\subset\Omega$, i.e. $[x]|_A=[y]|_A$, if $x(\omega )=y(\omega )$ {for} almost all $\omega \in A$. Clearly, this definition is independent of the choice of the representatives $x$ and $y$.

\begin{definition}
\label{def-local}
Let $\mathcal A\subset \mathcal P (X, \mathcal S)$.
An operator $h: \ {\mathcal A}\to {\mathcal P}(Y, \mathcal S)$
 is called \underline{local} if
\[
[x]|_A=[y]|_A \ \; \mbox{implies}
\ \; h[x]|_A=h[y]|_A
\]
for any $[x], [y]\in {\mathcal A}$ and $A\subset \Omega$.
\end{definition}

\begin{remark}\label{rem-def-local-operators}
  If the property of locality is only valid for all $A\in \mathcal F$, then it also valid for all $A\subset\Omega$. Indeed, for any representatives $x$ and $y$ of the classes $[x]$ and $[y]$, respectively, the set $B=\{\omega\in\Omega: \ x(\omega)=y(\omega)\}$ belongs to $\mathcal F$ and satisfies $P(A- B)=0$. Hence, the equivalence classes $h[x]$ and $h[y]$ coincide on $B$ and thus on $A$.
\end{remark}

Note that any local operator $h$ can be naturally defined on the set of all representatives of the equivalence classes belonging to   ${\mathcal A}$ if we put $hx$  to be an arbitrary representative of the class $h[x]$. The property of locality becomes then
\[
x(\omega )=y(\omega )\ \;\mbox{for}\ \;\omega \in A\ \;%
\mbox{a.s.} \ \; \mbox{implies}
\ \; hx(\omega )=hy(\omega )\ \;\mbox{for}\ \;\omega \in A\;
\mbox{a.s.} \ \ (\forall A\subset\Omega).
\]

\begin{remark}
  In this report, we only study local operators that are continuous in the topology of convergence in probability. Throughout the paper we will use the abbreviation \textit{LC} for such operators.

\end{remark}

\begin{remark}
The superposition operator $$h_f: \ {\mathcal P}(X, \mathcal S)\to {\mathcal P}(Y, \mathcal S), \ \ \mbox{defined by} \ \ (h_fx)(\omega)=f(\omega, x(\omega)),$$ where $f: \Omega\times X\to Y$ is a random function, is local. Clearly, this operator is well-defined, as $x(\omega)=\tilde x(\omega)$ a.s. implies $(h_fx)(\omega)=(h_f\tilde x)(\omega)$ a.s. If, in addition, the function $f$ is Carath\'{e}odory, i.e. it is measurable in $\omega\in\Omega$ for all $x\in X$ and continuous in $x\in X$ for almost all $\omega\in\Omega$, then $h_f$ is continuous in probability.

The It\^{o} integral is another example of an LC operator. More examples can be found in Appendix D.
\end{remark}

In this paper $T$ usually stands for an arbitrary linearly ordered set containing its maximal element, see Appendix A. A typical example is  ${T}=[a, b]$, and this is assumed in the remaining part of this section. In addition, we suppose that $X=C({T})$ or $X=L^r({T})$ ($1\le r < \infty)$.
%The case of more general $T$ and $X$ is studied in the forthcoming sections.

Let
\begin{equation}\label{eq-filtration}
({\mathcal F}_t)_{t\in {T}}
\end{equation}
 be a filtration on the probability space (\ref{eq-probability-space}), i. e. a nondecreasing family of
$\sigma$-subalgebras of ${\mathcal F}$,  all $\sigma$-subalgebras being complete w.r.t. $P$, i.e. containing all subsets of zero measure. The probability space (\ref{eq-probability-space}) with a filtration (\ref{eq-filtration}) on it is usually called \textit{a stochastic basis}.

An $\mathcal F\otimes \mbox{Bor} (T)$--measurable stochastic process $\xi(t)=\xi(\omega, t)$, $t\in T$, is called ${\mathcal F}_t$--adapted \cite{Oks} if $\xi(\cdot, t)$ is ${\mathcal F}_t$-measurable for all $t\in {T}$. Given a stochastic basis $\mathcal B$, we denote by $\mathcal Pa(X, \mathcal B)$ the set of all (equivalence classes of) ${\mathcal F}_t$-adapted stochastic processes whose trajectories a.s. belong to the space $X=L^r(T)$ or $C(T)$. Any equivalence class consists in this case of all indistinguishable stochastic processes \cite{Oks}. The inclusion $\mathcal Pa(X, \mathcal B)\subset {\mathcal P}(X, \mathcal S)$ induces a natural topology on the aforementioned space.

Recall that a set ${\mathcal K}\subset {\mathcal P}(X, \mathcal S)$ is called \textit{tight}
 if for any $\epsilon >0$ there exists a compact set $Q\subset X$ such that
$P\{\omega \,:\,x(\omega ) \notin Q\}<\epsilon $ whenever $x\in
{\mathcal K}$.
We say that
an operator $h: \ {\mathcal A}\to {\mathcal P}(X, \mathcal S)$ (${\mathcal A}\subset {\mathcal P}(X, \mathcal S)$)
 is \textit{tight-range} if 1) it maps ${\mathcal A}$ into a tight subset of ${\mathcal P}(x, \mathcal S)$ and 2) it is uniformly continuous on any
tight subset of ${\mathcal A}$. If $h$ only maps bounded subsets of ${\mathcal A}$ into tight subsets, then the operator $h$ is called  \textit{tight}.

This definition generalises the notions of compact and compact-range operators:
if $\Omega $ shrinks into a singleton, then the space ${\mathcal P}(X)$ coincides with $X$ and tight subsets become precompact in $X$. In this case, uniform continuity on compact subsets (and thus on their subsets) follows from continuity.

\begin{definition}\label{def-ext-general}
A stochastic basis $\mathcal B^*=(\Omega^*,{\mathcal F}^*,{\mathcal F}_t^*,P^*)$ is \textit{an expansion}
of the stochastic basis $\mathcal B =(\Omega ,{\mathcal F},{\mathcal F}_t, P)$ if
there exists a $({\mathcal F}^*,{\mathcal F})$- measurable surjection $c:\Omega^*\to \Omega$ such that
\begin{enumerate}
  \item $P^*c^{-1}=P$;
  \item $c^{-1}({\mathcal F}_t)\subset {\mathcal F}_t^* \hskip 0.3 truecm
 ( \forall t).$
\end{enumerate}
\end{definition}
Note that $\mathcal{P}a(X, \mathcal B)$ can be naturally identified with a linear topological subspace of the space $\mathcal{P}a(X, \mathcal B^*)$.

Expansions preserving the martingale property are of key interest in the theory of weak solutions of stochastic differential equations \cite{Jacod-1}. In particular, for the standard Wiener process $W(t)$ on $\mathcal B$ the process $W^*(t)=W(t)\circ c$ remains Wiener on any such expansion. In this paper, we only use a special version of the expansions preserving the martingale property, which we call \textit{Young expansions}. In this case, the disintegration of the probability measure $P^*$ is a Young measure, i.e. the weak limit of random Dirac measures generated by adapted random points, see Definition \ref{def-Young-extension}.

Given an LC operator $h\!\!: \!\mathcal{P}a(X, \mathcal B)\!\to\! {\mathcal P}(X, \mathcal S)$ and an expansion
$\mathcal B^*\!=(\Omega^*\!,{\mathcal F}^*\!,{\mathcal F}_t^*\!, P^*\!)$ of the given stochastic basis $\mathcal B =(\Omega ,{\mathcal F},{\mathcal F}_t, P)$, we say that an LC operator $h^*$ defined on $\mathcal{P}a(X, \mathcal B^*)$ is \textit{an extension} of the operator $h$ if the restriction of $h^*$ to $\mathcal{P}a(X, \mathcal B)$ coincides with $h$. Only extensions preserving locality and continuity are studied in this paper.

If an LC operator $h$ admits an LC extension, then this extension is unique, see Theorem \ref{prop-unique-local-extension}. Existence of LC extensions is a more delicate issue, see Subsection \ref{sec-LC-extensions}.

Let $h: \ \mathcal{P}a(X, \mathcal B)\to {\mathcal P}(X)$, where $X=C({T})$ or $X=L^r({T})$ ($1\le r < \infty$),  be an LC operator. If there exists a Young expansion
$\mathcal B^*=(\Omega^*,{\mathcal F}^*,{\mathcal F}_t^*,P^*)$ of the stochastic basis $\mathcal B =(\Omega ,{\mathcal F},{\mathcal F}_t, P)$ and $x^*\in \mathcal{P}a(X, \mathcal B^*)$ such that $h^*x^*=x^*$ $P^*$-a.s., then $x^*$ will be called a weak fixed point of $h$.

The following fixed-point theorem is a particular case of the main results presented in Section \ref{sec-inf-dim-fixed-point}. This version is formulated explicitly, because it may be directly useful in many applications.
\begin{theorem}\label{th-fixed-point}
  Let $X=C({T})$ or $X=L^r({T})$ ($1\le r < \infty$) and $h: \mathcal{P}a (X, \mathcal B)\to \mathcal{P}a (X, \mathcal B)$ be a local and tight-range operator. Then $h$ has at least one weak fixed point $x^*\in \mathcal{P}a (X, \mathcal B^*)$ for some Young expansion $\mathcal B^*$ of the stochastic basis $\mathcal B$.

  If the operator $h$ has at most one weak fixed point for any Young expansion $\mathcal B^*$ of $\mathcal B$, then each
    weak fixed point will be equivalent to a unique strong, i.e. belonging to the space $\mathcal{P}a (X, \mathcal B)$, solution of the equation $hx=x$.
 \end{theorem}
 \begin{proof}
   See Corollary \ref{cor-fixed-point-general}.
 \end{proof}

\section{Some properties of local operators in the spaces of adapted random points}\label{sec-3}
Starting with this section we assume that $X$ is an arbitrary separable Banach space and $T$ is an abstract set of indices. One of the aims is to define adapted random points in $X$ with respect to stochastic bases over $T$.

\subsection{Adapted random points in abstract Banach spaces}\label{sec-adapted-points}

Let $T$ be a linearly ordered set containing a maximal element $b.$ Line intervals $T=[a,b]$ serve as examples of such sets. We are also given a complete probability space (\ref{eq-probability-space}) and
a filtration (\ref{eq-filtration}) on it, i.e. a nondecreasing family
of complete $\sigma $-subalgebras ${\mathcal F}_t\subset {\mathcal F}$ ($t\in T$) indexed by elements of the set $T$.
The quadruple
\begin{equation}\label{eq-stoch-basis}
\mathcal B =(\Omega , {\mathcal F}, ({\mathcal F}_t)_{t\in {T}},P)
\end{equation}
 is addressed as a stochastic basis on the probability space (\ref{eq-probability-space}) over the set $T$.

\begin{definition}\label{def-projective-system}
By a projective system of linear topological spaces over $T$ we understand a triple $\mathcal X=(X_t, p^{ut}, {T})$, where $X_t$ are linear topological spaces ($t\in{T}$) and $p^{ut}:\ X_t\to X_u$ ($t, u\in{T}, \ t\ge u$) are linear continuous surjective maps satisfying the property $$p^{vu}\circ p^{ut}=p^{vt} \ \ \mbox{for all} \ \ t,u,v\in{T}, \ t\ge u\ge v.$$
\end{definition}
\begin{remark}\label{rem-projective-systems}
The following complements Definition  \ref{def-projective-system}:
\begin{enumerate}
\item Projective systems are also known as inverse systems in the literature; the maps $p^{ut}$ are usually called bonding maps.
  \item The definition implies that $p^{tt}$ are the identity maps on the respective spaces $X_t$ for all $t\in {T}.$
  \item In most propositions below we consider projective systems of separable Banach spaces, but in connection with expansions of probability spaces  projective systems of separable Fr\'{e}chet spaces may be necessary, see Example \ref{ex-iterated-stoch-bases}.
  \item Below we systematically use the simplifications $X\equiv X_{b}$ and $p^t\equiv p^{tb}$.
  \item As for any $x\in\mathcal{P}a(\mathcal X, \mathcal B)$, the map $p^{b}x=x$ must be $\mathcal F_b$-measurable, we can always assume, if necessary, that $\mathcal F=\mathcal F_b.$
\end{enumerate}
\end{remark}

An important example of a projective system is described in
\begin{definition}\label{def-Euclidean}
If  $T= T_m\equiv \{0,...,m\}$, $X_t\equiv E_i$ ($t=i\in T_m,$ $\mbox{dim}\, E_i=i$) are linear subspaces of the $m$-dimensional Euclidean space $E=E_m$, $E_j\subset E_i$ ($j\le i$) and $p^{ut}\equiv p^{ji}$ are the orthogonal projections of $E_{i}$ onto $E_{j}$, then the projective system $\mathcal E = (E_i, p^{ji}, T_m) $ will be addressed as \underline{a Euclidean projective system}.
\end{definition}
 The functional spaces like $L^r(T)$ and $C(T)$  give rise to natural projective systems over the line intervals $T$, see Subsection \ref{sec_ex_proj_systems} in Appendix D.

\begin{definition}\label{def adapted points}
Let $\mathcal B=(\Omega , {\mathcal F}, ({\mathcal F}_t)_{t\in {T}},P)$ be a stochastic basis and $\mathcal X=(X_t, p^{ut}, {T})$ be a projective system of separable Banach spaces.
  A random point $x\in\mathcal P(X, \mathcal S)$ ($X= X_b$) is called \underline{adapted} with respect to $\mathcal B$ and   $\mathcal X$ if $p^t(x)\equiv p^{tb}(x): \Omega\to X_t$ is $\mathcal{F}_t$-measurable for all $t\in{T}$.
  \end{definition}

\begin{definition}\label{def-space_Pa}
Let $\mathcal B$ be a stochastic basis and $\mathcal X$ be a projective system of separable Banach spaces.
The linear topological subspace of $\mathcal P(X, \mathcal S)$ consisting of all (equivalence classes of) adapted random points with respect to $\mathcal B$ and $\mathcal X$ will be denoted by $\underline{\mathcal{P}a(\mathcal X, \mathcal B).}$ If $\mathcal X$ or/and $\mathcal B$ are fixed, then the notation $\mathcal{P}a(\mathcal X, \mathcal B)$ may be simplified to $\mathcal{P}a(X, \mathcal B)$ or $\mathcal{P}a(X)$.
\end{definition}

This notation is easy to see to be consistent with the one used in Section \ref{sec-Prelim} for the case of the spaces $C[a,b]$ and $L^r[a,b]$.

\begin{example}\label{ex-trivial-stoch-basis}
  If $\mathcal F_t=\mathcal F$ for all $t\in T$, then $\mathcal Pa(\mathcal X, \mathcal B)=\mathcal P(X, \mathcal S),$ so that all statements proven for the spaces of adapted random points are automatically true for the spaces of all random points. However, the converse statements are in many cases not true. For instance, the representation theorem \ref{th-Nemytskii} for LC operators is only valid for $\mathcal P(X, \mathcal S),$ but not necessarily for $\mathcal Pa(\mathcal X, \mathcal B)$.
\end{example}

\subsection{Uniform continuity and tightness of local operators}\label{sec-uniform-continuity}

Let $X$ and $Y$ be separable Banach spaces.
The canonical uniformity \cite[p. 12-19]{Manfred} on the associated linear topological spaces of random points is understood in agreement with the topologies and the linear operations on these spaces. In particular, we have the following definition of uniform continuity of an operator $h: \mathcal A\to \mathcal P(Y, \mathcal S)$, $\mathcal A\subset \mathcal P(X, \mathcal S)$: \\
{\it For any $\varepsilon>0$, $\sigma>0$ there exist $\delta>0$, $\rho>0$ such that
$$
  P\{\|x_1-x_2\|_X\ge\rho\}<\delta \ (x_1,x_2 \in \mathcal A) \ \
  \Rightarrow \ \ P\{\|h(x_1)-h(x_2)\|_Y\ge\sigma\}<\varepsilon.
$$}
The translation invariant metric $d_X(x_1,x_2)=E\min \{\|x_1-x_2\|_X; 1\} $ gives rise to the same canonical uniformity on $P(X, \mathcal S)\times P(X, \mathcal S)$, as it is generated by translations of the same set of neighborhoods of the origin in $P(X, \mathcal S)$. This applies, of course, to the space $\mathcal P (Y, \mathcal S)$ as well.
Therefore the property of uniform continuity can be rewritten as
\\
{\it For any $\varepsilon>0$ there exists $\delta>0$ such that
$$
  d_X(x_1,x_2)<\delta \ (x_1,x_2 \in \mathcal A) \ \
  \Rightarrow \  \ d_Y(h(x_1),h(x_2))<\varepsilon.
$$}
By technical reasons it may be convenient to combine these two definitions: \\
{\it For any $\varepsilon>0$  there exist $\rho>0$, $\delta>0$ such that
$$
  P\{\|x_1-x_2\|_X\ge\rho\}<\delta \ (x_1,x_2 \in \mathcal A)
  \ \
  \Rightarrow \  \ d_Y(h(x_1),h(x_2))<\varepsilon.
  $$}
In the next definition we generalize the classical notion of a Volterra operator as the one "only depending on the past": $(\phi u)(s)=(\phi v)(s)$ ($a\le s\le t$) if $u(s)=v(s)$ ($a\le s\le t$) for any $a\le t \le b$. The operator here acts on functions defined on the line interval $T=[a,b]$. In the case of an arbitrary linearly ordered $T$ this definition can be extended in  the following manner:
\begin{definition}\label{def-Volterra}
  Let  $\mathcal L=(L_t, l^{ut}, {T})$ be a projective system of separable Fr\'{e}chet (in particular, Banach) spaces. We call an operator $\phi: L\to L$ ($L=L_b$) \underline{a generalized Volterra operator (map)} with respect to $\mathcal L$ if it generates a family of continuous operators $\phi^t: L_t\to L_t$ ($t\in {T}$) satisfying the properties $\phi=\phi^b$ and $l^{ut}\circ \phi^t=\phi^u\circ l^{ut}$ for all $t, u\in{T}, t\ge u$.
\end{definition}

\begin{remark}\label{rem-Volterra}
%Integral Volterra operators like $\int_{a}^{t}H(t,s,x(s))ds$ constitute a natural example of the operators from Definition \ref{def-Volterra} if the underlying spaces are $L_t=L^r[a,t]$, $C[a,t]$ or similar, $T=[a,b]$ and $l^{ut}$ is the restriction of functions defined on $[a,t]$ to the subinterval $[a,u]$.
The superposition operators generated by Volterra maps transform adapted random points to adapted random points. Indeed, if for $x\in \mathcal P(L, \mathcal S)$ the random point $l^{tb}(x)$ in $L_t$ is $\mathcal F_t$-measurable, then $l^{tb}(\phi x)=\phi^t(l^{tb}x)$ will be $\mathcal F_t$-measurable as well due to continuity of $\phi^t.$ This observation is important for our analysis, where the superpositions generated by finite dimensional Volterra maps are used to approximate LC operators defined on the spaces of adapted random points: it is essential that the domain and the range of the operators are invariant under approximations.
\end{remark}
This remark explains
 \begin{definition}\label{def-property-Pi}
 The projective system $\mathcal L=(L_t, l^{ut}, {T})$ of separable Fr\'{e}chet (in particular, Banach) spaces satisfies {Property $(\Pi)$} if  there exists a sequence $\pi_n: L\to L \ (L=L_b)$ of linear, continuous and finite dimensional generalized Volterra maps, which strongly converges to the identity map in $L$ as $n\to\infty.$
 \end{definition}
The property described in the definition is satisfied for most linear functional spaces used in applications, for instance, for $L^r[a,b]$ and $C[a,b]$, as it is shown in Example \ref{ex-Pi-for-C}.

The next definition is a reminder.
\begin{definition}\label{def-tight-set}
A set ${\mathcal K}\subset {\mathcal P}(X, \mathcal F)$ is called \underline{tight}
 if for any $\epsilon >0$ there exists a compact set $Q\subset X$ such that
$P\{\omega \,:\,x(\omega ) \notin Q\}<\epsilon $ whenever $x\in
{\mathcal K}$.
\end{definition}

\begin{remark}\label{rem_def_tight_sets}
  An equivalent, and sometimes more convenient, description of tightness says that $\mathcal K$ is tight if and only if for any $\sigma>0$, $\varepsilon>0$ there exists a compact set $Q\subset X$ such that
$P\{\omega \,:\,x(\omega ) \notin Q_\sigma\}<\epsilon $ whenever $x\in
{\mathcal K}$, where $Q_\sigma$ as the $\sigma$-neighborhood of the set $Q$.
\end{remark}

The theorem below is an important technical result.
\begin{theorem}
\label{th-uniform-cont}
 Suppose that $\mathcal X=(X_t, p^{ut}, {T})$ is a projective system of separable Banach spaces satisfying Property $(\Pi)$, $Y$ is another separable Banach space
 and
$h:\ {\mathcal P}a(\mathcal X, \mathcal B)\to {\mathcal P} (Y, \mathcal S)$ is a local operator.
 Then the
following statements are equivalent:
\begin{enumerate}
  \item $h: \ \mathcal Pa(\mathcal X, \mathcal B)\to \mathcal P(Y, \mathcal S)$ is uniformly continuous on each tight subset $\mathcal K\subset {\mathcal P}a(\mathcal X, \mathcal B)$;
  \item for any compact subset $Q$ of $X$ and any $\varepsilon>0$, there is $\rho>0$ such that $$\Vert u-v \Vert_X\le\rho \ \mbox{a.s. \ \ implies} \ \ \
d_Y(hu,hv)<\varepsilon$$ for all $u,v\in {\mathcal P}a(\mathcal X, \mathcal B)\cap{\mathcal P}(Q, \mathcal S)$; this is e.g. fulfilled if  $h$ is uniformly continuous on any subset ${\mathcal P}a(\mathcal X, \mathcal B)\cap{\mathcal P}(Q, \mathcal S)$, where $Q\subset X$ is an arbitrary
compact;
  \item there exists a function $O(\gamma)>0$ ($\gamma>0$), $\lim\limits_{\gamma\to +0}O(\gamma)=0$ such that for any compact $Q\subset X$ and any $\varepsilon>0$ there is $\delta>0$ satisfying the property:
\begin{equation}\label{eq-uniform-cont-1}
%\begin{array}{c}
  d_X(x,y)<\delta \ \ \ \mbox{implies} \ \ \
  d_Y(hx,hy)<\varepsilon+O(\gamma)
%\end{array}
   \end{equation}
for all $x, y \in \mathcal Pa(\mathcal X, \mathcal B), \ P\{x\notin Q\}<\gamma, \ P\{y\notin Q\}<\gamma.$
\end{enumerate}
\end{theorem}
\begin{proof}
  See Appendix B, Subsection \ref{sec-proof of uniform th}.
\end{proof}

In the case of continuous superposition operators, {Property $(\Pi)$} in Theorem \ref{eq-uniform-cont-1} can be omitted, see Example \ref{ex-uniform-caratheodory} in Appendix D.

The next definition introduced in \cite{Pon-1} generalises the notion of a compact operator.

\begin{definition}\label{def-tight-oper}
Let $X$ and $Y$ be separable Banach space and $h: \ {\mathcal A}\to {\mathcal P}(Y, \mathcal S)$, where ${\mathcal A}\subset {\mathcal P}(X, \mathcal S)$.
\begin{enumerate}
  \item The operator $h$ is called \underline{tight} if 1) it maps bounded subsets of ${\mathcal A}$ into tight subsets of ${\mathcal P}(Y,\mathcal S)$ and 2) it is uniformly continuous on any
tight subset of ${\mathcal A}$.
  \item The operator $h$ is called \underline{tight-range} if 1) it maps ${\mathcal A}$ into a tight subset of ${\mathcal P}(Y,\mathcal S)$ and 2) it is uniformly continuous on any
tight subset of ${\mathcal A}$.
\end{enumerate}
\end{definition}
This definition yields the class of (continuous) compact and compact-range operators if $\Omega $ is single-pointed. On the other hand, local operators are almost never compact. For instance, it can be proven that $h: \ \mathcal{P}(X, \mathcal S)\to \mathcal{P}(Y, \mathcal S)$ is local and compact if and only if either $P$ assumes finitely many values, or $Y$ contains finitely many points.

For nontrivial examples of tight operators see Subsection \ref{sec_ex_tight} in Appendix D.

\section{Extensions of local operators}\label{sec-extensions-operators}
Extensions of stochastic integrals are, in particular, used in the theory of weak solutions. For example, the operator $(Ju)(s)=\int_{a}^{t}u(s)dW(s)$, defined on $\mathcal{P}a(X, \mathcal B)$, where $X=C({T})$ or $X=L^r({T})$ ($1\le r < \infty$), admits a natural extension $J^*=\int_{a}^{t}u(s)dW^*(s)$ if the expansion  $\mathcal B^*$ of the stochastic basis $\mathcal B$ preserves the martingale property. Here $W(t)$ resp. $W^*(t)$ is the standard Wiener process on the stochastic basis $\mathcal B$ resp. on its expansion $\mathcal B^*$.

For general LC operators one needs to develop a martingale-independent technique.
In this section, we provide sufficient conditions for existence of an LC extension of an LC operator defined on a space of abstract adapted random points.

\subsection{Expansions of stochastic bases}\label{sec-extensions-bases}

  Expansions/changes of the underlying probability space are e.g. used if this space is not rich enough to host solutions of stochastic equations. Not all expansions preserve basic properties of stochastic integrals, and hence a fortiori we cannot hope that general LC operators can be extended to an arbitrary expansion of the original probability space. In this paper we use what we call \textit{Young expansions}, which is sufficient for our purposes.
  %In this paper, we take advantage of the technique going back to the papers \cite{Jacod-1} and \cite{Jacod-2}.

Let $\Omega^*=\Omega\times Z$, $Z$ be a Polish (e.g. separable Banach) space and $\mu$ be a measure on $\mathcal F\otimes \mbox{Bor}(Z)$, whose marginal coincides with $P$: $\mu (A\times Z)=PA$ for any $A\in \mathcal F$.

\textit{The disintegration} of the measure $\mu$ \cite[p. 19]{Crauel} is a random measure $\mu_\omega$ on $\mbox{Bor}(Z)$  for almost all $\omega\in\Omega$ such that
$$
\int_{\Omega\times Z}g(\omega, z)d\mu(\omega, z)=\int_\Omega\int_Z g(\omega, z)d\mu_\omega(z)dP(\omega)
$$
holds for every bounded measurable function $g: \ \Omega\times Z\to R.$ The corresponding differential form reads as $d\mu(\omega, z)=d\mu_\omega(z)dP(\omega)$, which can be conveniently abbreviated to $d\mu=d\mu_\omega dP$.

\textit{The narrow topology} on the set $Pr_\omega(Z)$ of all random measures on $\mbox{Bor}(Z)$ with the marginal $P$ is generated by the maps
$$
\mu\mapsto\mu(f)=E\int_Z fd\mu\equiv E\int_Zf(\omega, z)d\mu_\omega(z),
$$
where $f:\Omega \times Z\to R$ is an arbitrary bounded Carath\'{e}odory function \cite[p. 25]{Crauel}.

A neighborhood $U_{f_1,...,f_m, \delta}(\mu)$ of  $\mu\in Pr_\omega(Z)$ in the narrow topology consists of all $\nu\in Pr_\omega(Z)$ such that
$$
|E\int_Z f_id\mu-E\int_Z f_id\nu|<\delta \ \ \ (i=1,...,m).
$$
Here $f_i$ are bounded Carath\'{e}odory functions and $\delta >0.$

Definition \ref{def-ext-general} of an expansion of a stochastic basis is too general for our purposes. Therefore we introduce the notion of a Young expansion starting with probability spaces.
\begin{definition}\label{def-extension}
\underline{A Young expansion} $\mathcal S^*=(\Omega^*,{\mathcal F}^*,P^*)$
of the probability space $\mathcal S =(\Omega ,{\mathcal F}, P)$ satisfies the properties:
\begin{enumerate}
  \item $\Omega^*=\Omega\times Z$, $Z$ being a Polish (e.g. separable Fr\'{e}chet or Banach) space;
  \item $P^*$ is a probability measure on $\mathcal F\otimes \mbox{Bor}(Z)$ with the marginal $P$;
  \item the disintegration $P^*_\omega$ of $P^*$ is the limit (in the narrow topology) of a sequence of random Dirac measures $\{\delta_{\alpha_n(\omega)}\}$, where
$\alpha_n\in\mathcal P(Z, \mathcal S)$;
  \item $\mathcal F^*$ is the $P^*$-completion of the $\sigma$-algebra $\mathcal F\otimes \mbox{Bor}(Z)$.
\end{enumerate}
\end{definition}

\begin{remark}\label{rem_narrow_conv}
%\begin{enumerate}
      Property (3) in Definition \ref{def-extension} can be rewritten in terms of the measure $P^*$ and the measures $P\alpha_n^{-1}$, defined by
\begin{equation}\label{eq_measures P_alpha_-1}
  dP\alpha_n^{-1}(\omega,z)\equiv d\delta_{\alpha_n(\omega)}(z)dP(\omega),
\end{equation}
in the following way:
\begin{equation}\label{eq-narrow-convergence}
\begin{array}{c}
  E^*g\equiv\int_{\Omega^*}g(\omega, z)dP^*(\omega, z)=\lim\limits_{n\to\infty}\int_{\Omega^*}g(\omega, z)dP\alpha_n^{-1}(\omega, z) \\
  =\lim\limits_{n\to\infty}\int_{\Omega}g(\omega, \alpha_n(\omega))dP(w)=\lim\limits_{n\to\infty}E(g\circ\alpha_n)
\end{array}
  \end{equation}
for any bounded Carath\'{e}odory function $g: \Omega\times Z\to R$.
Strictly speaking, we should have written $P (\mbox{Gr}\, \alpha_n)^{-1}$ and not  $P\alpha_n^{-1}$, but we will keep the latter notation for the sake of simplicity.
\end{remark}

In the next definition we replace Polish spaces used in Definition \ref{def-extension} by separable Fr\'{e}chet spaces, because we want to construct expansions utilizing projective systems from Definition \ref{def-projective-system}.
\begin{definition}\label{def-Young-extension}
Suppose that $\mathcal Z=(Z_t, q^{ut}, {T})$ is a projective system of separable Fr\'{e}chet spaces.
\underline{A Young expansion} $\mathcal B^*=(\Omega^*,{\mathcal F}^*,({\mathcal F}_t^*)_{t\in {T}},P^*)$ of the stochastic basis $\\ \mathcal B =(\Omega ,{\mathcal F},({\mathcal F}_t)_{t\in {T}}, P)$, generated by  $\mathcal Z$, satisfies the following properties:
\begin{enumerate}
  \item $\Omega^*=\Omega\times Z$, where $Z=Z_b$;
  \item $P^*_\omega$ is the limit (in the narrow topology) of a sequence of random Dirac measures $\{\delta_{\alpha_n(\omega)}\}$, where
$\alpha_n\in\mathcal Pa(Z,\mathcal B)$;
  \item $\mathcal F^*$ is the $P^*$-completion of the $\sigma$-algebra $\mathcal F\otimes \mbox{Bor}\,(Z)$;
  \item $\mathcal F^*_t$ is the $P^*$-completion of the $\sigma$-algebra $\mathcal F_t\otimes (q^{t})^{-1}(\mbox{Bor}\,(Z_t))$ for any $t\in T.$
    \end{enumerate}
In particular, the probability space
$(\Omega^*,{\mathcal F}^*,P^*)$ is a Young expansion of the probability space $(\Omega ,{\mathcal F}, P)$.
\end{definition}
Some examples of Young expansions can be found in Subsection \ref{sec_ex_Young_ext}.

\begin{remark}\label{rem-Young-extension}
The mapping $c: \Omega^*\to\Omega$ from Definition \ref{def-ext-general} is defined as $c(\omega, z)=\omega$ in Definitions  \ref{def-extension} and \ref{def-Young-extension}. Clearly, $c$ is $(\mathcal F^*, \mathcal F)$-measurable resp. $(\mathcal F_t^*, \mathcal F_t)$-measurable for any $t\in T$.
\end{remark}

Given $S\in \mathcal F\otimes \mbox{Bor}\,(Z)$ we put $S(\omega)=\{z\in Z \ : \ (\omega, z)\in S$,
$\mbox{cl}\, S=\{(\omega, z): \ z\in \mbox{cl}\, S(\omega)\}$, where $\mbox{cl}(B)$ is the closure of a set $B\subset Z$, and $\partial S\equiv \mbox{cl}\, S\cap \mbox{cl}\, (\Omega^* - S)$ is the random boundary of the random set $S$. It can be shown (see e.g. \cite[p. 10]{Crauel}) that $\mbox{cl}\, S \in \mathcal F\otimes \mbox{Bor}\,(Z)$ if $Z$ is a Polish space. Therefore $\partial S\in \mathcal F\otimes \mbox{Bor}\,(Z)$ as well.

The next two definitions play a key role in the proof of the fixed-point theorem in Section \ref{sec-inf-dim-fixed-point}.

\begin{definition}\label{def-continuity-sets}
  Let $P^*$ be a Young probability measure defined on the $\sigma$-algebra $\mathcal F\otimes \mbox{Bor}\,(Z)$, where $Z$ is a Polish space. A set $S\in \mathcal F\otimes \mbox{Bor}\,(Z)$ is called \underline{a continuity set} of the measure $P^*$ (or simply, a $P^*$-continuity set) if $P^*(\partial B)=0$.
\end{definition}

In the next definition  we assume that $\mathcal B^*=(\Omega^*,{\mathcal F}^*,({\mathcal F}_t^*)_{t\in {T}},P^*)$ is
{a Young expansion} of the stochastic basis $\mathcal B =(\Omega ,{\mathcal F},({\mathcal F}_t)_{t\in {T}}, P)$. The corresponding expansion of the underlying probability space $\mathcal S =(\Omega ,{\mathcal F}, P)$ is denoted by $\mathcal S^*$.

 \begin{definition}\label{def-cont-adapted-points}
  Let $\mathcal X=(X_t, p^{ut}, T)$ be a projective system of separable Banach spaces and $X=X_b$.
\begin{enumerate}
  \item A random point $x\in \mathcal{P}(X, \mathcal S^*)$ is \underline{$P^*$-a.s. continuous} if there exists a subset $A\in\Omega^*$ of zero measure $P^*$ such that the maps $x(\omega, \cdot): \ Z\to X$ are continuous on  $Z(\omega)-A(\omega)$.
      %The set of all adapted  and $P^*$-a.s. continuous random points will be denoted by $\mathcal{C}a(X, \mathcal S^*)$.
  \item An adapted random point $x\in \mathcal{P}a(\mathcal X, \mathcal B^*)$ is \underline{simple} if $x=\bigcup\limits_{i=1}^s\alpha_iI_{A_i}$ for some $\alpha_i\in \mathcal Pa (\mathcal X, \mathcal B)$ and disjunct $P^*$-continuity subsets $A_i\in \mathcal F^*$ ($i=1,..,s$), $\bigcup\limits_{i=1}^sA_i=\Omega^*$.
      \item The set of all simple points will be denoted by $\mathcal{S}a(\mathcal X, \mathcal B^*)$.
\end{enumerate}
 \end{definition}
Clearly, the random points from $\mathcal{S}a(\mathcal X, \mathcal B^*)$ are $P^*$-continuous.

\begin{remark}\label{rem-narrow-topology}
    Note that conditions on the random function $g$ in Remark \ref{rem_narrow_conv} can be relaxed \cite{Crauel}: the equality (\ref{eq-narrow-convergence}) holds, in fact, for any $P^*$-a.s. continuous $g:\Omega\times Z\to R$. In particular, $$P^*A=\lim\limits_{n\to\infty}(P\alpha_n^{-1})(A)=P\{\alpha_n\in A\}$$
  for any $P^*$-continuity subset $A\in \mathcal F^*$.
\end{remark}

The approximation result below is used in the proof of the main results.

\begin{theorem}\label{th-approx-simple-points} Assume that $\mathcal X=(X_t, p^{ut}, {T})$ and $\mathcal Z=(Z_t, q^{ut}, {T})$ are two projective systems of separable Banach and Fr\'{e}chet spaces, respectively, both satisfying Property $(\Pi)$, and $\mathcal B^*=(\Omega^*,{\mathcal F}^*,({\mathcal F}_t^*)_{t\in {T}},P^*)$ is a
{Young expansion} of the stochastic basis $\mathcal B \! =\! (\Omega ,{\mathcal F},({\mathcal F}_t)_{t\in {T}}, P)$ generated by $\mathcal Z$.
Then for any $x, y \in \mathcal{P}a(\mathcal X, \mathcal B^*)$, satisfying $x|_A=y|_A$ $P^*$-a.s  for some $A\in \mathcal F^*$,  there exist $x_n, y_n\in \mathcal{S}a(\mathcal X, \mathcal B^*)$ and $P^*$-continuity subsets $A_n\in \mathcal F^*$, for which

(1) $x_n|_{A_n}=y_n|_{A_n}$ $P^*$-a.s.,

(2) $x_n\to x$, $y_n\to y$ in probability $P^*$ and

(3) $\lim\limits_{n\to\infty}P^*(A_n\triangle A)=0$.
\end{theorem}
\begin{proof}
  See Appendix B, Subsection \ref{sec-appr-simple-points}.
\end{proof}

\subsection{Construction of LC extensions}\label{sec-LC-extensions}

In this subsection we assume that
$\mathcal X$ is a projective system of separable Banach spaces, $Y$ is another separable Banach space, $\mathcal B$ is a stochastic basis and $\mathcal B^*$ is its Young expansion generated by a projective system of separable Fr\'{e}chet spaces $\mathcal Z=(Z_t, q^{ut}, {T})$, see Definition \ref{def-Young-extension}. Recall that in this case
$c: \Omega^*\equiv \Omega\times Z\to\Omega$ is the projection on the first factor. To simplify the notation, we put ${\mathcal P}(Y, \mathcal S)\equiv \mathcal P(Y)$, $\mathcal P^*(Y)\equiv P^*(Y, \mathcal S^*)$,
$\mathcal Pa (X) \equiv\mathcal{P}a(\mathcal X, \mathcal B)$, $\mathcal Pa^* (X) \equiv\mathcal{P}a(\mathcal X, \mathcal B^*)$ and $\mathcal Pa (Z) \equiv\mathcal{P}a(\mathcal Z, \mathcal B).$

Note that the linear homeomorphism $x\mapsto x\circ c$ naturally identifies the linear topological spaces $\mathcal Pa (X)$ and $\mathcal P(Y)$ with the respective linear topological subspaces of $\mathcal Pa^*(X)$ and $\mathcal P^*(Y)$. This justifies
\begin{definition}\label{def extension operators}
Let   $h: \mathcal Pa (X)\to \mathcal P(Y)$ be an LC operator. We say that an LC operator $h^*:\mathcal Pa^* (X)\to \mathcal P^*(Y)$ is \underline{an LC extension} of the operator $h$ if the restriction of $h^*$ to $\mathcal Pa (X)$ coincides with $h$.
\end{definition}
In the case of Young expansions generated by  Dirac measures, the extension of local operators can be constructed explicitly, as it is shown in the following remark.
\begin{remark}\label{ex-extention-Dirac-measure}
  Let $\mathcal B^*=(\Omega^*,{\mathcal F}^*,{\mathcal F}_t^*,P^*)$ be the Young expansion of the stochastic basis $\mathcal B =(\Omega ,{\mathcal F},{\mathcal F}_t, P)$ where the Young measure $P^*$ is generated by a random Dirac measure  $P^*=P\alpha^{-1}$ for some $\alpha\in \mathcal Pa (Z) $, i.e. $P^*(A)=P\{\omega\in\Omega: \ \alpha(\omega)\in A(\omega)\}$. In this case, the measure preserving map $\omega\mapsto (\omega,\alpha(\omega))$ gives rise to
the linear topological isomorphism  $\alpha_Y: \ y\mapsto y\circ\alpha$ between the spaces $\mathcal P (Y)$ and $\mathcal P^* (Y)$. Evidently, the inverse map is then given by  $\alpha^{-1}_Y: \tilde y \mapsto \tilde y\circ c$.
  Moreover, according to Example \ref{ex-Young-stoch-basis} the map $\alpha_X: \ x\mapsto x\circ\alpha$ is a linear topological isomorphism between the spaces $\mathcal Pa (X)$ and $\mathcal Pa^* (X)$. Let us, therefore, put $$h^*x=(h(x\circ \alpha)\circ c)=\alpha_Y^{-1}h\alpha_X.$$ Then $h^*:\ \mathcal Pa^*(X)\to \mathcal P^*(Y)$ is continuous, and we claim that $h^*$ is a local extension of $h$.
  Indeed, $h^*(x\circ c)=h(x\circ c\circ \alpha)\circ c=hx\circ c$, because $(c\circ\alpha)(\omega)=\omega$ for any $\omega\in\Omega,$ so that $h^*$ is an extension of $h$. To prove that $h^*$ is local, take $x,y\in\mathcal Pa^*(X)$ and $A\in\mathcal F\otimes \mbox{Bor}(Z)$ such that $x|_A=y|_A$ (it is sufficient to prove locality for such $A$). Then $B=\alpha^{-1}(A)\in \mathcal F$ and
  $$
  \begin{array}{c}
    x\circ\alpha|_B=y\circ\alpha|_B \ \ \Rightarrow \ \ h(x\circ\alpha)|_B=h(y\circ\alpha)|_B \ \ P-\mbox{a.s.}
    \\ \Rightarrow \ \ h(x\circ\alpha)\circ c|_{c^{-1}B}=h(y\circ\alpha)\circ c|_{c^{-1}B} \ \ P^*-\mbox{a.s.}
        \ \ \Rightarrow \ \ h^*x|_{c^{-1}B}=h^*y|_{c^{-1}B} \ \ P^*-\mbox{a.s.}
  \end{array}
    $$
But $$P^*(A\triangle c^{-1}B)=P\{\alpha\in A\triangle c^{-1}B\}=
P\left(\{\alpha\in A\}\triangle \{\alpha\in c^{-1}B\}\right)=P(B\triangle B)=0,
$$
so that $h^*x|_A=h^*y|_A \ P^*-\mbox{a.s.}$

In particular, this example shows that if $h$ is uniformly continuous on tight subsets of the space $\mathcal Pa(X)$ and the Young expansion $\mathcal B^*$ of $\mathcal B$ is generated by a random Dirac measure, i.e.  $P^*=P\alpha^{-1}$ ($\alpha\in \mathcal Pa (Z)$), then the extension $h^*$ of $h$ is uniformly continuous on tight subsets of the space $\mathcal Pa^*(X)$ as well. This follows from the fact that the linear topological isomorphisms $\alpha_X: \mathcal Pa (X)\to\mathcal Pa^* (X)$ and $\alpha_Y: \mathcal P(Y)\to \mathcal P^*(Y)$ preserve tight sets, because the map $\omega\mapsto (\omega,\alpha(\omega))$ is measure preserving.

The case of general Young expansions is considered in Theorem \ref{th-extension-operators-general}.
\end{remark}

The uniqueness property of LC extensions can be easily proven in a rather general setting.
\begin{theorem}\label{prop-unique-local-extension}
Let $\mathcal B^*=(\Omega^*,{\mathcal F}^*,({\mathcal F}_t^*)_{t\in {T}},P^*)$ be
{a Young expansion} of the stochastic basis $\mathcal B \! =\! (\Omega ,{\mathcal F},({\mathcal F}_t)_{t\in {T}}, P)$ and $\mathcal X=(X_t, p^{ut}, {T})$ be a projective system of separable Banach spaces satisfying Property $(\Pi)$.
  If an LC operator $h: \mathcal Pa (X)\to {\mathcal P}(Y)$ admits a continuous extension $h^*: \mathcal Pa^* (X)\to {\mathcal P}^*(Y)$, then this extension is unique.
\end{theorem}
\begin{proof}
  If $x\in \mathcal{S}a(\mathcal X, \mathcal B^*)$, then there exist $\alpha_i\in \mathcal Pa (X)$ and disjunct subsets $A_i\in \mathcal F^*$ ($i=1,..,s$), $\bigcup\limits_{i=1}^sA_i=\Omega^*$, such that $x=\sum\limits_{i=1}^s\alpha_iI_{A_i}$.
  The property of locality implies that any two extensions $h^*_1$ and $h^*_2$ of the operator $h$ must satisfy
  $$
  h^*_1x=h^*_2x=\sum\limits_{i=1}^sh(\alpha_i)I_{A_i} \ \ P^*-\mbox{a.s.}
  $$
By Theorem \ref{th-approx-simple-points}, the set $\mathcal{S}a(\mathcal X, \mathcal B^*)$ is dense in $\mathcal Pa^* (X)$. Therefore $h^*_1x=h^*_2x$ for all $x\in\mathcal Pa^* (X)$, as both
are continuous in the topology of this space.
\end{proof}

The next result generalizes the one considered in Remark \ref{ex-extention-Dirac-measure}.

\begin{theorem}\label{th-extension-operators-general}
  Let $\mathcal X=(X_t, p^{ut}, {T})$ be a projective system of separable Banach spaces  and $\mathcal B^*=(\Omega^*,{\mathcal F}^*,({\mathcal F}_t^*)_{t\in {T}},P^*)$ be
{a Young expansion} of the stochastic basis $\mathcal B \! =\! (\Omega ,{\mathcal F},({\mathcal F}_t)_{t\in {T}}, P)$ generated by a projective system of separable Fr\'{e}chet spaces $\mathcal Z=(Z_t, p^{ut}, {T})$. Assume that $\mathcal X$ and $\mathcal Z$ satisfy Property $(\Pi)$.
Then any
local operator $h: \mathcal Pa (X)\to {\mathcal P}(Y)$,  which is uniformly continuous on tight subsets, admits an LC extension $h^*: \mathcal Pa^* (X)\to {\mathcal P}^*(Y)$, which is also uniformly continuous on tight subsets.
\end{theorem}

\begin{proof}
  See Appendix B, Subsection \ref{sec-extension-LC}.
\end{proof}

\section{Main results}\label{sec-inf-dim-fixed-point}

In this section we justify the general infinite dimensional fixed-point theorem formulated in \cite{Pon-1} without a proof. The first step in this direction will be a finite dimensional fixed-point theorem for LC operators.

It is still assumed that
$\mathcal B$ is a stochastic basis on a complete probability space $\mathcal S$.

\begin{theorem}\label{th-finite-dim}
  Let $\mathcal X=(X_t, p^{ut}, {T})$ be a projective system of finite dimensional spaces.
  If $U$ is a closed, convex, bounded and nonempty subset of $X$, $\mathcal Pa(U)\equiv \mathcal Pa(\mathcal X, \mathcal B)\cap\mathcal P(U, \mathcal S)$ and $h: \ {\mathcal P}a(U)\to {\mathcal P}a(U)$ is an LC operator, then $h$ has at least one fixed point.
\end{theorem}
\begin{proof}
  See Appendix B, Subsection \ref{sec-proof-finite-dim}.
\end{proof}
 In the rest of the section we assume $\mathcal X=(X_t, p^{ut}, T)$ to be a projective system of arbitrary separable Banach spaces.
The following definition of a weak fixed point generalizes the one briefly described in Section \ref{sec-Prelim}:

\begin{definition}\label{def weak sol abstract}
Let $h: \mathcal Pa (X, \mathcal B)\to {\mathcal P}(X, \mathcal S)$ be an LC operator. If there exists an expansion
$\mathcal B^*$ of the stochastic basis $\mathcal B$, an LC extension $h^*: \ \mathcal{P}a(\mathcal X, \mathcal B^*)\to {\mathcal P}(X, \mathcal S^*)$ of the operator $h$ and a random point $x^*\in \mathcal Pa (X, \mathcal B^*)$ such that $h^*x^*=x^*$ $P^*$-a.s., then $x^*$ is called \underline{a weak fixed point} of the operator $h$.
\end{definition}

Note that $h^*$ in Definition \ref{def weak sol abstract} does exist if $\mathcal B^*$ is a Young expansion and $h$ is local and uniformly continuous on every tight subset of its domain, see Theorem \ref{th-extension-operators-general}.

\begin{remark}\label{rem-weak-sol}
  The notion of a weak solution is well-known in stochastic analysis.
  It is also a well-established practice to call solutions strong if they are defined on the original probability space $\mathcal S$. Following this terminology we call any fixed point of the operator $h$ belonging to the space  $\mathcal Pa (X, \mathcal B)$  \underline{a strong fixed point}.
 \end{remark}

Now we are able to formulate the main result of the paper.
\begin{theorem}\label{th-fixed-point-general}
  Let the projective system of separable Banach spaces $\mathcal X=(X_t, p^{ut}, T)$ satisfy Property (${\Pi}$) and
  $h: \mathcal{P}a (\mathcal X, \mathcal B)\to \mathcal{P}a (\mathcal X, \mathcal B)$ be a local  operator which is uniformly continuous on tight subsets of its domain.
  \begin{enumerate}
    \item If for some convex, closed and nonempty set $V\subset X$ the operator $h$ maps $\mathcal{P}a (\mathcal X, \mathcal B)\cap \mathcal P(V, \mathcal S)$ into its tight subset, then $h$ has at least one weak fixed point $x^*\in \mathcal{P}a(\mathcal X, \mathcal B^*)\cap \mathcal P(V, \mathcal S^*)$ for some Young expansion $\mathcal B^*$ of $\mathcal B$ defined on the probability space $\mathcal S^*$.
    \item If for any Young expansion $\mathcal B^*$ of $\mathcal B$ the associated (unique) LC extension $h^*$ of the operator $h$ has at most one fixed point in $\mathcal{P}a(\mathcal X, \mathcal B^*)$, then each weak fixed point of the operator $h$ will be equivalent to a unique strong, i.e. belonging to the space $\mathcal{P}a (\mathcal X, \mathcal B)$, solution of the equation $hx=x$.

  \end{enumerate}
     \end{theorem}
\begin{proof}
We will use the simplified notation for the spaces of random points. Given $U\subset X$ we put
 $$\mathcal Pa (U) \equiv\mathcal{P}a(\mathcal X, \mathcal B)\cap \mathcal P (U, \mathcal S) \ \ \ \mbox{and} \ \ \ \mathcal Pa^*(U) \equiv\mathcal{P}a(\mathcal X, \mathcal B^*)\cap \mathcal P (U, \mathcal S^*).$$
In particular,
$$\mathcal Pa (X)\equiv \mathcal{P}a(\mathcal X, \mathcal B) \ \ \ \mbox{and} \ \ \ \mathcal Pa^*(X)\equiv \mathcal{P}a(\mathcal X, \mathcal B^*).$$

The Young expansion in the first part of the proof will be generated by the projective family $\mathcal Z=(Z_t, q^{ut}, T)$ coinciding with $\mathcal X$:
\begin{equation}\label{eq_proj familiy_main_th}
  \mathcal Z=\mathcal X=(X_t, p^{ut}, T).
\end{equation}
In accordance with the notational agreement from Remark \ref{rem-projective-systems} we write $Z\equiv Z_b$, and $X\equiv X_b$, so that $X=Z$. The Young measure $P^*$ on $\Omega^*=\Omega\times Z=\Omega\times X$ will be constructed in the course of the proof.

\textit{Existence of a weak fixed point.}

Let $h(\mathcal Pa(V))\subset \mathcal K$ for some tight subset $\mathcal K\subset \mathcal Pa(V)$. For any $n\in N$ there exists a compact subset $Q^n\subset V$ such that $P\{x\notin Q^n\}<1/n$ for all $x\in \mathcal K$. As $V$ is convex we may assume that $Q^n$ is convex, too.

Consider the sequence $\pi_n: X\to X$ of finite dimensional linear Volterra maps converging strongly to the identity map, which exists due to Property (${\Pi}$). From Remark \ref{rem-Volterra} we know that the continuous superposition operators $h_{\pi_n}$ generated by $\pi_n$  map the space $\mathcal Pa(X)$ into itself. Moreover, the strong convergence of the sequence $\{\pi_n\}$ is uniform on compacts, and it is therefore easy to check that $h_{\pi_n}(x)=\pi_n(x)\to x$ uniformly on $\mathcal K$. That is why we may assume, without loss of generality, that $P\{\|\pi_n(x)-x\|\ge 1/n\}<1/n$ for all $x\in \mathcal K.$

For any $n\in N$ let us define the finite dimensional projective system by
\begin{equation}\label{eq-finite-dim-system}
  \mathcal X^n=(X^n_t, p^{ut}_n, T), \ \  \ \mbox{where} \ \ \ X^n=\pi_n(X), \ X^n_t=p^{tb}(X^n), \ p^{ut}_n=p^{ut}|_{X^n_t}.
\end{equation}
If $t\ge u \ge v$ and $x=p^{tb}y\in X^n_t$ (for some $y\in X^n$), then
$$(p^{vu}_n\circ p^{ut}_n)(x)=(p^{vu}|_{X^n_u}\circ p^{ut}|_{X^n_t})(x)=(p^{vu}\circ p^{ut})(x)=p^{vt}x=(p^{vt}|_{X^n_t})(x)=p^{vt}_nx,$$
as $p^{ut}x= (p^{ut} \circ p^{tb})(y)=p^{ub}y\in X_u^n.$ Therefore, $\mathcal X^n$ is a projective system.
Evidently, $\mathcal Pa(X^n)\subset \mathcal Pa(X)$ and  $h_{\pi_n}(\mathcal Pa(X)\subset \mathcal Pa(X^n)$.

By Lemma \ref{lem-projection} there are continuous projections $\phi_n: \ X^n\to X^n\cap Q^n$ such that the associated superposition operators satisfy $h_{\phi_n} (\mathcal Pa(X^n))\subset \mathcal Pa(X^n\cap Q^n)$. By construction,
\begin{equation}\label{eq-fixed-point-general-1}
P\{\|(\phi_n\circ\pi_n)(x)-x\|\ge 1/n\}<2/n \ \ \ \mbox{for all} \ \ \ x\in \mathcal K.
\end{equation}
For the LC operators $h_n\equiv h_{\phi_n}\circ h_{\pi_n} \circ h$ we have
$h_n(\mathcal Pa(X^n\cap Q^n))\subset \mathcal Pa(X^n\cap Q^n)$. Hence by the finite dimensional fixed-point theorem \ref{th-finite-dim},
\begin{equation}\label{eq_alpha_n}
\exists\alpha_n\in \mathcal Pa(X^n\cap Q^n) \ \ \ \mbox{
such that} \ \ \  h_n\alpha_n=\alpha_n \ P-\mbox{a.s.}
\end{equation}
As $\alpha_n\in \mathcal Pa(Q^n)\subset \mathcal Pa(V)$, we have $h\alpha_n\in \mathcal K$ for all $n\in N$. From (\ref{eq-fixed-point-general-1}) it then follows that
$$
\begin{array}{c}
P\{\|h\alpha_n-\alpha_n\|\ge 1/n\}=P\{\|h\alpha_n-h_n\alpha_n\|\ge 1/n\} \\
=P\{\|h\alpha_n-(h_{\phi_n}\circ h_{\pi_n} \circ h)\alpha_n\|\ge 1/n\}<2/n \ \ \ \mbox{for all} \ \ \ n\in N.
\end{array}
$$
Hence
\begin{equation}\label{eq-fixed-point-general-1a}
  \|h\alpha_n-\alpha_n\|_X\to 0 \ \ \ \mbox{in probability} \ P \ \mbox{as} \ n\to\infty.
\end{equation}
Moreover, the set $\{h\alpha_n | \ n\in N\}$ is tight, so that the set $\{\alpha_n: \ n\in N\}$ satisfies the assumptions described in Remark \ref{rem_def_tight_sets}, according to which the latter set is tight, too. Thus, the sequence of the random Dirac measures $\{\delta_{\alpha_n(\omega)}\}$ is precompact in the narrow topology of the space $Pr_\Omega(Z)$, and we may assume, without loss of generality, that  $\{\delta_{\alpha_n(\omega)}\}$ converges to a random probability measure $\mu_\omega$ on this $\sigma$-algebra. Using $Z=X$   let us define $P^*$ by
\begin{equation}\label{eq_P_*}
  dP^*(\omega, x)=d\mu_\omega(x)dP(\omega).
\end{equation}
In particular, we obtain
\begin{equation}\label{eq_conv_main_theorem}
  E^*g=\lim\limits_{n\to\infty}E(g\circ \alpha_n)\equiv \lim\limits_{n\to\infty}\int\limits_{\Omega^*}gdP\alpha_n^{-1},
\end{equation}
which holds for any bounded, $P^*$-a.s. continuous function $g$, see  Remark \ref{rem-narrow-topology}. Here
%Remember that these metrics define convergence in probability with respect to the corresponding measures (see e.g. Section \ref{sec-Prelim}).
$E$ resp. $E^*$ is the expectation with respect to the probability measure $P$ resp. $P^*$ and the measure $P\alpha_n^{-1}$ is defined by $
  P\alpha^{-1}(A\times B)=\{\omega\in A : \ \alpha(\omega)\in B\}
  $, see  Remark \ref{rem_narrow_conv}.

By Theorem \ref{th-extension-operators-general}, the operator $h$ admits a unique LC extension $h^*: \ \mathcal Pa^*(X)\to \mathcal Pa^*(X)$. We claim that the random point
\begin{equation}\label{eq_x_*}
x^*: \ \Omega^*\to X \ \ \mbox{defined as} \ \ x^*(\omega, x)=x,\ \  \mbox{where} \ \ x\in X=Z,
\end{equation}
is a fixed point of the operator $h^*$, i.e. $h^*x^*=x^*$ $P^*$-a.s.

First of all, let us check that $x^*\in \mathcal Pa^*(V)$ $P^*$-a.s. Indeed, for any $t\in T$ and any $B\in \mbox{Bor}\, (X_t)$, the set
 $$
%\begin{array}{l}
  \{(\omega, x)\in \Omega\times X: \ (p^tx^*)(\omega, x)\in B\}
=\{(\omega, x)\in \Omega\times X: \ p^t(x)\in B\} \\
=\Omega\times (p^t)^{-1}(B)
%\end{array}
$$
belongs to the $\sigma$-algebra $\mathcal F_t^*$, see Definition \ref{def-Young-extension}. Thus, $x^*\in \mathcal Pa^*(X)$. To see that $x^*$ takes values in $V$ $P^*$-a.s. we observe that by construction
$\alpha_n\in V$ a.s., so that $P\alpha^{-1}_n (\Omega\times V)=1$ for all $n\in N$. Therefore, by the Portmenteau theorem \cite[p.26]{Crauel}
$
P^*(\Omega\times V)\ge 1,
$
as $V$ is closed in $X$, which means that $P^*(\Omega\times (X-V))=0$ and $P^*\{x^*\notin V\}=0$.

Below we use the metrics
$$
%\begin{array}{c}
d_X(x,y)=E(\min\{\|x-y\|_X\};1), \ \ \ \ d_X^*(x,y)=E^*(\min\{\|x-y\|_X\};1)
%d_Y(u,v)=E(\min\{\|u-v\|_Y\};1), \ \  \ \ d_Y^*(u,v)=E^*(\min\{\|u-v\|_Y\};1)
%\end{array}
$$
on the spaces $\mathcal Pa(X)$ and $\mathcal Pa^*(X)$, respectively.

Let $H=h-id$, where $id$ is the identity map on $Pa(X)$. From (\ref{eq-fixed-point-general-2}) we have
\begin{equation}\label{eq-fixed-point-general-2}
  \|H\alpha_n\|_X\to 0 \ \ \ \mbox{in probability} \ P \ \  \mbox{as} \ n\to\infty.
\end{equation}
Evidently, $H^*\equiv h^*-id^*: \ \mathcal Pa^*(X)\to Pa^*(X)$ is the LC extension of the LC operator $H$, where $id^*$ is the identity map on  $Pa^*(X)$. We shall prove that $H^*x^*=0$ $P^*$-a.s.

The operator $H=h-id$ is uniformly continuous on  tight subsets of the space $\mathcal Pa(X)$. By Theorem \ref{th-uniform-cont}, there exists a function $O(\gamma)>0$ ($\gamma>0$), $\lim\limits_{\gamma\to +0}O(\gamma)=0$, such that for any compact $Q\subset X$ and any $\varepsilon>0$ there is $\delta>0$ satisfying the property:
\begin{equation}\label{eq-fixed-point-general-3}
%\begin{array}{c}
  d_X(x,y)<\delta \ \ \ \mbox{implies} \ \ \
  d_X(Hx,Hy)<\varepsilon+O(\gamma)
%\end{array}
   \end{equation}
for all $x, y \in \mathcal Pa(X), \ P\{x\notin Q\}<\gamma, \ P\{y\notin Q\}<\gamma.$

By Theorem \ref{th-approx-simple-points}, there exists a sequence $\{x_\nu\}\subset \mathcal{S}a(\mathcal X, \mathcal B^*)$, for which $d^*_X(x_\nu,x^*)\to 0$ as $\nu\to\infty$, so that
\begin{equation}\label{eq-fixed-point-general-4}
  d^*_X(x_\nu, x^*)<\delta \ \ \ \mbox{for all} \ \ \ \nu\ge \nu_1.
\end{equation}
Pick an arbitrary $\varepsilon>0$, $\gamma>0$ and find $\delta$ satisfying (\ref{eq-fixed-point-general-3}). Due to tightness of the sets $\{\alpha_n: \ n\in N\}\subset \mathcal Pa(X)$ and  $\{x_\nu: \ \nu\in N\}\subset \mathcal Pa^*(X)$ there is a convex compact $K_\gamma\subset X$ such that
\begin{equation}\label{eq-fixed-point-general-3a}
%\begin{array}{c}
  P\{\alpha_n\notin K_\gamma\}<\gamma \ \  \mbox{for any} \ \ n\in N
 \end{equation}
    {and}
   \begin{equation}\label{eq-fixed-point-general-3b}
  P^*\{x_\nu\notin K_\gamma \}<\gamma \ \  \mbox{for any} \ \ \nu\in N.
%\end{array}
\end{equation}

By continuity of the operator $H^*$, the sequence $\{H^*x_\nu\}$ converges to $H^*x^*$ in the space $\mathcal Pa^*(X)$. Hence
there exists $m_0\in N$ for which
\begin{equation}\label{eq-fixed-point-general-5}
  d^*_X(H^*x^*, H^*x_\nu)<\varepsilon \ \ \ \mbox{for all} \ \ \ \nu\ge \nu_2.
\end{equation}
Let $m=\max\{\nu_1; \nu_2\}$ and put $y=x_m$.
The set  $\{y\notin K_\gamma \}$ is a $P^*$-continuity set, as $y$ is $P^*$-a.s. continuous. Therefore, by (\ref{eq-fixed-point-general-3b}) and Remark \ref{rem-narrow-topology}
\begin{equation}\label{eq-fixed-point-general-6}
  P\{y\circ \alpha_n\notin K_\gamma\}=P\alpha_n^{-1}\{x_m\notin K_\gamma \}<\gamma  \ \ \ \mbox{for all} \ \ \ n\ge n_1.
\end{equation}
Let $E^n$ be the expectation with respect to the measure $P\alpha_n^{-1}$. As the function $\min\{\|y-x^*\|_X; 1\}$ is $P^*$-a.s. continuous, we get
\begin{equation}\label{eq-fixed-point-general-7}
\begin{array}{l}
  d_X(y\circ \alpha_n,\alpha_n)= E\min\{\|y\circ \alpha_n-\alpha_n\|_X; 1\} \\
 =E\min\{\|y\circ \alpha_n-x^*\circ\alpha_n\|_X; 1\} =E^n \min\{\|y-x^*\|_X; 1\}<\delta \ \ \ \mbox{for all} \ \ \ n\ge n_2
\end{array}
\end{equation}
 by (\ref{eq-fixed-point-general-4}), where $x_m=y$.
Combining  (\ref{eq-fixed-point-general-3a}), (\ref{eq-fixed-point-general-6}) and (\ref{eq-fixed-point-general-7}) and applying the estimate (\ref{eq-fixed-point-general-3}) yield
$$
  d_X(H(y\circ \alpha_n), H\alpha_n)<\varepsilon+O(\gamma) \ \ \ \mbox{for all} \ \ \ n\ge \max\{n_1; n_2\}.
$$
Minding (\ref{eq-fixed-point-general-2}) we find $n_3\in N$ such that $d_X(H\alpha_n,0)<\varepsilon$ for all $n\ge n_3$. Therefore
\begin{equation}\label{eq-fixed-point-general-8}
  d_X(H(y\circ \alpha_n), 0)<2\varepsilon+O(\gamma) \ \ \ \mbox{for all} \ \ \ n\ge \max\{n_1; n_2; n_3\}.
\end{equation}
By construction, $y=x_m\in \mathcal{S}a(\mathcal X, \mathcal B^*)$. Therefore, $y$ can be represented as
 $y=\sum\limits_{i=1}^sc_iI_{A_i}$ for some disjunct $P^*$-continuity sets $A_i\in \mathcal F\otimes \mbox{Bor}\, (X)$, $\bigcup\limits_{i=1}^sA_i=\Omega^*$, $c_i\in \mathcal Pa(X)$ ($i=1,...,s$), see Definition \ref{def-cont-adapted-points}. Then, by the representation
(\ref{eq-def-h-0}), we obtain
  $H^*y=\sum\limits_{i=1}^sh(c_i)I_{A_i}$ $P^*$-a.s. On the other hand,
  $$
  H(y\circ \alpha_n)=h\left(\sum\limits_{i=1}^sc_iI_{\{\alpha_n\in A_i\}}\right)=\sum\limits_{i=1}^sHc_iI_{\{\alpha_n\in A_i\}} \ \ P-\mbox{a.s.}
  $$
  by the property of locality of the operator $H$.
  As the random point $H^*y$ is $P^*$-a.s. continuous,  we obtain from (\ref{eq-fixed-point-general-8}) that
  $$
  \begin{array}{l}
  d^*_X(H^*y, 0)=E^* \min\{\|\sum \limits_{i=1}^sh(c_i)I_{A_i}\|_Y;1\}\\
  =\lim\limits_{n\to\infty}E^n \min\{\|\sum\limits_{i=1}^sh(c_i)I_{A_i}\|_Y;1\}
  =\lim\limits_{n\to\infty}E\min\{\|H(y\circ \alpha_n)\|_Y;1\}
    \le 2\varepsilon+O(\gamma)
  \end{array}
    $$
  and minding (\ref{eq-fixed-point-general-4}), where $x_\nu=y$, we arrive at the estimate
  $d^*_X(H^*x^*, 0)< 3\varepsilon+O(\gamma)$. As $\varepsilon>0 $ and $\gamma>0$ were arbitrary and $\lim\limits_{\gamma\to +0}O(\gamma)=0$, we see that  $d^*_X(H^*x^*, 0)=0$, so that $H^*x^*=0$ and $h^*x^*=x^*$ $P^*$-a.s. This completes the proof of the first part of the theorem.

\textit{Existence of a strong fixed point.}

 Let $x^*$ be the only weak fixed point of the operator $h$ defined on a Young expansion $\mathcal B^*=(\Omega^*, \mathcal F^*, \mathcal F^*_t, P^*)$, where $\Omega^*=\Omega\times Z$ for some separable Fr\'{e}chet space $Z$.

 Consider
 two copies of the Young expansion $\mathcal B^{*}$:
 $$
 \mathcal B^{i}=(\Omega^{i}, \mathcal F^{i}, \mathcal F^{i}_t, P^{i})=(\Omega^{*}, \mathcal F^{*}, \mathcal F^{*}_t, P^{*})=\mathcal B^{*} \ \ \ (i=1,2),
 $$
 so that, in particular, $Z=Z_1=Z_2$, as well as their product
 $$
 \mathcal B^{**}=(\Omega^{**}, \mathcal F^{**}, \mathcal F^{**}_t, P^{**}),
 $$
 defined by
  %\begin{array}{c}
 \begin{enumerate}
   \item $\Omega^{**}=\Omega\times Z\times Z$;
   \item $dP^{**}(\omega,z_1, z_2)=dP^{**}_\omega(z_1, z_2)dP(\omega),$ {where}
   $P^{**}_\omega(z_1, z_2)=P^*_\omega(z_1)\otimes P^*_\omega(z_2)$, $z_i\in Z$ ($i=1,2$);
   \item $\mathcal F^{**}$ is the $P^{**}$-completion of the $\sigma$-algebra $\mathcal F\otimes \mbox{Bor}\,(Z)\otimes \mbox{Bor}\,(Z)$
   \item $\mathcal F^{**}_t$ is the $P^{**}$-completion of the $\sigma$-algebra $\mathcal F_t\otimes (q^{t})^{-1}(\mbox{Bor}\,(Z_t))\otimes (q^{t})^{-1}(\mbox{Bor}\,(Z_t))$ for any $t\in T,$ where $Z_t=q^t(Z),$ $q^t=q^{tb},$
 \end{enumerate}
  Denote
  $$
  \mathcal Pa^i (X)\equiv\mathcal Pa(\mathcal X, \mathcal B^i)=\mathcal Pa^*(X)
  \ \ \
  \mbox{and} \ \ \
 \mathcal Pa^{**}(X)\equiv \mathcal P(\mathcal X, \mathcal B^{**})
 $$
 and let $h^i$ and $h^{**}$ be the LC extensions of the operator $h$ to the spaces $\mathcal Pa^i (X)$ and $\mathcal Pa^{**}(X)$, respectively. Due to the uniqueness of LC extensions (Theorem \ref{prop-unique-local-extension}), $h^{**}$ is an LC extension of each of the LC operators $h^i: \mathcal Pa^i (X)\to \mathcal Pa^i (X)$ to the space $\mathcal Pa^{**}(X)$.

 By construction, $h^ix_i=x_i$ $P^i$-a.s.  ($i=1,2$), where $x_i$ is a copy of $x^*$ if $Z$ is replaced by $Z^i$. Put $x^{**}_i(\omega, z_1, z_2)=x_i(\omega, z_i)$ and observe that $x^{**}_i\in \mathcal Pa^i(X)$. Therefore,
 $$
 h^{**}x^{**}_i=h^ix_i=x_i=x^{**}_i \ \ \ P^{**}-\mbox{a.s.} \ \ \ (i=1,2).
 $$
 By uniqueness, $x^{**}_1=x^{**}_2$ $P^{**}$-a.s., so that $$(P^*_\omega\otimes P^*_\omega)\{(z_1, z_2): x^*(\omega, z_1)=x^*(\omega, z_2)\}=1 \ \ \ P-\mbox{a.s.}$$
  Hence,
  there must exist $\alpha: \Omega\to X$ such that $P^*_\omega=\delta_{\alpha(\omega)}$ $P$-a.s.
    To verify that $\alpha\in\mathcal Pa(X)$ we note that $P^*_\omega$ is a Young measure, so that it is the limit (in the narrow topology) of the sequence of the Dirac measures $\{\delta_{\alpha_n(\omega)}\}$ for some $\alpha_n\in\mathcal Pa(X)$. Then
  $\alpha=\lim\limits_{n\to\infty}\alpha_n$ in the topology of the space $\mathcal Pa(X)$, which proves that $\alpha\in\mathcal Pa(X)$. Finally,
  $$
  \|H\alpha\|_X=\lim\limits_{n\to\infty}\|H\alpha_n\|_X=\lim\limits_{n\to\infty}\|h\alpha_n-\alpha_n\|_X=0
  $$
  in probability $P$ due to continuity of $H=h-id$. Thus, $h\alpha=\alpha$ $P$-a.s. On the other hand,
  $
  \alpha = x^{*} \ \ \ P^*-\mbox{a.s.}  $ by construction. Therefore, $x^*$ is $P^*$-equivalent to $\alpha$, which means the weak solution $x^*$ is, in fact, strong.
   The theorem is proven.
\end{proof}
For $V=X$ we obtain the following generalization of Theorem \ref{th-fixed-point}:
\begin{corollary}\label{cor-fixed-point-general}
  Let the projective system of separable Banach spaces $\mathcal X=(X_t, p^{ut}, T)$ satisfy Property (${\Pi}$) and
  $h: \mathcal{P}a (X)\to \mathcal{P}a (X)$ be a local and tight-range operator. Then $h$ has at least one weak fixed point $x^*\in \mathcal Pa^*(X)$ for some  Young expansion $\mathcal B^*$ of the stochastic basis $\mathcal B$.

  If for any Young expansion $\mathcal B^*$ of $\mathcal B$, the operator $h$ has at most one weak fixed point in $\mathcal Pa^*(X)$, then each weak fixed point of the operator $h$ will be equivalent to a unique strong, i.e. belonging to the space $\mathcal{P}a (\mathcal X, \mathcal B)$, solution of the equation $hx=x$.
\end{corollary}

For some applications of this theorem see Subsection \ref{sec_ex_weak_sol} in Appendix D.

\appendix
\numberwithin{equation}{section}
%{subsection}

\section{Overview of the notation and definitions}\label{sec-Notation}

%\subsection{Frequently used notations and definitions}
\begin{itemize}
  \item $I_A$ is the indicator of a set $A$, i.e. $I_A(u)=1$ if $u\in A$ and  $I_A(u)=0$ otherwise.
      \item $T$ is a linearly ordered set with a maximal element $b\in T$, see Subsection \ref{sec-adapted-points}.
  \item $\mbox{Bor}(M)$ is the $\sigma$-algebra of all Borel subsets of a separable metric space $M$.
  \item $\mathcal G_1\otimes \mathcal G_2$ is the product of the $\sigma$-algebras $\mathcal G_i$ ($i=1,2$).
  \item $\mathcal S=(\Omega , {\mathcal F}, P)$ is a complete probability space, see (\ref{eq-probability-space}).
  \item $\mathcal B =(\Omega , {\mathcal F}, ({\mathcal F}_t)_{t\in {T}},P)$
 is a stochastic basis on the probability space $\mathcal S$, see (\ref{eq-stoch-basis}).
         \item ${\mathcal P}(M, \mathcal S)$ is the set of all (equivalence classes) of random points on the probability space $\mathcal S$ with values in a separable metric space $M$; the topology on ${\mathcal P}(M, \mathcal S)$ is defined by convergence in probability; this topology is metrizable by the metric $d_M(x,y)=E\min\{\rho(x,y); 1\}$ ($\rho$ is a metric on $M$); ${\mathcal P}(M, \mathcal S)$ can be simplified to ${\mathcal P}(M)$ if $\mathcal S$ is fixed, see Section \ref{sec-Prelim}.
           \item $\mathcal X=(X_t, p^{ut}, {T})$ is a projective system of separable Banach spaces $X_t$, where $p^{ut}:\ X_t\to X_u$ ($t, u\in{T}, \ t\ge u$) are linear continuous surjective maps satisfying the property $p^{vu}\circ p^{ut}=p^{vt} \ \ \mbox{for all} \ \ t,u,v\in{T}, \ t\ge u\ge v;$ the notational agreements throughout the paper: $X_b\equiv X$, $p^{tb}\equiv p^t$, see Definition \ref{def-projective-system}.
               \item The Euclidean projective system $\mathcal E = (E_i, p^{ji}, T_m) $ ($T_m\equiv \{0,...,m\}$) is generated by the $m$-dimensional Euclidean space $E=E_m$,  a decreasing sequence of its linear subspaces $E_i$ ($\mbox{dim}\, E_i=i$) and orthogonal projections $p^{ji}: E_{i}\to E_{j}$, see Definition \ref{def-Euclidean}.
      \item $\mathcal Z=(Z_t, q^{ut}, {T})$ is a projective system of separable Fr\'{e}chet spaces $Z_t$, where the bonding maps $q^{ut}$ satisfy the same property as $p^{ut}$ above; the notational agreement: $Z_b\equiv Z$, $q^{tb}\equiv q^t$; $\mathcal Z$ is used to construct Young expansions of stochastic bases, see Definition \ref{def-Young-extension}.
      \item Let  $\mathcal L=(L_t, l^{ut}, {T})$ be a projective system of separable Fr\'{e}chet spaces ($L=L_b$). We call $\phi: L\to L$  {a generalized Volterra operator (map)} with respect to $\mathcal L$ if it generates a family of operators $\phi^t: L_t\to L_t$ ($t\in {T}$) satisfying the properties $\phi=\phi^b$ and $l^{ut}\circ \phi^t=\phi^u\circ l^{ut}$ for all $t, u\in{T}, t\ge u$, see Definition \ref{def-Volterra}.
   \item The projective system $\mathcal L=(L_t, l^{ut}, {T})$ of separable Fr\'{e}chet spaces satisfies {Property $(\Pi)$} if  there exists a sequence $\pi_n: L\to L \ (L=L_b)$ of linear, continuous and finite dimensional generalized Volterra maps, which strongly converges to the identity map in $L$ as $n\to\infty,$
see Definition \ref{def-property-Pi}.
\item A random point $x\in\mathcal P(X, \mathcal S)$ is called adapted with respect to the stochastic basis $\mathcal B$ and the projective system $\mathcal X$ if $p^t(x): \Omega\to X_t$ is $\mathcal{F}_t$-measurable for all $t\in{T}$,  see Definition \ref{def adapted points}.
  \item $\mathcal{P}a(\mathcal X, \mathcal B)$ is the linear topological subspace of the space $\mathcal P(X, \mathcal S)$ consisting of all (equivalence classes of)     adapted points with respect to $\mathcal B$ and  $\mathcal X$; if $\mathcal X$ and/or $\mathcal B$ are fixed, then the notation $\mathcal{P}a(\mathcal X, \mathcal B)$ can be simplified to $\mathcal{P}a(X, \mathcal B)$ or $\mathcal{P}a(X).$
  %\item $\mathcal Ca(\mathcal X, \mathcal B^*)$ is the set of all a.s. $P^*$-continuous adapted random points, see Definition \ref{def adapted points}.
  \item $\mathcal Sa(\mathcal X, \mathcal B^*)$ consists of all $P^*$-a.s. continuous, simple random points $x: \Omega^*\to X$, see Definition \ref{def adapted points}.
  \item A local operator $h$ is characterized by the property
$
x|_A=y|_A \ \mbox{a.s.} \ \Rightarrow
\ hx|_A=hy|_A  \ \mbox{a.s.}
$
for all $A\subset \Omega$, see Definition \ref{def-local}.
  \item The superposition operator $h_f$ is {defined by} $(h_fx)(\omega)=f(\omega, x(\omega)),$ where $f: \Omega\times X\to Y$ is a given random function; any superposition operator is local.
      \item An LC operator is a local operator which is continuous in probability; the superposition operator $h_f$ is an LC operator if $f$ satisfies the Carath\'{e}odory conditions, see Section \ref{sec-Prelim}.

      \item A set ${\mathcal K}\subset {\mathcal P}(X, \mathcal F)$ is called {tight}
 if for any $\epsilon >0$ there exists a compact set $Q\subset X$ such that
$P\{\omega \,:\,x(\omega ) \notin Q\}<\epsilon $ whenever $x\in
{\mathcal K}$, see Definition \ref{def-tight-set}.
  \item
An operator
 is called tight (resp. tight-range) if it 1) maps bounded subsets of its domain (resp. the entire domain) into tight subsets of its range and 2)
it is uniformly continuous on
tight subsets of its domain, see Definition \ref{def-tight-oper}.
\item Given $\alpha\in\mathcal P(\mathcal F, Z)$, the measure $P\alpha^{-1}$ on the $\sigma$-algebra $\mathcal F\otimes \mbox{Bor}(Z)$ is  defined by
  $$
  P\alpha^{-1}(A\times B)=\{\omega\in A : \ \alpha(\omega)\in B\}
  $$
  %This is a simplified notation, in fact, we should have written $P (\mbox{Gr}\alpha)^{-1}$ for this measure.
\item $Pr_\Omega(Z)$ is the set of all random measures on $\mbox{Bor}(Z)$ with the marginal $P$; $Z$ is a Polish space, see Subsection \ref{sec-extensions-bases}.
\item The narrow topology on the set $Pr_\Omega(Z)$ of all random measures $\mu_\omega$ on $\mbox{Bor}(Z)$ with the marginal $P$ is generated by the maps
$$
\mu\mapsto\mu(f)=E\int_Zf(\omega, z)d\mu_\omega(z),
$$
where $f:\Omega \times Z\to R$ is an arbitrary bounded Carath\'{e}odory function, see Subsection \ref{sec-extensions-bases}.
     \item Given a projective system of separable Fr\'{e}chet spaces $\mathcal Z=(Z_t, q^{ut}, {T})$,
    a Young expansion $\mathcal B^*\!\!=\!\!(\Omega^*,{\mathcal F}^*\!,({\mathcal F}_t^*)_{t\in {T}},P^*\!)$ of the stochastic basis $\mathcal B\!\! =\!\! (\Omega ,{\mathcal F}\!,({\mathcal F}_t)_{t\in {T}}\!, P\!)$, generated by $\mathcal Z$, satisfies the following properties:  1) $\Omega^*=\Omega\times Z$, where $Z=Z_b$;
  2) $P^*_\omega$ is the limit point (in the narrow topology) of a sequence of random Dirac measures $\delta_{\alpha_\nu(\omega)}$, where
$\alpha_\nu\in\mathcal Pa(Z,\mathcal B)$; 3) $\mathcal F^*$ is the $P^*$-completion of the $\sigma$-algebra $\mathcal F\otimes (\mbox{Bor}\,(Z_t))$ and 4) $\mathcal F^*_t$ is the $P^*$-completion of the $\sigma$-algebra $\mathcal F_t\otimes (q^{t})^{-1}(\mbox{Bor}\,(Z_t))$ for any $t\in T,$ see Definition \ref{def-Young-extension}.

\end{itemize}

%%%%%%%%%%%%%%%%%%%%%

%%%%%%%%%%%%%%%%%%%%

\section{Proof of the auxiliary results}

We start this section with some technical results.

\subsection{Lemmata}

\begin{lemma}\label{lem-isomorphism}
  Let $\mathcal B$ be a stochastic basis on the probability space $\mathcal S$ and $\mathcal X=(X_t, p^{ut}, {T})$ be a projective system of finite dimensional Banach spaces. Then, given a linear bijection $G: X\to E_m $, $E_m$ being the $m$-dimensional Euclidean space, there exists a finite stochastic basis $\mathcal B_m =(\Omega , {\mathcal F}, ({\mathcal F}_i)_{i\in {T(m)}}, P)$ on $\mathcal S$, for which the superposition operator $(h_G x)(\omega)=G(x(\omega))$ defines a linear isomorphism between the linear topological spaces $\mathcal{P}a(\mathcal X, \mathcal B)$ and $\mathcal{P}a(\mathcal E, \mathcal B_m)$, where $\mathcal E$ is a Euclidean projective system from Definition \ref{def-Euclidean}.
  \end{lemma}
\begin{proof}
First of all, we notice that $X$ can be identified with $E_m$ if we replace the bonding maps $p^t\equiv p^{tb}$ ($b$ is the maximal element in $T$) with $p^{t}\circ G^{-1}$ and leave the remaining bonding maps unchanged. In this case, $G$ becomes the identity map, and we have to prove that $\mathcal{P}a(\mathcal X, \mathcal B)=\mathcal{P}a(\mathcal E, \mathcal B_m)$ for some finite stochastic basis $\mathcal B_m$. Let us construct it. To this end, consider the nonincreasing family of subspaces $\mbox{Ker}\,p^{t}$ of the space $X$. For each $t\in T$ we define
$E_t$ to be the orthogonal complement of $\mbox{Ker}\,p^{t}$
in the space $X$. Let $T'\subset \{0,...,m\}$ be the set of indices $i$ such that there exists $E_t$ for which $i=\mbox{dim}\, E_t$. We define $E_i=E_t$ for these $t$ and $p^{ji}$
to be the induced linear maps from $E_i$ onto $%
E_j $ ($i\ge j$) defined as $\kappa_u^{-1}\circ p^{ut}\circ \kappa_t$ ($t\in T_i$, $u\in T_j$), where $\kappa_s\equiv p^s|_{E_k}$ ($s\in T_k$) is the linear isomorphism between the spaces $E_k$ and $X_s$ of the same dimension.

Changing the basis in $E_m$ we may always assume that
  $E_i=\{(x_1,...,x_i, 0, ...,0)\}$ and $p^{ji}$ ($i, j\in T'$, $i\ge j$) is the orthogonal projection, which removes the coordinates $(x_{j+1},..., x_{i})$.

This defines a Euclidean projective system $\mathcal E'=(E_i, p^{ji}, T')$.

Putting
\[
{\mathcal F}_i\equiv \bigcap_{\mbox{dim}\,X_t=i}{\mathcal F}_t
\]
results in the finite stochastic basis $\mathcal B' =(\Omega , {\mathcal F}, ({\mathcal F}_i)_{i\in {T}'}, P)$.

 By construction, the map $p^t|_{E_i}$ is a linear isomorphism onto $X_t$ if $\mbox{dim}\,X_t=i$. Therefore $p^tx$ is $\mathcal F_t$-measurable for all $t$ satisfying $\mbox{dim}\,X_t=i$ if and only if $p^ix$ is $\mathcal F_i$-measurable. Thus, we have proven that $x\in \mathcal{P}a(\mathcal X, \mathcal B)$ if and only if $p^ix$  is $\mathcal F_i$-measurable for any $i\in T'.$ Hence $\mathcal{P}a(\mathcal X, \mathcal B)=\mathcal{P}a(\mathcal E, \mathcal B').$

  To extend the projective system $\mathcal E'$ from the subset $T'\subset \{0,1,...,m\}$ to the entire set ${T}_m=\{0,1,...,m\}$ let us take any $k\in T_m-T'$, put $E_k=\{(x_1,...,x_k, 0, ...,0)\}$ and define $p^{lk}: E_k\to E_l$ ($k, l\in T_m, k\ge l$) to be the orthogonal projection removing the coordinates  $(x_{l+1},..., x_{k})$. This yields the projective system $\mathcal E_m=(E_i, p^{ji}, T_m)$. The corresponding filtration $(\mathcal F_i)_{i\in T_m}$ coincides with the previous one if $i\in T'$, while for $i\in T_m-T'$ we
   put $\mathcal F_i=\mathcal F_k$,  where $k$ is the least number from $T'$ which exceeds $i$.

  Evidently, for any $x:\ \Omega \to E$, the random point $p^i(x)\equiv p^{im}(x)$  is $\mathcal F_i$-measurable for any $i\in T(m)$ if and only if $p^ix$  is $\mathcal F_i$-measurable for any $i\in \bar T$, so that
  $
  \mathcal Pa(\mathcal X, \mathcal B)=\mathcal Pa(\mathcal E', \mathcal B')=\mathcal Pa(\mathcal E_m, {\mathcal B}_m).
  $
\end{proof}

\begin{lemma}\label{lem-projection} Let $\mathcal X=(X_t, p^{ut}, {T})$ be a projective system of finite dimensional Banach spaces and $\mathcal B$ be a stochastic basis on a probability space $\mathcal S$.
 Then for any nonempty, convex and compact subset $U$ of $X=X_b$ there exists a continuous projection $\phi: \, X\to U$, for which
  $h_\phi(\mathcal{P}a(\mathcal X, \mathcal B))=\mathcal{P}a(\mathcal X, \mathcal B)\cap \mathcal P(U, \mathcal S).$
\end{lemma}

\begin{proof}

 Step 1. We first replace $\mathcal X$ and $\mathcal B$ with the Euclidean projective system $\mathcal E=(E_i, p^{ji}, T_m)$ and a finite stochastic basis $\mathcal B_m$. The set $U\subset X$ will be at this step replaced by a nonempty, convex and compact subset $W\subset E\equiv E_m$.

 Redefining the coordinates we may assume that $E_i\!\!=\!\!\{\!(x_1,...,x_i, 0, ...,0)\!\}$ and $p^{ji}$ are the orthogonal projection, which removes the coordinates $(x_{j+1},..., x_{i})$.
 Let $p^i\equiv p^{im}$, $W_i=p^{i}(V)$ and construct a Volterra projection $\psi^i: \ E_i\to W_i$ by induction. For $i=1$, the set $W_1$ is a closed, bounded interval $[a,b]$, so that we simply put $\psi^1=\pi_{[a,b]}$, where
  $$\pi_{[a,b]}(x_1)=x_1 \ \mbox{if} \ x_1\in [a,b], \ \pi_{[a,b]}(x_1)=a \ \mbox{if} \ x_1< a \ \mbox{and} \  \pi_{[a,b]}(x_1)=b \ \mbox{if} \ x_1>b.$$
  Assuming that $\psi^{k-1}: \, E_{k-1}\to W_{k-1}$ is constructed, we observe that for each $x^{k-1}\in W_{k-1}$ the set $\{(x^{k-1}, x_k)\}\cap W_k$ is again a closed, bounded interval $[a(x^{k-1}), b(x^{k-1})]$, where the functions $a(\cdot) \le b(\cdot)$ are continuous on $W_{k-1}$, as $W_k$ is convex and compact. Put $$\psi^k((x^{k-1}, x_k))=
  \left(\psi^{k-1}(x^{k-1}), \pi_{[a(x^{k-1}), b(x^{k-1})]}(x_k)\right).$$
  Then $\psi^k: E_k\to W_k$ is continuous and by construction satisfies $p^{k-1,k}\circ\psi^k=\psi^{k-1}\circ p^{k-1,k}$. Therefore, $\psi^k$ is Volterra, and this completes the induction argument. Note that the superposition operator $h_\psi$ maps adapted points into adapted points, see Remark \ref{rem-Volterra}. Thus, we have proven the lemma for the case of $\mathcal E$ and $\mathcal B_m.$

  Step 2. Applying Lemma \ref{lem-isomorphism} we can reduce the general case to the one considered in step 1. Assume that the linear map $G: X \to E_m$ induces the linear topological isomorphism $h_G: \ \mathcal Pa(\mathcal X, \mathcal B)\to \mathcal Pa(\mathcal E, \mathcal B_m)$, put $W=G(U)\subset E$ and define $\phi=G^{-1}\circ \psi\circ G$, where $\psi: E\to W$ is a Volterra projection. By construction, $\phi$ is a continuous projection from $X$ onto $U$. Note that $h_\phi=h_{G^{-1}}\circ h_\psi\circ h_G=h^{-1}_G\circ h_\psi\circ h_G$. As the mapping $h_\psi: \mathcal Pa(\mathcal E, \mathcal B_m)\to Pa(\mathcal E, \mathcal B_m)\cap \mathcal P (W, \mathcal S)$ is a continuous projection, then so is the mapping $h_\phi: \mathcal{P}a(\mathcal X, \mathcal B)\to \mathcal{P}a(\mathcal X, \mathcal B)\cap \mathcal P(U, \mathcal S).$
   \end{proof}

\begin{lemma}\label{lem-appr-by-continuity-sets}
  Let $P^*$ be a Young probability measure defined on the $\sigma$-algebra $\mathcal F\otimes \mbox{Bor}\,(Z)$, where $Z$ is a Polish space.
  Suppose that $A=\bigcup\limits_{i=1}^s A_i$, where $A_i\in \mathcal F\otimes \mbox{Bor}\,(Z)$ are disjunct subsets and $\varepsilon >0$. Then there exist disjunct $P^*$-continuity subsets $B_i\in \mathcal F\otimes \mbox{Bor}\,(Z)$ such that $B=\bigcup\limits_{i=1}^s B_i$ and $\sum\limits_{i=1}^sP^*(A_i\triangle B_i)<\varepsilon.$
\end{lemma}
\begin{proof}
  Step 1. We first prove this result for $s=1$, i.e. for a given $A\in \mathcal F\otimes \mbox{Bor}\,(Z)$ and $\varepsilon >0$ we shall find a $P^*$-continuity subset $B\in \mathcal F\otimes \mbox{Bor}\,(Z)$ such that $P^*(A\triangle B)<\varepsilon$. Indeed, there exist $\Omega_j\in \mathcal F$ and closed subsets $C_j\subset Z$ ($j=1,...,J$) such that the set $A_\varepsilon=\bigcup\limits_{j=1}^J (\Omega_j\times C_j)$ satisfies
  $P^*(A\triangle A_\varepsilon)<\varepsilon/2.$ Consider $\delta$-neighborhoods $C_j^\delta$ of the sets $C_j$. Clearly,  $\partial (C_j^\delta)$ have no common points for different $\delta$. Therefore, there exist sequences $\delta_{n}^j\to 0$  ($j=1,...,J$, $n\to\infty$), for which
$P^*(\partial (\Omega_j\times C_j(\delta_{n}^j)))=0$ for all $j=1,...,J$ and $n\in N$. On the other hand, $\bigcap\limits_{n=1}^\infty (\Omega_j\times C_j(\delta_{n}^j))=\Omega_j\times C_j$, so that there is a number $k\in N$ such that $P^*((\Omega_j\times C_j(\delta_{k}^j))-(\Omega_j\times C_j))<\varepsilon /2J$ for all $j=1,...,J$. The set $B=\bigcup\limits_{j=1}^J (\Omega_j\times C_j(\delta_{n}^j))$ is {a continuity set} of the measure $P^*$ and
$$
P^*(A\triangle B)\le P^*(A\triangle A_\varepsilon)+\sum\limits_{i=1}^J P^*((\Omega_j\times C_j(\delta_{n}^j))-(\Omega_j\times C_j))<\varepsilon.
$$

Step 2. Consider the case of $s=2$. Let $A_1, A_2\in \mathcal F\otimes \mbox{Bor}\,(Z)$, $A_1\cap A_2=\emptyset$, $A=A_1\cup A_2$. It follows from step 1 that for any $\varepsilon >0$ there is a $P^*$-continuity subset $B\in \mathcal F\otimes \mbox{Bor}\,(Z)$ such that $P^*(A\triangle B)<\varepsilon$. We shall find two $P^*$-continuity subsets $B_1$, $B_2$ such that
\begin{equation}\label{eq-sec-extensions-prob-spaces-1}
  B_1 \cap B_2=\emptyset, \ \ \  B_1 \cup B_2=B, \ \ \ \ P^*(A_1\triangle B_1)+P^*(A_2\triangle B_2)<\varepsilon.
\end{equation}
Put $C_1=A_1\cap B$ and $C_2=B-C_1=B-A_1$. Clearly, $C_1\cap C_2=C_1\cap A_2=C_2\cap A_1=\emptyset$ and $C_1\cup C_2=B$. Therefore,
$$
(A_1\triangle C_1)\cap (A_2\triangle C_2)=\emptyset \ \ \ \mbox{and} \ \ \ (A_1\triangle C_1)\cup (A_2\triangle C_2)=A\triangle B,
$$
so that
$$
P^*(A_1\triangle C_1)+ P^*(A_2\triangle C_2)=P^*(A\triangle B)<\varepsilon
$$
Let $\sigma=\varepsilon-P^*(A\triangle B)>0$. Applying the result from step 1 for $A_1$ and $A_2$ yields two   $P^*$-continuity subsets $B'_1, B'_2\in \mathcal F\otimes \mbox{Bor}\,(Z)$ such that $P^*(A_1\triangle B'_1)<\sigma/2$ and $P^*(A_2\triangle B'_2)<\sigma/2$. Define $B_1=B'_1\cap B$ and $B_2=B-B_1$. Then
$$
P^*(B_1\triangle C_1)=P^*((B_1\cap B)\triangle C_1)\le P^*(B'_1\triangle C_1)<\sigma/3,
$$
as $C_1\subset B$. Moreover,
$$
P^*(B_2\triangle C_2)=P^*((B-B_1)\triangle (B-C_1))=P^*(B_1\triangle C_1)<\sigma/3,
$$
as $B_1, C_1\subset B$. Summarizing we obtain
$$
\begin{array}{c}
P^*(A_1\triangle B_1)+ P^*(A_2\triangle B_2)\le P^*(A_1\triangle C_1)+ P^*(A_2\triangle C_2)
 \\
+ P^*(B_1\triangle C_1)+ P^*(B_2\triangle C_2)=P^*(A\triangle B)+\frac{2\sigma}{3}<\varepsilon,
\end{array}
$$
which concludes the proof if $s=2$.

Step 3. The general case is treated by induction. Suppose that the statement is proven for $s-1$, define $A^1=\bigcup\limits_{i=1}^{s-1}A_i$, $A^2=A_s$ and construct, as in step 2, two disjunct $P^*$-continuity subsets $B^1$, $B^2$, for which $\varepsilon^1+\varepsilon^2<\varepsilon$, where $\varepsilon^1=P^*(A^1\triangle B^1)$ and $\varepsilon^2=P^*(A^1\triangle B^1)$. Applying the induction hypothesis, we get
disjunct $P^*$-continuity subsets $B_1,..., B_{s-1}$ such that
$$
\sum\limits_{i=1}^{s-1}P^*(A_i\triangle B_i)<\varepsilon_1 \ \ \ \mbox{and} \ \ \ \bigcup\limits_{i=1}^{s-1}B_i=B^1.
$$
Adding $B_s\equiv B^2$ to $B_1,...,B_{s-1}$ yields $s$ a set of disjunct  $P^*$-continuity subsets satisfying
$$
\sum\limits_{i=1}^{s}P^*(A_i\triangle B_i)=\sum\limits_{i=1}^{s-1}P^*(A_i\triangle B_i)+P^*(A^2\triangle B^2)
<\varepsilon_1+\varepsilon_2<\varepsilon.
$$
\end{proof}

\begin{lemma}\label{th-Nemytskii}
  Suppose that $U(\omega)$ ($\omega\in\Omega$) is a random closed, convex, bounded and nonempty subset of $R^m$ such that $$\mbox{Gr}\, U\equiv \{(\omega, U(\omega)) \ : \ \omega\in\Omega)\}\in \mathcal F\otimes \mbox{Bor}\,(R^m).$$ Let $$\mathcal A={\mathcal P}(U)\equiv {\mathcal P}(R^n)\cap \{x: \ x(\omega)\in U(\omega) \ \mbox{a.s}\}$$ and $h: {\mathcal A} \to {\mathcal A}$ be an LC operator. Then $h$ has at least one fixed point in $\mathcal A$.
\end{lemma}
\begin{proof}
  The proof is based on the generalization of \textit{the Nemytskii conjecture}. The latter states that the Carath\'{e}odory conditions on $F$ are not only sufficient, but also necessary for the superposition operator $h_F$ to be continuous in measure. This conjecture, in a slightly adjusted form, was proven in \cite{Pon-0}, together with its generalization for arbitrary LC operators. More precisely, the main result in \cite{Pon-0} says that for an LC operator $h: {\mathcal A} \to {\mathcal A}$ there exists a Carath\'{e}odory function $f: \mbox{Gr}\, U\to {R^m}$ such that $hx=h_fx$ $P$-a.s. for any $x\in \mathcal A$. Evidently, $f(\omega,\cdot)$ leaves the set $U(\omega)$ a.s. invariant.
  By Brouwer's fixed-point theorem, the set $\mbox{Fix}\,(\omega)$ consisting of all fixed points $x_\omega\in U(\omega)$ of the map $f(\omega, \cdot):U(\omega) \to U(\omega)$ is a.s. nonempty. On the other hand, the function $F(\omega, x)=f(\omega, x)-x$  is Carath\'{e}odory and hence $\mathcal F\otimes \mbox{Bor}\,(R^m)$-measurable. Therefore, $\{(\omega, \mbox{Fix}\,(\omega)), \ \omega\in\Omega\}=G^{-1}(0)\in \mathcal F\otimes \mbox{Bor}\,(R^m)$ and by the measurable selection theorem  (see e.g. \cite[p. 10]{Crauel}) there exists a $\mathcal F$-measurable function $x: \Omega\to R^n$ such that $x(\omega)\in U(\omega)$ a.s. Thus, $x\in \mathcal A$ and, by construction, $hx=h_fx=x$ a.s.
\end{proof}

Let us remark that the representation theorem from \cite{Pon-0} is not valid for all subsets $\mathcal A \subset \mathcal P(R^n, \mathcal S).$ On the other hand, the fixed-point result from Lemma \ref{th-Nemytskii} is not valid either for arbitrary closed, convex, bounded and nonempty subsets of $\mathcal P(R^n, \mathcal S)$, see \cite{Pon-2}.

%Now we will extend the previous result to the case of sets of adapted random points.
\subsection{Proof of Theorem \ref{th-uniform-cont}}\label{sec-proof of uniform th}
\hfill \\

\textit{1) $\Rightarrow$ 2)} is trivial as ${\mathcal P}a(Q)$ is tight if $Q$ is compact.

\textit{2) $\Rightarrow$ 3).} We will use the third description of uniform continuity (see Subsection \ref{sec-uniform-continuity}).
Let $Q_0\subset X$ be an arbitrary compact and $\gamma>0$ be fixed.
Define $Q$ to be the closed convex hull of the set $\bigcup_{n\in N}\pi_n (Q_0)$. Clearly, $Q$ is compact and $Q_0\subset Q$.
Pick arbitrary $\varepsilon >0$
%assuming, without loss of generality, that $\varepsilon<\gamma$,
and
choose $\rho>0$ so that
 \begin{equation}\label{eq-6-1}
%\mbox{
\|x'-y'\|_X\le\rho \ \mbox{a.s} \ \ \ \mbox{implies} \ \ \
d_Y(hx',hy')<\frac{\varepsilon}{3} \ \ \ \forall x', y'\in \mathcal Pa(Q).
\end{equation}
Take arbitrary $x, y\in \mathcal Pa(\mathcal X, \mathcal B)$ which satisfy ${  P}\{\Vert x-y\Vert_X >{%
\frac \rho 3}\}<\frac \epsilon 3 $ and $P\{x\notin Q_0\}<\gamma, P\{y\notin Q_0\}<\gamma$ and
fix a sufficiently large number $n$ (depending on $x$ and $y$), for which

1) ${
P}\{\Vert \pi _nx-x\Vert_X >\frac \rho 3 \}<\frac \varepsilon 3$, ${
P}\{\Vert \pi _ny-y\Vert_X >\frac \rho 3 \}<\frac \varepsilon 3$,

\noindent
so that

\begin{equation}\label{eq-6-2}
{P}\{\Vert \pi _nx-\pi _ny\Vert_X > \rho \}<\varepsilon,
\end{equation}

\noindent
and

2)
$d_Y(h(\pi _nx),hx)<\frac \varepsilon 3$, \
$d_Y(h(\pi _ny),hy)<\frac \varepsilon 3$.

Using $\pi_n$ let us define the finite dimensional projective system  $\mathcal X^n$ as it is done in (\ref{eq-finite-dim-system})
and consider
the direct product $\mathcal E^n$ of two copies of $\mathcal X^n$, the compact convex subset $W^n= \{(x,y)\in \pi_n(X)\times \pi_n(X): \ x,y\in Q, \|x-y\|_X\le \rho\}$ and the continuous projection $\phi_n: \pi_n(X)\times \pi_n(X)\to W^n$ such that the corresponding superposition operator $h_{\phi_n}$ maps $\mathcal Pa(\mathcal E^n, \mathcal B^*)$ to $\mathcal Pa(\mathcal E^n, \mathcal B^*)\cap \mathcal P(W^n, \mathcal B^*)$. Such a projection exists due to Lemma \ref{lem-projection}. Put $(u,v)=h_{\phi_n}(\pi_{n}x, \pi_{n}y).$ By construction, $\|u-v\|_X\le\rho$, which implies $d_Y(hu,hv)<\frac \varepsilon 3$ due to (\ref{eq-6-1}).

By (\ref{eq-6-2})
$u$ and $v$ coincide with $\pi_nx$ and $\pi_ny$, respectively, on a measurable subset $\Omega'$ of $\Omega$ where $x, y$ belong to $Q_0$ (because in this case $\pi_nx$ and $\pi_ny$ belong to $Q$)  and where $\Vert x-y\Vert_X\le\rho$.  Therefore, $P(\Omega-\Omega')< 3\gamma.$
%\[
%P\{x'\ne \pi _mx\}<\varepsilon \ \ \ \mbox{and} \ \ \ P\{x'\ne \pi _my\}<\varepsilon,
%\]
%as $x, y\in \mathcal K$ and $\pi_mx, \pi_my \in \pi_m(Q)\subset C$.

%\[
%\mbox{\bf  P}\{x^{\prime }\ne \pi _N(x)\}\le
%\]
%\[
%\le \mbox{\bf  P}\{\pi _N(x)\not{\in}\pi _N(Q)\}+\mbox{\bf P}\{\pi _N(y)\not%
%{\in}\pi _N(Q)\}+\mbox{\bf P}\{\Vert \pi _N(x)-\pi _N(y)\Vert \ge \delta
%\}\le
%\]
%\[
%\le \mbox{\bf P}\{x\not{\in}Q\}+\mbox{\bf  P}\{y\not{\in}Q\}+\mbox{\bf  P}%
%\{\Vert \pi _N(x)-\pi _N(y)\Vert \ge \delta \}\le 3\varepsilon .
%\]
%By a similar reason, $\mbox{\bf
%P}\{x^{\prime }\ne \pi _N(x)\}\le 3\varepsilon $.

Hence
\[
\begin{array}{c}
d_Y(hx,hy)\le
d_Y(hx,h(\pi _nx))+
d_Y(hy,h(\pi_ny))+
d_Y(h(\pi_nx), h(\pi_ny)) \\
<\frac{2\varepsilon}{3} +
d_Y(hu,hv)+P(\{u\ne
\pi _nx\}\cup\{v\ne \pi_ny\})<\varepsilon+P(\Omega-\Omega')<\varepsilon+3\gamma.
\end{array}
\]
Setting $O(\gamma)=3\gamma$ completes the proof of the statement.

\textit{3) $\Rightarrow$ 1).}
Let $\mathcal K$ be a tight subset of $\mathcal Pa (\mathcal X, \mathcal B).$ Take arbitrary $\varepsilon >0, \sigma>0$ and find a compact $Q\subset X$ for which $P\{x\notin Q\}<\gamma$ for any $x\in \mathcal K$, where $O(\gamma)<\frac \varepsilon 2$. By assumption, there exist $\delta >0$ such that $d_X(x,y)<\delta$ implies $d_Y(hx,hy)<\frac \varepsilon 2 + O(\gamma)<\varepsilon$. Therefore $h$ is uniformly continuous on $\mathcal K$.
\subsection{Proof of Theorem \ref{th-approx-simple-points}}\label{sec-appr-simple-points}
\hfill \\

We split up the proof into 5 steps. In steps 1-4 we prove

\textit{The simplified version of Theorem \ref{th-approx-simple-points}:}
{\it
For any $x\in \mathcal{P}a(\mathcal X, \mathcal B^*)$ such that $x|_A=0$ $P^*$-a.s. on the set $A\in \mathcal F^*$ and any $\varepsilon>0$, there exist $y\in \mathcal{S}a(\mathcal X, \mathcal B^*)$ and a $P^*$-continuity subset $B\in \mathcal F^*$, for which

(1) $y|_{B}=0$ $P^*$-a.s.,

(2) $d_X^*(x,y)<\varepsilon$, and

(3) $P^*(A\triangle B)<\varepsilon$.}
%We remark that $P^*A>0$ is only possible if $0\in V$.

In the course of the proof  the random point $x$ and the set $A$ will be successively simplified by constructing special approximations with an arbitrary precision. This will be done in steps 1 - 3. In step 4 the proof of the simplified version of Theorem \ref{th-approx-simple-points} will be completed. Here we will use Lemma \ref{lem-appr-by-continuity-sets} and simplifications from steps 1-3. The proof of Theorem \ref{th-approx-simple-points} will be finished in step 5.

Here and in the sequel $d_X^*(u,v)$ is the distance on the metric space $\mathcal{P}(X, \mathcal S^*)$, where $\mathcal S^*$ is the probability space hosting the stochastic basis $\mathcal B^*$.

\textit{Step 1. The random point $x$ may be assumed to take values in a finite dimensional subspace.}
Pick an arbitrary $\varepsilon >0$ and denote by $\pi_\nu: X\to X$ linear finite dimensional Volterra maps strongly converging to the identity map in $X$ as $\nu\to\infty$. We use the index $\nu$ instead of $n$ is this proof, as $n$ is already included in the formulation of Theorem \ref{th-approx-simple-points}.
%They exist due to  Property $(\Pi)$ of the projective system $\mathcal X$.
%Let $\pi_\nu^t: X_t\to X_t$ be the maps generated by $\pi_\nu$, see Definition \ref{def-Volterra}.

Let $\mathcal X^\nu$ be the finite dimensional projective systems defined in (\ref{eq-finite-dim-system}) by means of $\pi_\nu$.
Note that $\pi_\nu x\in \mathcal{P}a(\mathcal X^\nu, \mathcal B^*)\subset \mathcal{P}a(\mathcal X, \mathcal B^*)$ due to Remark \ref{rem-Volterra}.
Evidently, $\pi_\nu x|_A=0$ $P^*$-a.s. and the strong convergence of the sequence  $\{\pi_\nu\}$ to the identity map in $X$ implies convergence of
$\{\pi_\nu x\}$ to $x$ in probability $P^*$ as $\nu\to\infty$. Therefore, $x$  can be approximated, with an arbitrary precision, by $\pi_\nu x$ for sufficiently large $\nu$, the set $A$ being unchanged. All this means that the projective system $\mathcal X$ can be replaced by its finite dimensional approximation $\mathcal X^\nu$.
Moreover, utilizing the construction from the proof of  Lemma \ref{lem-isomorphism}, we can replace $\mathcal X^\nu$ and $\mathcal B^*$ with the projective system $\mathcal E=(E_i, p^{ji}, T_m)$ and the finite stochastic basis $\mathcal B_m^* =(\Omega^* , {\mathcal F^*}, ({\mathcal F}^*_i)_{i\in {T}_m}, P)$, respectively, constructed as follows:

1) $E_i=\{(x_1,...,x_i, 0, ...,0)\}$ and $p^{ji}$ ($i\ge j$) is the orthogonal projection, which removes the coordinates $(x_{j+1},..., x_{i})$;

2) $T_i=\{t: \ \mbox{dim} X_t=k\}, \ \ \mbox{where} \ k=\mbox{min} \{j\ge i: \ T_j\ne\emptyset\}$;

3)
$\mathcal F_i^*=\bigcap_{t\in T_i} F_t^*$ ($i\in {T}_m)$.

Evidently, $\mathcal F_i^*$ is the $P^*$-completion of the $\sigma$-algebra
$
\mathcal F_i\otimes \bigcap_{t\in T_i}(q^{t})^{-1}(\mbox{Bor}\,(Z_t)),
$
where
$
{\mathcal F}_i=
\bigcap\limits_{t\in T_i}{\mathcal F}_t.
$
 Due to Lemma \ref{lem-isomorphism}, there exists a local linear isomorphism between the topological spaces $\mathcal Pa (\mathcal X^\nu, \mathcal B^*)$ and $\mathcal Pa(\mathcal E, \mathcal B^*_m)$, so that the latter can replace the former in the next steps of the proof.

\textit{Step 2. The random point $x$ may be assumed to take finitely many values.}

We proceed with assuming that $x\in \mathcal Pa(\mathcal E, \mathcal B_m^{*})$, which is easy to see to be equivalent to the representation $x=\sum\limits_{i=1}^m\alpha_ie_i$ a.s., where $(e_1,...,e_m)$ is the standard basis in $E=E_m$ and $\alpha_i: \ \Omega \to R$ ($i=1,...,m$) is $\mathcal F_i^*$-measurable ($i=1,...,m$).
From the property $x |_A =0$ $P^*$-a.s., we conclude that $\alpha_i |_A=0$ $P^*$-a.s. for all $i=1,...,m$. Putting $\Omega_i^{*,1}=\{\omega^*\in\Omega^*: \ \alpha_i(\omega^*)=0\}\in \mathcal F_i^*$, we obtain $P^*\left(A\triangle (\cap_{i=1}^m \Omega_i^{*,1})\right)=0$.

%Let $\varepsilon >0$ be arbitrary.
Using standard approximation technique for the $F_i^*$-measurable, real valued functions $\alpha_i$ we can find, for arbitrary $\varepsilon$ and each $i$, sets $\Omega_{ij}\in \mathcal F_i$, $B_{ij}\in \bigcap_{t\in T_i}(q^{t})^{-1}(\mbox{Bor}\,(Z_t))$, real constants $a_{ij}$ ($1\le i \le m, 1\le j \le s$) and a natural number $1\le r \le s$ satisfying the following properties:
\begin{equation}\label{eq_x_y}
\begin{array}{l}
P^*(\Omega^{*}\triangle(\bigcup\limits_{j=1}^s(\Omega_{ij}\times B_{ij})))<\varepsilon, \ \  \  a_{ij}=0 \ (1\le i \le m, 1\le j \le r)\\
d^*_R(\alpha_i,\alpha_i')<\frac{\varepsilon}{m}, \ \ \
P^*(\Omega_i^{*,1}\triangle \Omega_i^{*,2})<{\varepsilon}, \ \ \ \mbox{where} \\
\Omega_i^{*,2}=\bigcup\limits_{j=1}^r(\Omega_{ij}\times B_{ij}) \ \ \ \mbox{and} \ \ \
      \alpha_i'=\sum\limits_{j=1}^s a_{ij}I_{\Omega_{ij}\times B_{ij}}, \ \ \  \alpha_i' |_{\Omega_i^{*,2}}= 0.
 \end{array}
\end{equation}
Here $d_R^*$ is the following metric on the space of $F_i^*$-measurable random points:
\begin{equation}\label{eq_metric_R}
  d_R^*(\alpha,\alpha')=E^*(\min\{|\alpha-\alpha'|;1\}.
\end{equation}
In what follows we assume, by technical reasons, that
the norm in the finite dimensional space $E$ is defined as $\|x\|_E=\sum\limits_{i=1}^m|\alpha_i|$. In this case,
$$
d^*_E(x, x')=E^*(\min\{||x-x'||_E;1\}\le \sum\limits_{i=1}^mE^*(\min\{|\alpha_i-\alpha_i'|;1\}<\varepsilon,
$$
where $x'=(\alpha_1',...,\alpha_m').$
As $\varepsilon >0$ is arbitrary, we can redefine $x=(\alpha_1,...,\alpha_m)$ and $A$ to be
\begin{equation}\label{eq_x_y_redefined_2}
x=x'=(\alpha_1',...,\alpha_m')\ \ \ \mbox{and} \ \ \ A=\bigcap\limits_{i=1}^m\Omega_i^{*,2},
\end{equation}
where $\alpha_i'$, $\beta_i'$ and $\Omega_i^{*,2}$ are defined in (\ref{eq_x_y}). By construction, $x$, so redefined, assumes finitely many values and $x|_A=0$ on the new subset $A$. This simplification is used in Step 3.

%%%%%%%%%%%10.11.2021 Stopped here

\textit{Step 3. The random points $x$ and $y$ can be assumed to be measurable with respect to the $\sigma$-algebra of random cylinder sets.}

Examples show that the $\sigma$-algebras $\bigcap_{t\in T_i}(q^{t})^{-1}(\mbox{Bor}\,(Z_t))$ may not necessarily be the Borel $\sigma$-algebras on some Polish space, so that Lemma \ref{lem-appr-by-continuity-sets} cannot be directly used in connection with these $\sigma$-algebras. However, Property $(\Pi)$ for the projective system $\mathcal Z$ helps to avoid this problem by replacing the $\sigma$-algebras $F^*_i$ by their finite dimensional approximations based on the finite dimensional projective systems
$\mathcal Z^\nu=(Z^\nu, q^{ut}|_{Z_\nu}, T)$, where $Z^\nu=q_\nu(Z)$, so that the corresponding intersections of cylinder $\sigma$-algebras will be Borel on some Polish space.

Let $\tau_\nu: Z\to Z$ be finite dimensional Volterra maps, which strongly converge to the identity map in $Z$ as $\nu\to\infty$,
and $\tau_\nu^t:\ Z_t\to Z_t$ be the maps generated by $q_\nu$, see Definition \ref{def-Volterra}.
For any $i=1,...,m$, $t\in T_i$ consider
$$
\phi_{t,\nu}\equiv q^t\circ \tau_\nu=\tau^t_\nu\circ q^t: \ Z \to Z_t^\nu \equiv q^t(Z^\nu),
$$
and the associated measure $P^*_{t,\nu}\!\equiv\! P^*\phi_{t,\nu}^{-1}$ defined on the $\sigma$-algebra $\mathcal F_i\otimes (q^t)^{-1}(\mbox{Bor}(Z^\nu))$ by
$$
P^*_{t,\nu}(\Delta)=P^*\{(\mbox{id}\times \phi_{t,\nu})^{-1}(\Delta)\},
$$
where $\mbox{id}:\Omega\to\Omega$ is the identity map.

Denote by $\mathcal F^*_{t, \nu}$ the completion of the $\sigma$-algebra $\mathcal F_i\otimes \mbox{Bor}(Z^\nu)$ w.r.t. the measure  $P^*_{t,\nu}$ and put, for any $B_{ij}$ from Step 2,
$B_{ij}^{t,\nu} = \phi_{t,\nu}(B_{ij})$. As $B_{ij}\in \mbox{Bor} (Z)$, its image $B_{ij}^{t,\nu}$ under the continuous map $\phi_{t,\nu}$ can be obtained by an $\mathcal A$-operation from the closed subsets of the space $Z^\nu_t$, see e.g. \cite[Th. 2.4.2]{Bog}. Then, using the same $\mathcal A$-operation we obtain the set $\Omega_{ij}\times B_{ij}^{t,\nu}$ from the subsets belonging to the family
$$\Sigma\equiv \{\Omega_{ij}\times\ \mbox{closed subsets of} \ Z_t^\nu\},$$
which is closed under countable intersections and finite unions. Therefore, by \cite[Th. 2.2.9]{Bog} the set
$\Omega_{ij}\times B_{ij}^{t,\nu}$ is $\mathcal F^*_{t, \nu}$-measurable for all $1\le i\le m$,
$1\le j \le r$.

Let us pick some $t_i\in T_i$ ($i=1,...m$).
The strong convergence of the sequence $\tau_\nu$ to the identity map in $Z$ implies that
$$
\bigcap\limits_{\nu=1}^\infty (\Omega_{ij}\times \phi_{t_i,\nu}^{-1}(B_{ij}^{t_i,\nu}))=\Omega_{ij}\times B_{ij} \ \ (1\le i\le m, \ 1\le j \le s).
$$
Therefore, for any $\varepsilon>0$, there exists $\nu\in N$ such that
\begin{equation}\label{eq_step 3}
%P^*(\Omega^{*}\triangle(\bigcup\limits_{j=1}^s(\Omega_{ij}\times \phi_{t_i,\nu}^{-1}(B_{ij}^{t_i,\nu}))))<\frac{\varepsilon}{m}, \ \
P^*\left((\Omega_{ij}\times \phi_{t_i,\nu}^{-1}(B_{ij}^{t_i,\nu}))\triangle (\Omega_{ij}\times B_{ij})\right)<\frac{\varepsilon}{2s^2m}
\end{equation}
for all $1\le i\le m$,
$1\le j \le s$. By the Volterra property, $$\phi_{t_i,\nu}^{-1}(B_{ij}^{t_i,\nu})=(q^{ut_i})^{-1}\left(\phi_{u,\nu}^{-1}(B_{ij}^{u,\nu})\right)$$ if $u\le t_i$, $u\in T_i.$ Therefore, the estimates in (\ref{eq_step 3}) hold true for any $u\le t_i$, $u\in T_i$.

As $\mbox{dim}\, q^{t_i}(Z_\nu)<\infty$, there is $u_i\in T_i$, $u_i\le t_i$, for which $\mbox{dim}\, q^{u_i}(Z_\nu)\le \mbox{dim}\, q^t(Z_\nu)$ for all $t\in T_i$. Then $q^{u_iu}$ is the identity map for all $u\le u_i$, $u\in T_i$, so that
$\mathcal F_i\otimes \mbox{Bor}\,(Z_u^\nu)=\mathcal F_{i}\otimes \mbox{Bor}\,(Z_{u_i}^\nu)$ for all $u\in T_i$, $u\le u_i$.

Denote $$
\begin{array}{c}
C_{ij}=B_{ij}^{u_i,\nu}-\left(\bigcup\limits_{k< j}B_{ij}^{u_i,\nu}\right)\in \mathcal F^*_{t,\nu}, \\
  B_{ij}'=\phi_{t,\nu}^{-1}(B_{ij}^{u_i,\nu})\in \mathcal F^*, \ \ \
  C_{ij}'=\phi_{t,\nu}^{-1}(C_{ij})\in \mathcal F^*, \\
  \bar B_{ij}=\Omega_{ij}\times B_{ij}, \ \ \  \bar B_{ij}'=\Omega_{ij}\times B_{ij}', \ \ \ \bar C'_{ij}=\Omega_{ij}\times C_{ij}'
\end{array}
$$
for all $1\le i \le m$, $1\le j \le s$. By definition, $\bar B'_{ij}$ are $\bar C'_{ij}$ are random cylinder sets, for which
$\bigcup\limits_{j=1}^s\bar C'_{ij}=\bigcup\limits_{j=1}^s\bar B'_{ij}$ for all $i=1,...,m$.

We claim that
\begin{equation}\label{eq_Cij}
  P^*\left((\Omega_{ij}\times B_{ij})\triangle (\Omega_{ij}\times C_{ij}')\right)=P^*\left(\bar B_{ij}\triangle \bar C_{ij}'\right)<\frac{\varepsilon}{sm} \ \ \ (1\le i \le m, \ 1\le j \le s).
\end{equation}

Indeed, minding $\bar C_{ij}'\subset \bar B_{ij}'$, $\bar B_{ik}\cap \bar B_{ij}=\emptyset$  we obtain
$$
\begin{array}{l}
P^*\left(\bar B_{ij}'\triangle \bar C_{ij}'\right)= P^*(\bar B_{ij}'-( \bar B_{ij}'-\bigcup\limits_{k<j}\bar B_{ik}'))
=P^*(\bigcup\limits_{k<j}(\bar B_{ik}'\cap \bar B_{ij}'))\le \sum\limits_{k<j}P^*(\bar B_{ik}'\cap \bar B_{ij}')
\\
= \sum\limits_{k<j}P^*((\bar B_{ik}'\cap \bar B_{ij}')\triangle (\bar B_{ik}\cap \bar B_{ij}))\le
 P^*\left(\bar B_{ij}\triangle \bar B_{ij}'\right)+\sum\limits_{k<j}P^*\left(\bar B_{ik}\triangle \bar B_{ik}'\right)  \\
  <  \frac{k\varepsilon}{2s^2m}\le \frac{\varepsilon}{2sm}
\end{array}
$$
by (\ref{eq_step 3}). Therefore,
$$
P^*\left(\bar B_{ij}\triangle \bar C_{ij}'\right)\le P^*\left(\bar B_{ij}\triangle \bar B_{ij}'\right)+P^*\left(\bar B_{ij}'\triangle \bar C_{ij}'\right)<\frac{\varepsilon}{2s^2m}+ \frac{\varepsilon}{2sm}\le \frac{\varepsilon}{sm},
$$
which justifies (\ref{eq_Cij}).

%%%%%%%%%%%
Put now
$$
\begin{array}{c}
\alpha_i''=\sum\limits_{j=1}^s a_{ij}I_{\Omega_{ij}\times C_{ij}'}, \ \ \
%\ \  \beta_i''=\sum\limits_{j=1}^s a_{ij}I_{\Omega_{ij}\times C_{ij}'},
%\\
  \Omega_i^{*,3}=\bigcup\limits_{j=1}^r(\Omega_{ij}\times C'_{ij}).
\end{array}
  $$
By construction,
$$
\alpha_i'' |_{\Omega_i^{*,3}}= 0 \ \ \mbox{and} \ \
P^*(\Omega_i^{*,3}\triangle \Omega_i^{*,2})\le \sum\limits_{j=1}^rP^*\left(\bar B_{ij}\triangle \bar C_{ij}'\right)<\frac{r\varepsilon}{sm}\le \frac{\varepsilon}{m},
$$
where $\Omega_i^{*,2}=\bigcup\limits_{j=1}^r(\Omega_{ij}\times B_{ij})$, as it was defined in (\ref{eq_x_y}).

Observe that
$$d_R^*(\alpha'_i, \alpha''_i)\le \sum\limits_{j=1}^sP^*\left(\bar B_{ij}\triangle \bar C_{ij}'\right)<\frac{s\varepsilon}{sm} < \frac{\varepsilon}{m}.$$
Hence, as in step 2, we obtain that $d^*_E(x,x'')<\varepsilon$, where $(\alpha_1'',...,\alpha_m'').$

In addition, we have
$$
P^*\left(A\triangle \bigcap\limits_{i=1}^m\Omega_i^{*,3}\right)=P^*\left(\left(\bigcap\limits_{i=1}^m\Omega_i^{*,3}\right)\triangle \left(\bigcap\limits_{i=1}^m\Omega_i^{*,3}\right)\right)\le
\sum\limits_{i=1}^m P^*(\Omega_i^{*,3}\triangle \Omega_i^{*,2})< \varepsilon.
$$
As $\varepsilon >0$ is arbitrary, we can again redefine $x$ and $A$ to be
\begin{equation}\label{eq_x_y_redefined_3}
x=x''=(\alpha_1'',...,\alpha_m''), \ \ \  \mbox{and} \ \ \ A=\bigcap\limits_{i=1}^m\Omega_i^{*,3}, \ \ \ \mbox{respectively}.
\end{equation}

The great advantage of (\ref{eq_x_y_redefined_3}) compared with (\ref{eq_x_y_redefined_2}) is that the sets $\bar C'_{ij}=\Omega_{ij}\times \phi_{t,\nu}^{-1}(C_{ij})$ are random cylinder sets for all $u\in T_i$, $u\le u_i$, so that the set $A$ and the random point $x$  can be, without loss of generality, assumed to belong to the $P^*$-completion of the cylinder $\sigma$-algebra $\mathcal F_i\otimes (\phi_{u_i,\nu})^{-1}(\mbox{Bor}(Z^\nu_{u_i}))$ and be measurable with respect to this $\sigma$-algebra, respectively.

 This enables us to apply Lemma \ref{lem-appr-by-continuity-sets}, which is done in the final step of the proof.

\textit{Step 4. Final approximation of $x$ and $A$.}

According to Step 3, we may assume that
$$
x=(\alpha_1,...,\alpha_m),
$$
where
\begin{equation}\label{eq_step4}
%\begin{array}{l}
\alpha_i=\sum\limits_{j=1}^s a_{ij}I_{\Omega_{ij}\times (\phi_{u_i,\nu})^{-1}(C_{ij})}, \ \
  %\beta_i=\sum\limits_{j=1}^s b_{ij}I_{\Omega_{ij}\times (\phi_{u_i,\nu})^{-1}(C_{ij})},
%\\
  \alpha_i|_{\Omega_i^{*}}=0,  \ \ \Omega_i^{*}\in \mathcal F_i\otimes
   (\phi_{u_i,\nu})^{-1}(\mbox{Bor}(Z^\nu_{u_i}))
 %  P^*(\Omega^{*}\triangle(\bigcup\limits_{j=1}^s(\Omega_{ij}\times C'_{ij})))<\frac{\varepsilon}{m}
%\end{array}
\end{equation}
for some $\nu\in N$ and all $1\le i\le m$, so that, in particular, $x|_A=0$, where $A=\bigcap\limits_{i=1}^m\Omega_i^*$.
Moreover,  $u_i\in T_i$ can be chosen in such a way that $\mathcal F_i\otimes (\phi_{u_i,\nu})^{-1}(\mbox{Bor}(Z^\nu_{u_i}))=\mathcal F_i\otimes (\phi_{u,\nu})^{-1}(\mbox{Bor}(Z^\nu_{u}))$ for all $u\le u_i$, $u\in T_i$.

Consider the probability spaces $\left(\Omega\times Z^\nu_{u_i}, \mathcal F_i\otimes \mbox{Bor}(Z^\nu_{u_i}), P^*_{u_i,\nu}\right)$, the $\mathcal F_i\otimes \mbox{Bor}(Z^\nu_{u_i})$-measurable random variable
$$\hat{\alpha}_i=\sum\limits_{j=1}^s a_{ij}I_{\Omega_{ij}\times C_{ij}}, \ \ \ \hat x=(\hat\alpha_1,...,\hat\alpha_m)
%\ \
%  \hat{\beta}_i=\sum\limits_{j=1}^s a_{ij}I_{\Omega_{ij}\times C_{ij}}
$$
and the $\mathcal F_i\otimes \mbox{Bor}(Z^\nu_{u_i})$-measurable set $\hat\Omega_i=\bigcup\limits_{j=1}^r (\Omega_{ij}\times C_{ij}).$ By construction, $\alpha_i=\hat\alpha_i\circ (\phi_{u_i,\nu})^{-1}$ and $(\mbox{id}\times \phi_{u_i,\nu})^{-1}(\hat\Omega_i)=\Omega_i^*$, so that $\hat\alpha_i|_{\hat\Omega_i}=0.$

We pick an arbitrary $\varepsilon >0$ and apply Lemma \ref{lem-appr-by-continuity-sets} to the disjunct sets $A_{ij}\times C_{ij}$ and $(\Omega\times Z^\nu_{u_i})-\bigcup\limits_{j=1}^sA_{ij}\times C_{ij}$. By this, we arrive at disjunct $P^*_{u_i,\nu}$-continuity sets $\hat\Omega_{ij}^{c}\in \mathcal F_i\otimes \mbox{Bor}(Z^\nu_{u_i})$, for which $$
%P^*_{u_i,\nu}((\Omega\times Z^\nu_{u_i})-\bigcup\limits_{j=1}^{s}\hat\Omega_{ij}^c)<\frac{\varepsilon}{m} \ \  \mbox{and} \ \
\sum\limits_{j=1}^s P^*_{u_i,\nu}((A_{ij}\times C_{ij})\triangle \hat\Omega_{ij}^c)<\frac{\varepsilon}{m} \ \ \mbox{for all} \ i=1,...,m.$$
Therefore,
%$$
\begin{equation}\label{eq_step4_1}
P^*_{u_i,\nu}(A_i^{c}\triangle \hat\Omega_i)\le \sum\limits_{j=1}^s P^*_{u_i,\nu}((A_{ij}\times C_{ij})\triangle \hat\Omega_{ij}^c)<\frac{\varepsilon}{m}, \ \  \mbox{where} \ \ A_i^{c}=\bigcup\limits_{j=1}^{r}\Omega_{ij}^c.
\end{equation}
%$$
Define $\alpha^c_i=a_{ij}$ on the disjoint sets $\hat\Omega_{ij}^n$, $i=1,...,m$. By construction, $\alpha^c_i$ are $\mathcal F_i\otimes \mbox{Bor}(Z^\nu_{u_i})$-measurable random variables and $\alpha_i^c|_{A_i^c}=0$, $i=1,...,m$. In addition,
$$
P^*_{u_i,\nu}\{\alpha^c_i\ne \hat\alpha_i\}\le P^*_{u_i,\nu}(A_i^{c}\triangle \hat\Omega_i)<\frac{\varepsilon}{m} \ \ (i=1,...,m),
$$
so that
\begin{equation}\label{eq_step4_2}
\begin{array}{l}
d_R^\nu(\hat\alpha_i, \alpha_i^c)=E^\nu\{\min |\alpha_i-\alpha_i^c|; 1\}\le P^*_{u_i,\nu}\{\alpha^c_i\ne \alpha_i\}<\frac{\varepsilon}{m}, \\
  d_E^\nu(\hat x, \hat x^c)=E^\nu\{\min \|\alpha_i-\alpha_i^c\|_E; 1\}<{\varepsilon},
\end{array}
\end{equation}
where $E^\nu$ is the expectation associated with the probability $P^*_{u_i,\nu}$ and $\hat x^c=(\alpha_1^c,...,\alpha_m^c)$.

%%%%%%%%
In Step 3 we proved that $\mbox{Bor}(Z^\nu_{u_i})=\mbox{Bor}(Z^\nu_{u})$ for all $u\in T_i$, $u\le u_i$. Hence the random variables
$\alpha^c_i\circ \tau_\nu^{u_i}$ are measurable with respect to $\bigcap\limits_{u\le u_i, u\in T_i} \mathcal F_i\otimes \mbox{Bor} (Z_u)=\mathcal F^*_i$, which means that
$$y\equiv\sum\limits_{i=1}^m(\alpha_i^c\circ \phi_{u_i,\nu}) e_i\in \mathcal Sa(\mathcal E, \mathcal B^{*}).$$  %$y_n\equiv\sum\limits_{i=1}^m(\beta_i^n\circ \phi_{u_i,\nu}) e_i\in \mathcal Sa(\mathcal E, \mathcal B^{*})$.

From the definitions of the measure $P^*_{u_i,\nu}$, the random point $y$ and the estimates (\ref{eq_step4_1})-(\ref{eq_step4_2}) we obtain

\begin{equation}\label{eq_approx_auxiliary}
d_E^*(x, y)=d_E^\nu(\hat x, \hat y)<\varepsilon, \ \ \ y|_{B}=0,
\end{equation}
where
$$
 B=\bigcap\limits_{i=1}^m(\mbox{id}\times \Phi_{u_i,\nu})^{-1}(A_i^c)\in \mathcal F^*
$$
is a $P^*$-continuity set.

Finally,
$$
\begin{array}{l}
P^*(A\triangle B)= P^*\left(\bigcap\limits_{i=1}^m\Omega_i^*\triangle B\right)
= P^*\left(\left(\bigcap\limits_{i=1}^m (\mbox{id}\times\Phi_{u_i,\nu})^{-1}(\hat\Omega_i)\right)\triangle\left(\bigcap\limits_{i=1}^m(\mbox{id}\times\Phi_{u_i,\nu})^{-1}(A_i^c)\right)\right) \\
  = P^*_{u_i,\nu}\left(\bigcap\limits_{i=1}^m \hat\Omega_i\triangle\bigcap\limits_{i=1}^m A^c_i\right)\le
  \sum\limits_{i=1}^mP^*_{u_i,\nu}( \hat\Omega_i\triangle A_i^{c})< \frac{m\varepsilon}{m}=\varepsilon.
\end{array}
$$
Now we return to the projective system $\mathcal X^\nu$, which in step 1 was replaced by $\mathcal Sa(\mathcal E, \mathcal B^{*})$. We see that $y\in \mathcal Sa(\mathcal X^\nu, \mathcal B^{*})\subset \mathcal Sa(\mathcal X, \mathcal B^{*}).$ From (\ref{eq_approx_auxiliary}) we obtain
$d_X^*(x, y)<\varepsilon$ and $y|_{B}=0$, where $B$ is a $P^*$-continuity set satisfying $P^*(A\triangle B)<\varepsilon$.
The proof of the simplified version of Theorem \ref{th-approx-simple-points} is complete.

\textit{Step 5. Proof of the full version of Theorem \ref{th-approx-simple-points}.}

 Let $x, y\in \mathcal Pa(\mathcal X, \mathcal B^{*})$ and $x|_A=y|_A$ $P^*$-a.s. for some $A\in \mathcal F^*$. From the already proven simplified version we can deduce density of $\mathcal Sa(\mathcal X, \mathcal B^{*})$ in the space $\mathcal Pa(\mathcal X, \mathcal B^{*})$ by simply putting $A=\emptyset$. Pick any sequence $x_n\in \mathcal Sa(\mathcal X, \mathcal B^{*})$, $x_n\to x$ in probability $P^*$ ($n\to \infty$) and find, using Lemma \ref{lem-appr-by-continuity-sets}, a sequence of $P^*$-continuity sets $A_n\in\mathcal F^*$ such that $P^*(A\triangle A_n)\to 0$ ($n\to \infty$). Applying the simplified version of Theorem \ref{th-approx-simple-points} to $z=x-y$, $z|_A=0$ $P^*$-a.s. we find a sequence $z_n\in \mathcal Sa(\mathcal X, \mathcal B^{*})$ such that $z_n|_A=0$ $P^*$-a.s. and $z_n\to z$ in probability $P^*$ ($n\to \infty$). Put $y_n=x_n+z_n\in \mathcal Sa(\mathcal X, \mathcal B^{*})$. Clearly, $\{x_n\}$, $\{y_n\}$ and $\{A_n\}$  satisfy conditions (1)-(3) of Theorem \ref{th-approx-simple-points}. Theorem \ref{th-approx-simple-points} is proven.

\begin{remark}\label{rem-approximation-simple-points}
From steps 1-4 it follows that if the values of $x$ a.s. belong to some compact $Q_0\subset X$, then the values of its approximation $y$ may be chosen to belong to the closed convex hull $Q$ of the precompact set $\bigcup\limits_{\nu=1}^\infty\pi_\nu(Q_0)$. This remark will be used in the proof of Theorem \ref{th-extension-operators-general}.
\end{remark}

\subsection{Proof of Theorem \ref{th-extension-operators-general}}\label{sec-extension-LC}
\hfill \\
  We will use the simplified notation from Subsection \ref{sec-LC-extensions} in the proof.

  We start with constructing an extension of $h$ to the subspace $\mathcal{S}a(\mathcal X, \mathcal B^*)$ of the space $\mathcal Pa^*(X)$. By Definition \ref{def-cont-adapted-points}, any $x$ from this space can be written as   $x=\sum\limits_{i=1}^s c_iI_{A_i}$ for some $c_i\in \mathcal Pa(X)$ and disjunct subsets $A_i\in \mathcal F^*$ ($i=1,..,s$), $\bigcup\limits_{i=1}^sA_i=\Omega^*$. Define
  \begin{equation}\label{eq-def-h-0}
  h^0x=\sum\limits_{i=1}^sh(c_i)I_{A_i}
  \end{equation}
  and consider another element $y\in \mathcal{S}a(\mathcal X, \mathcal B^*)$ coinciding with $x$ on some subset $C\subset \Omega^*$. Then $y$ can be represented as  $y=\sum\limits_{k=1}^\sigma d_kI_{B_k}$ for some $d_k\in \mathcal Pa(X)$ and disjunct subsets $B_k\in \mathcal F^*$ ($k=1,..,\sigma$), $\bigcup\limits_{k=1}^\sigma B_k=\Omega^*$. By assumption,
   $x=c_i=d_k$ $P$-a.s. on each subset $A_i\cap B_k\cap C$, so that $h(c_i)=h(d_k)$ $P$-a.s. on $A_i\cap B_k\cap C$ by locality of $h$. Then
      \begin{equation}\label{eq-extension-auxil}
   \begin{array}{c}
   h^0x|_C=\sum\limits_{i=1}^sh(c_i)I_{A_i\cap C}=\sum\limits_{i=1}^s \sum\limits_{k=1}^\sigma h(c_i)I_{A_i\cap B_k\cap C}
 \\
= \sum\limits_{i=1}^s \sum\limits_{k=1}^\sigma h(d_k)I_{A_i\cap B_k\cap C}=\sum\limits_{k=1}^\sigma h(d_k)I_{B_k\cap C} \ \ P^*-\mbox{a.s.}
   \end{array}
   \end{equation}
        If $C=\Omega^*$, i.e. if $x=y$ $P^*$-a.s., then equality (\ref{eq-extension-auxil}) means that definition (\ref{eq-def-h-0}) is  up to a set $P^*$-zero measure independent of the alternative representation of $x$. If $C$ is an arbitrary subset of $\Omega^*$, then (\ref{eq-extension-auxil}) proves locality og $h^0$ on its domain.

     Next we prove uniform continuity of $h^0$ on tight subsets of the set $\mathcal{S}a(\mathcal X, \mathcal B^*)$.
    For this purpose, we fix a sequence $\alpha_\nu\in \mathcal{P}a(\mathcal Z, \mathcal B)$ ($\nu\in N$) such that  the disintegration $P^*_\omega$ of the measure $P^*$ is the limit of the sequence of the random Dirac measures $\{\delta_{\alpha_\nu}\}$ in the narrow topology and define the auxiliary probability spaces and the stochastic bases by
   \begin{equation}\label{eq-appr-stoch-bases}
     \mathcal S_\nu = (\Omega^*, \mathcal F^*, P\alpha_\nu^{-1}) \ \ \ \mbox{and} \ \ \ \mathcal B_\nu = (\Omega^*, \mathcal F^*, (\mathcal F^*_t)_{t\in T}, P\alpha_\nu^{-1}).
   \end{equation}

   As it was shown in Remark \ref{ex-extention-Dirac-measure}, for every $\nu\in N$ there exists an LC operator $h^\nu: \mathcal Pa (\mathcal X, \mathcal B_\nu)\to \mathcal P (Y, \mathcal S_\nu)$, which extends the operator $h$. By Theorem \ref{prop-unique-local-extension}, this extension is unique, so that by the construction from Remark \ref{ex-extention-Dirac-measure}
   \begin{equation}\label{eq-def-h-nu}
   h^\nu x=h(x\circ\alpha_\nu)=h\sum\limits_{i=1}^sh(c_i)I_{A_i} \ \ \ P\alpha_\nu^{-1}-\mbox{a.s.}
   \end{equation}
      for any $
   x=\sum\limits_{i=1}^sc_iI_{A_i}$, where $c_i\in \mathcal Pa(X)$ and $A_i\in \mathcal F^*$ ($i=1,..,s$) are disjunct sets with the property $\bigcup\limits_{i=1}^sA_i=\Omega^*$. This means, in particular, that $h^0x$ defined in (\ref{eq-def-h-0}) is $P\alpha_\nu^{-1}$-equivalent to $h^\nu x$.

   Pick an arbitrary tight subset $\mathcal K\subset \mathcal Sa(\mathcal X, \mathcal B^*)$ and arbitrary $\varepsilon >0$. Then there is a compact $Q_0\subset X$ such that
   $$
   P^*\{x\notin Q_0\}<\varepsilon \ \ \ \mbox{for all} \ \ \ x\in \mathcal K.
   $$
   The closed convex hull $Q$ of the set $\bigcup\limits_{n=1}^\infty \pi_n(Q_0)$ is compact as well.
As  $h$ is uniformly continuous on the tight set $\mathcal Pa(Q)\equiv \mathcal Pa(X)\cap \mathcal P(Q,\mathcal B)$,
we can find $\rho>0$ and $\delta>0$, $\delta<\varepsilon $ such that
$$
\|u-v\|_X\le \rho \ \ \ \mbox{implies} \ \ \ d_Y(hu,hv)<\varepsilon \ \ \forall u,v\in \mathcal Pa(Q),
$$
where $d_Y$ is the metric on $\mathcal P(Y).$
Choose arbitrary $x,y\in\mathcal K$ satisfying $P^*\{ \|x-y\|_X\ge \frac \rho 3\}<\frac \delta 3$ and find a sufficiently large $n_0\in N$ such that
$$
P^*\{\|\pi_{n}x-x\|_X\ge \frac\rho 3\}<\frac \delta 3 \ \ \ \mbox{and} \ \ \ P^*\{\|\pi_{n}y-y\|_X\ge \frac\rho 3\}<\frac \delta 3 \ \ \ (\forall n\ge n_0).
$$
Evidently, $P^*\{ \|\pi_{n}x-\pi_{n}y\|_X\ge \rho \}<\delta $ ($n\ge n_0$).

Now we use the construction from the proof of Theorem \ref{th-uniform-cont}, see Subsection \ref{sec-proof of uniform th} and consider the direct product $\mathcal E^n$ of two copies of the finite dimensional projective subsystem $\mathcal X^n$ defined in (\ref{eq-finite-dim-system}), the compact convex subset $W^n= \{(x_1,x_2)\in \pi_n(X)\times \pi_n(X): \ x_1,x_2\in Q, \|x_1-x_2\|_X\le \rho\}$ and the continuous projection $\phi_n: \pi_n(X)\times \pi_n(X)\to W^n$ such that the corresponding superposition operator $h_{\phi_n}$ maps $\mathcal Pa(\mathcal E^n, \mathcal B^*)$ to $\mathcal Pa(\mathcal E^n, \mathcal B^*)\cap \mathcal P(W^n, \mathcal B^*)$. Such a projection exists due to Lemma \ref{lem-projection}. Put $(u,v)=h_{\phi_n}(\pi_{n}x, \pi_{n}y).$ By construction, $\|u-v\|_X\le\rho$ and
\begin{equation}\label{eq-estimate for-u-and-v}
\begin{array}{l}
  P^*\{u\ne \pi_nx \ \& \ v\ne \pi_ny\}= P^*\{(\pi_{n}x, \pi_{n}y)\notin W^n\}\\
  \le P^*\{\pi_{n}x\notin Q\}+P^*\{\pi_{n}y\notin Q\}+P^*\{\|\pi_{n}x-\pi_{n}y\|_X\ge \rho\}\le 2P^*\{x\notin Q_0\}+\delta\\
  < 2\varepsilon +\delta < 3\varepsilon
\ \ \ (n\ge n_0).
\end{array}
\end{equation}
As $x$ and $y$ and hence $\pi_nx$, $\pi_ny$, $u$ and $v$ belong to $\mathcal{S}a(\mathcal X, \mathcal B^*)$, the sets  $\{\pi_nx\ne u\}$ and $\{\pi_ny\ne v\}$ are $P^*$-continuity subsets of $\Omega^*$, so that by Remark \ref{rem-narrow-topology}
$$
\begin{array}{c}
P\{\pi_n(x\circ\alpha_\nu)\ne u\circ\alpha_\nu\}=P\alpha^{-1}_\nu\{\pi_nx\ne u\}<\delta \ \ \ \mbox{and} \\
P\{\pi_n(y\circ\alpha_\nu)\ne v\circ\alpha_\nu\} = P\alpha^{-1}_\nu\{\pi_ny\ne v\}<\delta \ \ \ (\forall \nu\ge\nu_0)
\end{array}
 $$
for sufficiently large $\nu_0$.

For any $\nu$ the random points $u\circ\alpha_\nu$ and $v\circ\alpha_\nu$ belong to $\mathcal Pa(Q)$ and satisfy $\|u\circ\alpha_\nu-v\circ\alpha_\nu\|_X\le\rho$, so that
$$
E_\nu\min\{\|h^\nu u-h^\nu v\|;1\}=
E\min\{\|h(u\circ\alpha_\nu) -h(v\circ\alpha_\nu)\|;1\}
  <\varepsilon   \ \ \ (\nu\ge\nu_0),
$$
where $E_\nu$ is the expectation with respect to the measure $P\alpha_\nu^{-1}.$
Making use of representations (\ref{eq-def-h-0}) and (\ref{eq-def-h-nu}) and Remark \ref{rem-narrow-topology}, we can let $\nu\to\infty$ in the last estimate giving
 $$d^*(h^0 u,h^0 v)\equiv E^*\min\{\|h^0 u-h^0 v\|; 1 \}\le \varepsilon,$$ where
 $E^*$ is the expectation with respect to the measure $P^*.$
 By locality of $h^*$ and estimate (\ref{eq-estimate for-u-and-v}),
 $$
 \begin{array}{c}
 d^*(h^0 (\pi_nx), h^0(\pi_ny))\le d^*(h^0 u, h^0 v) \\
 +P^*\{\pi_nx\ne u \ \& \ \pi_ny\ne v\}
 <\varepsilon +3\varepsilon=4\varepsilon \ \ \ (\forall n\ge n_0).
 \end{array}
  $$
On the other hand, $h^0(\pi_nx)=\sum\limits_{i=1}^sh(\pi_n(c_i))I_{A_i}$ $P^*$-a.s. for some $c_i\in \mathcal Pa(X)$, see
(\ref{eq-def-h-0}). As $\pi_n(c_i))\to c_i$ in the topology of the space $\mathcal Pa(X, \mathcal B^*)$, we obtain $h^0(\pi_nx)\to h^0x$ ($n\to\infty$) in the topology of the space $\mathcal Pa^*(X)$. Similarly, $h^0(\pi_ny)\to h^0y$ ($n\to\infty$) in this topology, so that
$$d^*(h^0 x, h^0y)\le  4\varepsilon,$$
which yields uniform continuity of $h^0$ on the tight set $\mathcal K$.

In the final part of the proof, we use Theorem \ref{th-approx-simple-points}, according to which the set $\mathcal{S}a(\mathcal X, \mathcal B^*)$ is dense in the space $\mathcal{P}a^*(X)$. The operator $h^0: \ \mathcal{S}a(\mathcal X, \mathcal B^*)\to \mathcal P^*(Y)$ is uniformly continuous on tight and, thus, on precompact subsets of $\mathcal{S}a(X, \mathcal B^*)$. Therefore,  the operator $h^0$ admits a unique continuous extension $h^*: \mathcal{P}a^*(X)\to\mathcal P^*(Y)$.

To show locality of $h^*$ we pick $x,y\in\mathcal{P}a^*(X)$, $x|_A=y|_A$ for some $A\in \mathcal F^*$  and find two sequences $x_n, y_n\in \mathcal{S}a(\mathcal X, \mathcal B^*)$, again using Theorem \ref{th-approx-simple-points}, such that $x_n\to x$, $y_n\to y$ in probability and $x_n|_{A_n}=y_n|_{A_n}$ where $P^*(A\triangle A_n)\to 0$ as $n\to\infty.$ Then  $h^*x_n\to h^*x$, $h^*y_n\to h^*y$ and $I_{A_n\cap A}\to I_{A}$ in probability, so that
$(h^*x_n)I_{A_n\cap A}\to (h^*x)I_A$ and $(h^*y_n)I_{A_n\cap A}\to (h^*y)I_A$ in probability
 as $n\to\infty$. Therefore, $hx|_A=hy|_A$ $P^*$-a.s., and the operator $h^*$ is local.

Due to Theorem \ref{th-uniform-cont}, uniform continuity of $h^*$ on arbitrary tight subsets of the space $\mathcal Pa^*(X)$ is equivalent to uniform continuity on any set $\mathcal K_0=\mathcal Pa^*(X)\cap \mathcal P(Q_0, \mathcal S^*)$, where $Q_0\subset X$ is compact. The closed convex hull $Q$ of the set $\bigcup\limits_{n=1}^\infty \pi_n(Q_0)$ is again compact in $X$, so that $h^0$ is uniformly continuous on the set $\mathcal K=\mathcal{S}a(\mathcal X, \mathcal B^*)\cap \mathcal P(Q, \mathcal S^*)$ and hence on  its closure in the topology of the space $\mathcal Pa^*(X)$. On the other hand, this closure contains the set $\mathcal K_0$, because by Theorem \ref{th-approx-simple-points} and Remark
\ref{rem-approximation-simple-points}, for any $x\in \mathcal K_0$ there exist $$x_n\in \mathcal{S}a(\mathcal X, \mathcal B^*)\cap \mathcal P(\pi_n(Q), \mathcal F^*)\subset \mathcal K$$
that converges to $x$. Hence $h^*$ is uniformly continuous on $\mathcal K_0$. The theorem is proven.

\subsection{Proof of Theorem \ref{th-finite-dim}}\label{sec-proof-finite-dim}
\hfill \\
    Using the LC operator $h_\phi$  from Lemma \ref{lem-projection} we can replace $h$ by $h\circ h_\phi$, thus obtaining an LC operator mapping the space $\mathcal Pa (\mathcal X, \mathcal B)$ to the subset $\mathcal Pa (U)$. Moreover, taking advantage of the construction used in the proof of Lemma \ref{lem-isomorphism} we can replace the operator $h$ by the LC operator $h_G^{-1}\circ h \circ h_G$, where  the linear topological isomorphism $h_G: \ \mathcal Pa(\mathcal X, \mathcal B)\to \mathcal Pa(\mathcal E, \mathcal B_m)$ generated by the linear isomorphism $G:  X\equiv X_b\to E_m\equiv E$, $\mathcal E=(E_i, p^{ji}, T(m))$ is a projective system of Euclidean spaces $E_i=\{(x_1,...,x_i, 0, ...,0)\}$, $p^{ji}$ ($i\ge j$) is the orthogonal projection, which removes the coordinates $(x_{j+1},..., x_{i})$,  and  $\mathcal B_m =(\Omega , {\mathcal F}, ({\mathcal F}_i)_{i\in {T}_m}, P)$ is a finite stochastic basis.
Therefore, we can assume that $U\subset E$ and
  $h: \ \mathcal Pa(E)\to \mathcal Pa(U)$, where $\mathcal Pa(E)\equiv \mathcal Pa(\mathcal E, \mathcal B_m).$

    The idea of the proof is to study the spaces $\mathcal Pa(E_i, \mathcal F_i)$ ($i=0,...,m$) by induction applying Lemma \ref{th-Nemytskii} at each step, so that we obtain the statement of the theorem at $i=m$. The problem here is that the operator $h$ is not supposed to be Volterra, so that a priori there exist no "truncated versions" $h_i$ of it defined on $\mathcal Pa(E_i, \mathcal F_i)$. However, we will show that using locality of $h$ gives us opportunity to partially define the operators $h_i$ and thus find a sequence of "partial" fixed points in the subspaces $\mathcal Pa(E_i, \mathcal F_i)$.

  \textit{Step 1: Coincidence set of two $\sigma$-algebras.} Let $\mathcal G\subset \mathcal G'\subset \mathcal F$ be two complete $\sigma$-algebras with respect to the measure $P$. Consider the family $\mathcal O$ of subsets $A\in\mathcal{G}$ for which $\mathcal G\cap A=\mathcal G'\cap A$ and let $\gamma=\sup\limits_{A\in \mathcal O}PA$. Pick any sequences $\gamma_n\to\gamma$ for which there is $A_n\in\mathcal O$ with $PA_n=\gamma_n$ and define $\hat\Omega=\bigcup_{n=1}^\infty A_n$. Clearly, $P\hat\Omega=\gamma$ and $\mathcal G\cap \hat\Omega=\mathcal G'\cap \hat\Omega$. If $B\in\mathcal O$ is arbitrary, then $\hat\Omega\cup B\in\mathcal O$, so that $P(\hat\Omega\cup B)\le\gamma$ and hence $P(B-\hat\Omega)0=$. We have proven that $\hat\Omega$ is the largest (up to a zero measure) set belonging to $\mathcal O$. We will call this set the coincidence set of the  $\sigma$-algebras $\mathcal G\subset \mathcal G'$. By the definition, if $C\in \mathcal G'-\mathcal G$ and $PC>0$, then there is a subset $C\in \mathcal G'$ such that $PC'>0$ and $\hat\Omega\cap C'=\emptyset$. For the pair $\mathcal F_i$ and $\mathcal F_{i+1}$ from the above finite filtration we fix $\hat\Omega_i\in \mathcal F_i$ ($i=0,...,m$) to be one of the realisations of the coincidence sets  of the $\sigma$-algebras $\mathcal F_i$ and $\mathcal F_{i+1}$. We also assume by definition that $\hat\Omega_m=\emptyset$, which is formally possible to achieve if we define $\mathcal F_{m+1}=\mathcal F_m\otimes \mbox{Bor}\,[0,1]$ and equip $\mbox{Bor}\,[0,1]$ with the Lebesgue measure.

  \textit{Step 2: Construction of auxiliary (truncated) local operators.} Let us first introduce truncated spaces of adapted random points. For any $i=0,...,m$ we put $\Omega_i=\Omega-\hat\Omega_i\in \mathcal F_i$ and let $\mathcal P_i$ consist of all $x: \Omega_i \to E_i$ for which there exists  $\tilde x\in\mathcal Pa(E)$ such that $x= p^i(\tilde x)|_{\Omega_i}$ where $p^i\equiv p^{im}$. Below we show that the operator $h$ induces LC operators $h_i: \ \mathcal P_i\to \mathcal P_i$ by the formula
  \begin{equation}\label{eq-induced-local-op}
    h_ix=(p^i\circ h)(\tilde x)|_{\Omega_i} \ \ \mbox{where} \ \ x= p^i(\tilde x)|_{\Omega_i} \ \ \mbox{and} \ \ x\in\mathcal Pa(E).
  \end{equation}
  Let us check the following property:
\begin{equation}\label{eq-second finte dim -1}
    p^i(\tilde x)|_A=p^i(\tilde y)|_A, \ (\tilde x, \tilde y\in \mathcal Pa(E), \ A\in\mathcal F_{i}, \ A\subset\Omega_i) \ \ \mbox{implies} \ \ p^i\circ h(\tilde x)|_{A}=p^i\circ h(\tilde y)|_{A} \ \mbox{a.s.}
\end{equation}

   Assume, on the contrary, that the last equality is not fulfilled on a subset of $A$ of a positive measure. By the definition of the set $\Omega_i$ as the complement of the coincidence set $\hat\Omega_i$, there exists a set $B\subset A$, $B\in \mathcal F_{i+1}-\mathcal F_{i}$, of a positive measure such that
   \begin{equation}\label{eq-second finte dim -2}
     p^i(\tilde x(\omega))= p^i(\tilde y(\omega)) \ \ \mbox{and} \ \ (p^i\circ h)(\tilde x(\omega))\ne (p^i\circ y)(\tilde y(\omega)) \ \ \ (\omega\in B).
   \end{equation}
Put $\tilde z =\tilde x$ on $B\in \mathcal F_i$ and $\tilde z=\tilde y$ on $\Omega - B$. We claim that $\tilde z \in \mathcal Pa(E_m)$. Indeed, if $k>i$, then $p^k(\tilde z)=p^k(\tilde x)$ on $B\in \mathcal F_i$ and $p^k(\tilde z)=p^k(\tilde y)$ on $\Omega - B$. Hence  $p^k(\tilde z)$ is $\mathcal F_k$-measurable for $k>i$. If $j\le i$, then $p^j(\tilde z)=p^{ji}(p^i(\tilde y))$ on $\Omega-B$ and $p^j(\tilde z)=p^{ji}(p^i(\tilde x))=p^{ji}(p^i(\tilde y))$ on $B$, so that $p^j(\tilde z)= p^{ji}(p^i(\tilde y))$ on $\Omega$ and therefore $p^j(\tilde z)$ is $\mathcal F_j$-measurable for $j\le i$. Thus, $\tilde z \in \mathcal Pa(E)$.

By locality of $h$, $(p^i\circ h)(\tilde z)=(p^i\circ h)(\tilde x)$ a.s. on $B$ and  $(p^i\circ h)(\tilde z)=(p^i\circ h)(\tilde y)$ a.s. on $\Omega-B$. Therefore,  $(p^i\circ h)(\tilde z)\ne (p^i\circ h)(\tilde x)$ a.s. on $A-B$
by (\ref{eq-second finte dim -2}). Hence
$$
\begin{array}{c}
D\equiv \{(\omega, u): \ \omega\in B, \ y=(p^i\circ h)(\tilde x(\omega))\}
\\
= \{(\omega, u): \, \omega\in A, \ y=(p^i\circ h)(\tilde x(\omega))\}\!\cap\! \{(\omega, u): \, \omega\in A, \, y=(p^i\circ h)(\tilde z(\omega))\}\in \mathcal F_i\otimes \mbox{Bor}\, (E_i) \, \mbox{a.s}.
\end{array}
$$
By the well-known measurable projection theorem \cite{Wagner},  $B=\{\omega\in\Omega: \exists u\in E_i \, | \, (\omega, u)\in D\}\in \mathcal F_i.$ But this contradicts the assumption $B\in \mathcal F_{i+1}-\mathcal F_{i}$. We have proven
(\ref{eq-second finte dim -1}).

This property with $A=\Omega_i$ guarantees that the following operator is well-defined on the set $\mathcal P_i$:
$$x\in \mathcal P_i \ \  \Rightarrow \ \ h_ix\equiv p^i\circ h(\tilde x)|_{\Omega_i},\ \ \mbox{where} \ \ x= p^i(\tilde x)|_{\Omega_i}, \tilde x\in\mathcal Pa(E).$$
The same property with an arbitrary $A\subset \Omega_i$ yields locality of the operator $h_i$ on $\mathcal P_i$.
Continuity of $h_i$ follows from the fact that $\mathcal P_i$ is topologically embedded in $\mathcal{P}(E, \mathcal S)$ via the map $x=(\eta_1,...,\eta_i)\mapsto (\eta_1,...,\eta_i, 0,...,0)I_{\Omega_k}$.
Note that by construction, $\hat\Omega_m=\emptyset$, so that $\Omega_m=\Omega$ and $h_m=h.$

From  (\ref{eq-induced-local-op}) and (\ref{eq-second finte dim -1}) we also have the following property:
\begin{equation}\label{eq-second finte dim -3}
  h_j(x_j)(\omega)=p^{ji}h_i(x_i)(\omega) \ \mbox{a.s. on} \ \Omega_i\cap\Omega_j \ \ \mbox{if} \ \ \ x_j=p^{ji}x_i \ \ (i\ge j).
\end{equation}

\textit{Step 3. Construction of partial fixed points.}

The statement to be proven by induction:

\textit{For any $i$ ($0\le i\le m$) there exist $\tilde x_j\in\mathcal Pa(E) $ ($0\le j\le i$) such that $x_{ij}\equiv p^j(\tilde x_i)|_{\Omega_j}$ is a fixed point of the operator $h_j$ for all $0\le j\le i$.}

The statement is trivial for $i=0$ and it is equivalent to  Theorem \ref{th-finite-dim} if $i=m$. Assume that it is true for some $0\le i<m$ and check that it is also true for $i+1$.

Define the sets
\begin{equation}\label{eq-second finte dim -4}
W_j(\omega)=\left\{ \begin{array}{c}
                    p^{i+1}(V)  \ \ \mbox{if} \ \ \omega\in \hat\Omega_j=\Omega-\Omega_j\\
                    (p^{i+1,j})^{-1}(x_{ij}(\omega))\cap  p^{i+1}(V) \ \ \mbox{if} \ \ \omega\in \Omega_j
                  \end{array}
\right.
\end{equation}
and put $U_{i+1}(\omega)=\bigcap_{j=0}^i W_j(\omega).$ By construction, $\{(\omega, W_j(\omega)) \ : \ \omega\in\Omega\}\in \mathcal F_{i+1}\otimes \mbox{Bor}\, (E_{i+1})$.
For the set $\mathcal P(U_{i+1})$ consisting of all $\mathcal F_{i+1}$-measurable random points $\omega\mapsto S_{i+1}(\omega)$ we check that
$\mathcal P(S_{i+1})= \mathcal Pa(E_{i+1}, \mathcal B_{i+1})$, where $\mathcal B_{i+1}\equiv (\Omega, (\mathcal F_j)_{0\le j\le i+1}, \mathcal F, P)$  is the truncated stochastic basis.

Pick any $z\in \mathcal P(S_{i+1})$ and $j\le k \le i$. Then for any $\omega\in\Omega_k$, we have $$p^{k,i+1}(z(\omega))=x_{ik}(\omega)=p^k(\tilde x_i(\omega)),$$ and therefore
$$
p^{j,i+1}(z(\omega))=(p^{jk}\circ p^{k,i+1})(z(\omega))=(p^{jk}\circ p^k)(\tilde x_i(\omega))=p^j(\tilde x_i(\omega))
$$
($\omega\in \mathcal O_j\equiv \bigcup\limits_{j\le k \le i}\Omega_k$). The set  $\mathcal O_j$ belongs to $\mathcal F_j$, because the $\sigma$-algebras $\mathcal F_k$ ($j\le k \le i+1$) coincide, by construction, on its complement $\Omega-\mathcal O_j=\bigcap\limits_{j\le k \le i}\hat\Omega_k$. Therefore, $p^{j,i+1}(z)|_{\mathcal O_j}=p^j(\tilde x_i)|_{\mathcal O_j}$ is $\mathcal F_j$-measurable for any $j\le i$. On the other hand, $p^{j,i+1}(z)|_{\Omega-\mathcal O_j}$ is $\mathcal F_j$-measurable as well: as it was already mentioned, $\mathcal F_{i+1}=\mathcal F_i=...= \mathcal F_j$ on this subset. We have proven that $z\in Pa(E_{i+1}, \mathcal B_{i+1})$.

From (\ref{eq-second finte dim -3}) it follows that the LC operator
$$
(Hz)(\omega)=\left\{
\begin{array}{c}
z(\omega) \ \ \mbox{if} \ \ \omega\in \hat\Omega_{i+1}
  \\
 (h_{i+1}z)(\omega)  \ \ \mbox{if} \ \ \omega\in \Omega_{i+1}
\end{array}
\right.
$$
leaves the subset $\mathcal P(S_{i+1})$ invariant.

Finally, we observe that the operator $H$ satisfies the assumption of Lemma \ref{th-Nemytskii}, so that there exists $\tilde x_{i+1}\in \mathcal P(S_{i+1})= \mathcal Pa(E_{i+1}, \mathcal B_{i+1})\subset \mathcal Pa (E)$, for which $H(x_{i+1})=x_{i+1}$ a.s. By construction, $x_{i+1}|_{\Omega_{i+1}}$ is a fixed point of the operator $h_{i+1}$. On the other hand, due to (\ref{eq-second finte dim -4}) we have that $$p^j(\tilde x_{i+1})|_{\Omega_j}=(p^{j,i+1}\circ p^{i+1})(\tilde x_{i+1})|_{\Omega_j}|=x_{ij}|_{\Omega_j},$$
which by assumption is a fixed point of the operator $h_j$. The induction argument, and hence the proof of Theorem \ref{th-finite-dim}, is complete.

\section{Some additional properties of local operators}\label{sec-Local-Op}

%In this section we justify some general properties of local operators defined in the spaces of adapted random points.

Unlike the results of Appendix B, the propositions collected in Appendix C are only used in the examples of Appendix D and not in the proof of the fixed-point theorems for LC operators.

Below it is assumed that $\mathcal S$ is a complete probability space (\ref{eq-probability-space}), $\mathcal B$ is a stochastic basis (\ref{eq-stoch-basis}) on it,
$\mathcal X=(X_t, p^{ut}, {T})$ is a projective system of separable Banach spaces satisfying Property $(\Pi)$, the associated finite dimensional linear Volterra operators being $\pi _n$.
 %and $h:\ {\mathcal P}a(\mathcal X, \mathcal B)\to {\mathcal P} (Y, \mathcal S)$ is a local operator ($Y$ is another separable Banach space).

\begin{proposition}\label{prop-tight-operators}
A local operator $h:\ {\mathcal P}a(\mathcal X, \mathcal B)\to {\mathcal P}(Y, \mathcal S)$ is tight if it is tight on any subset ${\mathcal P}a(\mathcal X, \mathcal B)\cap {\mathcal P}(B_r, \mathcal S)$ where $B_r=\{x\in X: \ \Vert x\Vert_X\le r\}$.
\end{proposition}
\begin{proof}
  Due to Proposition \ref{th-uniform-cont} we only have to prove that the set $h({\mathcal M})$ is tight for any bounded  ${\mathcal M}\subset {\mathcal P}%
a(X)$.
Taking arbitrary $\varepsilon ,\sigma >0$ we can find $r>0$ such that the inequality $P\{x\notin B_r\}<%
{\varepsilon }$ holds for all $x\in {\mathcal M}$. There exists $n$ (depending on $x$) such that
\[
\mbox{\bf P}\Vert h(\pi _nx)-hx\Vert_Y \ge \sigma \}<\varepsilon.
\]
By Lemma \ref{lem-projection}, there exists $x_r\in \mathcal Pa(\mathcal X, \mathcal B)\cap \mathcal P(\pi_n(B_r), \mathcal S)$ with the property $(\pi_nx)(\omega)=x_r(\omega)$ if $(\pi_nx)(\omega)\in \pi_n(B_r)$. Therefore,
$P\{x_r\ne \pi _nx\}<\varepsilon $.
This property and locality
of $h$ yield
$
P\{hx_r\ne h(\pi _nx)\}<\varepsilon .
$
The strong convergence of  $\{\pi_n\}$ to the identity map in $X$ implies boundedness of the set $D_r=\cup_{n\in N}\pi_n(B_r)$.
By assumption, $h$ maps the set ${\mathcal P}a(\mathcal X, \mathcal B)\cap {\mathcal P}(D_r, \mathcal S)$ into a tight subset of ${\mathcal P}a(\mathcal X, \mathcal B)$. Therefore, there exists a compact $G\subset X$ for which
\[
P\{hy\notin G\}<\varepsilon \ \ \mbox{for
all}\ \ y\in {\mathcal P}a(\mathcal X, \mathcal B)\cap {\mathcal P}(D_r, \mathcal S).
\]
In particular, this is satisfied for $y=x_r$ and, denoting
the $\sigma$-neighborhood of $G$ by $G_\sigma$, we get
\[
P\{hx\notin G_\sigma \}\le P\{hx_r\notin G\}+
P\{\Vert h(\pi _nx)-hx_r\Vert_Y \ge \sigma \}+\mbox{ P}\{hx_r\ne
h(\pi _nx)\}<3\varepsilon .
\]
This property and Remark \ref{rem_def_tight_sets} yield tightness of the set $h{\mathcal M}$.
\end{proof}

\begin{proposition}
\label{prop-sequence-tight-op} Suppose that the sequence of local and tight operators $h_n: {\mathcal P}a(\mathcal X, \mathcal B)\to
{\mathcal P}(Y, \mathcal S)$ ($n\in N$) converges to an operator $h: {\mathcal P}a(\mathcal X, \mathcal B)\to
{\mathcal P}(Y, \mathcal S)$ uniformly on any subset ${\mathcal P}a(\mathcal X, \mathcal B)\cap {\mathcal P}(B_r, \mathcal S)$ as $n\to\infty$. Then
$h$ is local and tight as well.
\end{proposition}
\begin{proof}
If $x, y \in {\mathcal P}a(\mathcal X, \mathcal B)$ and $xI_A=yI_A$ for some $A\in \mathcal F$, then $(h_nx)I_A=(h_ny)I_A$ for all $n\in N$, as all $h_n$ are local. Therefore, $(hx)I_A=(hy)I_A$, because $\{(h_nx)I_A\}$ and $\{(h_ny)I_A\}$ converge in probability to $(hx)I_A$ and $(hy)I_A$, respectively. Hence $h$ is local.

 The operators $h_n$  are tight and hence uniformly continuous on any subset $\mathcal C_r\equiv {\mathcal P}a(\mathcal X, \mathcal B)\cap {\mathcal P}(B_r, \mathcal S)$ (which is tight). Then so is the operator $h$, as the sequence $\{h_n\}$ converge uniformly to $h$ uniformly on $\mathcal C_r$. Applying Proposition \ref{th-uniform-cont} yields uniform continuity  of $h$ on an arbitrary tight subset of its domain.

It remains to prove that $h$ maps bounded subsets of ${\mathcal P}a(\mathcal X, \mathcal B)$ into tight subsets of ${\mathcal P}(Y, \mathcal S)$. According to Proposition \ref{prop-tight-operators} it is sufficient to check that $h(\mathcal C_r)$ is tight for all $r>0$. Using again uniform convergence of $\{h_n\}$ on $\mathcal C_r$, we find, for any
$\varepsilon>0$ and $\sigma>0$, a number $m\in N$ such that
${P}\{\|hx - h_mx\|_Y\ge\sigma\}<\varepsilon $ whenever $x\in
\mathcal C_r$. As $h_m$ is tight, there exists a compact subset $K\subset Y$ such that
${P}\{h_nx \notin K\}
<\varepsilon $ for all $x\in \mathcal C_r$. Therefore,
${P}\{hx\notin K_{\sigma}\}<2\varepsilon, $
where $K_\sigma$ is the $\sigma$-neighborhood of $K$. As $\varepsilon>0$ and $\sigma>0$ were arbitrary, the set $h(\mathcal C_r)$ is tight by Remark \ref{rem_def_tight_sets}.
\end{proof}

\begin{proposition}\label{prop-transform-tight-sets}
If a local operator $h:\ {\mathcal P}a(\mathcal X, \mathcal B)\to {\mathcal P}(Y, \mathcal S)$ is uniformly
continuous on any tight subset of its domain, then $h$ maps tight
sets into tight ones.
\end{proposition}
\begin{proof}

Step 1. Assume first that $X$ is finite dimensional and prove that $h(\mathcal K_0)$ is tight for any $\mathcal K_0={\mathcal P}a(\mathcal X, \mathcal B)\cap \mathcal P(Q, \mathcal S)$, where $Q\subset X$ is compact. Observe that by Lemma  \ref{lem-isomorphism}, any linear bijection $G: \ X\to E$ induces the linear isomorphism $h_G$ between the spaces
${\mathcal P}a(\mathcal X, \mathcal B)$ and ${\mathcal P}a(\mathcal E, \mathcal B_m)$, where
$\mathcal E=(E_i, p^{ji}, T_m)$, $E_i=\{(x_1,...,x_i, 0, ...,0)\}$, $E=E_m$, $\mathcal B_m$ is a finite stochastic basis with a filtration $(\mathcal F_i)_{i\in T_m}$ and $p^{ji}$ are the orthogonal projections, which remove the coordinates $(x_{j+1},..., x_{i})$. Observe that $x=(x_1,...x_m)\in \mathcal Pa(\mathcal E, \mathcal B_m)$ if and only if $x_i$ is a $\mathcal F_i)$-measurable random variable.

From now on we replace $\mathcal X$ with $\mathcal E$, so that $Q\subset E$. We write $\mathcal Pa(H)$ for $\mathcal Pa(\mathcal E, \mathcal B')\cap \mathcal P(H, \mathcal S)$ if $H\subset E$.

Choose a sufficiently large $m$-dimensional cube $\Pi_m=\{(x_1,....x_m)\in E=-r\le x_i \le r, \ i=1,...,m\}$ containing $Q$ and put $\Pi_i=p^{i}(\Pi_m)$. Each $\Pi_i$ is an $i$-dimensional cube.
For arbitrary $\sigma>0$ and $\varepsilon>0$ find $\rho>0$ such that
\begin{equation}\label{eq-C-1}
\Vert x - x' \Vert_E < \rho \ \mbox{a.s.} \ \ \ \Rightarrow \ \ \ P\{\Vert hx-hx'\Vert_Y\ge \sigma\}<\varepsilon
\end{equation}
for all $x, x' \in \mathcal Pa(\Pi_m).$ This follows from uniform continuity of $h$ on $\mathcal Pa(\Pi_m)$.
We want to construct a finite subset $F\subset \Pi_m$ satisfying
the following condition: for any $x\in{\mathcal P}a(\Pi_m)$ there exists $x'\in{\mathcal P}a(F)$ for
which $\|x-x'\|_E<\rho$. For this purpose, we divide the interval $[-r,r]$ into disjoint intervals $[-r,-r+\xi]$, $(-r+\xi, -r+2\xi]$,..., $(r-\xi, r]$, where $\xi<\frac{\rho}{\sqrt{m}}$. This
induces the partition of the cubes $\Pi_i$ into disjoint cubic cells $\Pi(J_i)$ ($J_i$ is an associated $i$-dimensional multi-index) of equal size and the diameter less than $\rho$.  Let $c(\Pi(J_i))$ be the center of the cubic cell $\Pi(J_i)$, let the finite set $%
F_i$ consist of all these centers and put $F=F_m$.
By construction, the projection $p^{ji}$ maps each cell $\Pi(J_i)$ onto some cell $\Pi(J_j)$, and in this case $p^{ji}(c(\Pi(J_i)))=c(\Pi(J_j))$.
%Let $c(\Pi_i(J_i, \rho))$ be the center of the cubic cell $\Pi_i(J_i, \rho)$ and

Given $x\in {\mathcal P}a(\Pi_m)$ define
%the random point $p^{i}x: \Omega_i\to E_i$ is $\mathcal F_i)$-measurable by construction. Define
\[
A(J_i)\equiv \{\omega \in \Omega \;:\;p^ix(\omega )\in \Pi(J_i)\}\in\mathcal F_i
\ \ \
\mbox{and} \ \ \
x'_i(\omega)=c(\Pi(J_i))\;\;\;\;\hbox{if}\;\;\;\;\omega \in A(J_i)
\]
and put  $x'=x'_m$.
Evidently, $\Vert x-x'\Vert <\rho $ a.s. and $p^{im}x=x'_i$ for any $i\in T_m$, which implies that $x'\in {\mathcal P}a(F)$.

Let $F=\{f_1, ... f_s\}$ and $h(f_k)=y_k$ ($1\le k \le s$). The set $h(F)=\{y_k: 1\le k\le s\}\subset \mathcal P(Y, \mathcal S)$ contains finitly many random points, so that there exists a compact $C\subset Y$ and a set $B\subset\Omega$ such that $PB\ge 1-\varepsilon$ and
 $y_k(\omega)\in C$ for all $1\le k \le s$ and almost all $\omega\in B$.
On the other hand, arbitrary $u\in \mathcal Pa(F)$ can be represented as
$u=\sum_{k=1}^sf_k I_{B(k)}$ for some measurable subsets $B(k)$. By locality of $h$, we then obtain $hu=\sum_{k=1}^s h(f_k)I_{B(k)}=\sum_{k=1}^s y_k I_{B(k)}$. Hence $(hu)(\omega)\in C$ for almost all $\omega\in B$, so that $P\{hu\notin C\}<\varepsilon$ for all $u\in \mathcal Pa(F)$. Now, for an arbitrary
 $x\in {\mathcal P}a(\Pi_m)$ we put $u=x'$ and minding (\ref{eq-C-1}) yields $P\{hx\notin C_\sigma\}<2\varepsilon$, where $C_\sigma$ is the $\sigma$-neighborhood of $C$. By  Remark \ref{rem_def_tight_sets} it means that $h(\mathcal K_0)\subset h({\mathcal P}a(\Pi_m))$ is tight.

Step 2. Consider now the case of a general $X$. Let $\mathcal K_0$ be an arbitrary tight subset of $\mathcal Pa(X)$, $\sigma$ and $\varepsilon$ two positive numbers and $Q_0$ be a compact subset of $X$ satisfying the property $%
P\{x\notin Q_0\}<\varepsilon $ $(\forall x\in {\mathcal K}_0)$.

Put $\mathcal K =\bigcup_
{n\ge 1}\pi_n(\mathcal K_0)$. This set is tight, as each $y\in\mathcal K$ satisfies $P\{y\notin Q\}<\varepsilon $, where
 $Q$ is the closed convex hull of the precompact set $\bigcup_
{n\ge 1}\pi_n(Q_0)$. By uniform continuity of $h$ on tight subsets, we can find $\rho>0$ and $\delta>0$ such that
$$
P\{\Vert x-y\Vert_X\ge \rho\}<\delta \ \ \mbox{implies} \ \
 P\{\Vert hx-hy\Vert_Y\ge \sigma\}<\varepsilon \ \ \ \forall x\in\mathcal K.
$$
As $\pi_nx\to x$ in probability, we can find an $m\in N$ with the property $P\{\Vert x-\pi_mx\Vert_X\ge \rho\}<\delta$. According to Step 1, the set $h(\mathcal Pa(\pi_m(Q_0))$ is tight in $\mathcal P(Y, \mathcal S)$, so that there exist a compact $C\subset Y$ such that $P\{hz\notin C\}<\varepsilon$ for all $z\in \mathcal Pa(\pi_m(Q_0))$. For any $x\in \mathcal K_0$ we put $y=\pi_nx\in\mathcal K$. By Lemma \ref{lem-projection}, there exists
$z\in \mathcal Pa(\pi_m(Q_0))$ such that $y(\omega)=z(\omega)$ as long as $y(\omega)\in \pi_m(Q_0)$, so that $P\{y\ne z\}<\varepsilon$. Thus, for the $\sigma$-neighborhood $C_\sigma$ of $C$ we obtain
$$
\begin{array}{l}
  P\{x\notin C_\sigma\}\le P\{hz\notin C\}+P\{\Vert hx-hz\Vert_Y\ge \sigma\} \\
  \le\varepsilon+ P\{\Vert hx-hy\Vert_Y\ge \sigma\}+P\{hy\ne hz\}<2\varepsilon + P\{y\ne z\}<3\varepsilon,
\end{array}
$$
because $\{hy\ne hz\}\subset \{y\ne z\}$ due to locality of $h$.
By Remark \ref{rem_def_tight_sets}, the set $h(\mathcal K)$ is tight.
\end{proof}

\begin{proposition}
\label{prop-tightness-superposition} Let  $\mathcal X$, $\mathcal Y$ and $\mathcal U$ be projective system of separable Banach spaces. Let the operators $h_1:{\mathcal P}a(\mathcal X, \mathcal B)\to {\mathcal P}a(\mathcal U, \mathcal B)$ and
$h_2:{\mathcal P}a(\mathcal U, \mathcal B)\to {\mathcal P}(Y, \mathcal S)$ be local and uniformly continuous
on tight subsets of the corresponding domains. Then the operator $%
h=h_1\circ h_2$ will be local and tight if either

1) $h_2$ is bounded (i.e. it maps bounded sets into
bounded ones) and $h_1$ is tight, or

2) $h_2$ is tight and the projective system  $\mathcal U$ satisfies Property $(\Pi)$.
\end{proposition}

\begin{proof}
Evidently, the superposition of local operators is local. Now,
the first statement follows directly from the definitions, while the second statement follows from Proposition \ref{prop-transform-tight-sets}.
\end{proof}

\begin{remark}
Properties of the operators in Propositions \ref{prop-sequence-tight-op}-\ref{prop-tightness-superposition} mimic to some extent the corresponding properties of deterministic operators: 1) the limit of a sequence of compact operators is compact if the convergence is uniform on bounded subsets; 2) continuous operators map compact sets into compact sets; 3) the superposition $h_1\circ h_2$ of two continuous operators is compact if either $h_2$ is bounded and $h_1$ is compact or $h_2$ is compact.

The property of locality is essential for the results in this section: none of them is, in general, true if at least one
of the involved operators is not local.
\end{remark}

\section{Examples}\label{sec_examples}

\subsection{Examples of projective systems}\label{sec_ex_proj_systems}

\begin{example}\label{ex-proj-families}
Euclidean projective systems $\mathcal E = (E_i, p^{ji}, T_m) $, see Definition \ref{def-Euclidean}.
\end{example}

\begin{example}\label{ex-Pi-for-C}

Let  $t\in T\equiv [a, b]$, $X_t=C[a,t]$, $p^{ut}:\ C[a,t]\to C[a,u]$ be the restriction maps. We prove that the projective system
  $\mathcal X=(C[a,t], p^{ut}, T)$ satisfies {Property ($\Pi$)}.

  Due to the linear rescaling of the variable $t$, it suffices to consider $T=[0,1]$. For $n\in N$ we put $\delta_n=\frac{1}{n}$ and define
  \begin{equation}\label{eq-projections_C}
    (\pi_nx)(t)=\sum\limits_{k=0}^{n-1}\left[(x({k}{\delta_n})-x({(k-1)}{\delta_n}))(nx-k)+x({(k-1)}{\delta_n})\right]I_{[{k}{\delta_n}, {(k-1)}{\delta_n})},
  \end{equation}
for any $x\in C[0,1]$ ($x(-\delta_n)=0$). As
$$
\begin{array}{l}
(\pi_nx)({k}{\delta_n}\!+\!0)=\left(x({k}{\delta_n})-x({(k-1)}{\delta_n})\right)(n\cdot {k}{\delta_n}-k)+x({(k-1)}{\delta_n})=x({(k-1)}{\delta_n}) \ \ \ \mbox{and}
 \\
(\pi_nx)({k}{\delta_n}\!-\!0)=\left(x({(k-1)}{\delta_n})-x({(k-2)}{\delta_n})\right)(n\cdot {(k-1)}{\delta_n}-k)+x({(k-2)}{\delta_n})=x({(k-1)}{\delta_n}),
 \end{array}
$$
the piecewise linear function $\pi_nx$ is continuous for all $t\in T$. On the other hand, if $x(s)=y(s)$ ($0\le s \le t$) and $t\in [{k}{\delta_n}, {(k-1)}{\delta_n})$, then $(\pi_n x)(s)=(\pi_n y)(s)$ ($0\le s \le t$) due to (\ref{eq-projections_C}). Evidently, this implies the Volterra property of $\pi_n$ in the sense of Definition \ref{def-Volterra}. Finally, uniform continuity of $x\in  C[0,1]$ implies $\|\pi_n x-x\|_{C[0,1]}\to 0$ ($n\to\infty$), so that the sequence $\{\pi_n\}$ strongly converges to the identity operator in the space $C[0,1]$. {Property ($\Pi$)} is verified.
\end{example}

\begin{example}\label{ex-Pi-for-Lp}

Let $1\le r <\infty$, $t\in T\equiv [a, b]$, $X_t=L^r[a,t]$, $p^{ut}:\ L^r[a,t]\to L^r[a,u]$ be the restriction maps. The projective system
  $\mathcal X=(L^r[a,t], p^{ut}, T)$ satisfies {Property ($\Pi$)} as well. To check this, we observe that the sequence of operators
  $$
  (\tau_nx)(t)=\int\limits_a^t\Delta_n(t-s)x(s)ds \ \ \ (n\in N),
  $$
  where $\Delta_n(u)\ge 0$ ($u\in R$) is a continuous function satisfying the properties $\Delta_n(u)=0$ outside $[a, a+\frac{b-a}{n}]$ and $\int\limits_R \Delta(u)du=1$ strongly converges to the identity operator in the space $L^r[a,b]$ (due to the standard argument). On the other hand, $\tau_n(L^r[a,b])\subset C[a,b]$ and since the topology on $C[a,b]$ is stronger, than the topology on $L^r[a,b]$, the sequence of finite dimensional Volterra maps $\tau_n\circ\pi_n$ ($\pi_n: C[a,b]\to C[a,b]$ were defined in the previous example) strongly converges to the identity operator in the space $L^r[a,b]$. By construction, this sequence satisfies all requirements needed for {Property ($\Pi$)}.
\end{example}
\begin{remark}
From Examples \ref{ex-Pi-for-C}, \ref{ex-Pi-for-Lp} and Corollary   \ref{cor-fixed-point-general} we deduce Theorem  \ref{th-fixed-point}, the "light version" of the fixed-point theorem for LC operators, as Young expansions preserve the martingale property, which is proven below in Lemma \ref{lem-martingale-property}.

\end{remark}

\subsection{Examples of adapted random points}\label{sec-ex-adapted-points}

\begin{example}\label{ex-adapted points}
Let the projective system $\mathcal X$ be defined as in Examples \ref{ex-Pi-for-C} or \ref{ex-Pi-for-Lp}, i.e. $X$ is either $C[a,b]$ or $L^r[a,b]$ ($1\le r <\infty$). Let $\mathcal B$ be a right-continuous stochastic basis, i.e. $\mathcal F_t=\bigcap\limits_{s>t}\mathcal F_s$. Then using the standard approximation procedure  (see e.g. \cite{Oks}) it is straightforward to see that $\mathcal Pa(\mathcal X,\mathcal B)$ coincides with the space of all (equivalence classes of indistinguishable) stochastic processes that are $\mathcal F_t$-adapted, $\mathcal F\otimes \mbox{Bor}(R)$-measurable  and whose trajectories a.s. belong to the space $C[a,b]$ and $L^r[a,b]$, respectively.
\end{example}

\subsection{Examples of local operators}\label{sex_ex_local_op}

\begin{example}
  Any finite linear combination of local (resp. local and continuous) operators is again local (resp. local and continuous).
\end{example}

\begin{example}\label{ex_Carath}
 The superposition operator $$h_f: \ {\mathcal P}(X, \mathcal S)\to {\mathcal P}(Y, \mathcal S), \ \ \mbox{defined by} \ \ (h_fx)(\omega)=f(\omega, x(\omega)),$$ where $f: \Omega\times X\to Y$ is a given random function, is local, as   $x(\omega)=y(\omega)$ a.s. on $A\subset \Omega$ implies $$(h_fx)(\omega)=f(\omega, x(\omega))=f(\omega, y(\omega))=(h_fy)(\omega) \ \mbox{a.s. on } \ A.$$
 If, in addition, $f: \Omega\times X\to Y$ is a Carath\'{e}odory map, i.e. $f(\cdot, x)\in \mathcal{P} (Y, \mathcal S)$ for all $x\in X$ and
$f(\omega, \cdot)$ is continuous for almost all $\omega\in\Omega$, then the superposition operator $h_f: \mathcal{P} (X, \mathcal S)\to \mathcal{P} (Y, \mathcal S)$ is continuous in probability. The converse is true as well: If a local operator $h: \ {\mathcal P}(X, \mathcal S)\to {\mathcal P}(Y, \mathcal S)$ is continuous in probability, then $h=h_f$ for some Carath\'{e}odory map $f: \Omega\times X\to Y$, see \cite{Pon-0}. This result is also valid for random subsets of $X$, see Theorem \ref{th-Nemytskii}.
\end{example}

\begin{example}\label{ex_stoch_int_local}
Let the projective system $\mathcal X$ be as in Examples \ref{ex-Pi-for-C} or \ref{ex-Pi-for-Lp}. The It\^{o} integral
$$(Ju)(s)=\int_{a}^{t}u(s)dW(s)$$ is a LC operator acting from the space $\mathcal Pa(\mathcal X,\mathcal B)$ to the space $\mathcal Pa(\mathcal Y,\mathcal B)$ consisting of adapted stochastic processes with the trajectories belonging to $X=C[a,b]$ or $L^r[a,b]$ ($2\le r <\infty$) and $Y=C[a,b]$ or $L^q[a,b]$ ($1\le q <\infty$), respectively. In this example, the domain of the local operator is a proper subset of ${\mathcal P}(X, \mathcal S)$, and the representation by a Carath\'{e}odory function is no longer true. Otherwise, the It\^{o} integral would have been a Lebesgue-Stilties integral by the Riesz representation theorem.
\end{example}

\begin{example}\label{ex-local-comp}
This example generalises Example \ref{ex_stoch_int_local}.

The composition $(hx)(t)=\int_{a}^{t}F(s,x(s))dW(s)$
of
the It\^{o} integral with a superposition operator is an LC operator acting from $\mathcal Pa(X)$ to $\mathcal Pa(Y)$, where $X=C[a,b]$ and $Y=C[a,b]$ or $Y=L^q[a,b]$ ($1\le q<\infty$), provided that  the following conditions are satisfied:
\begin{itemize}
  \item $F(\cdot, \cdot, x)$ is $\mathcal F\otimes \mbox{Bor}([a,b])$-measurable for all $x\in R^n$
    \item $F(\cdot, t, x)$ is $\mathcal F_t$-adapted for any $t\in [a,b]$ and $x\in R^n$;
  \item $F(\omega, t, \cdot)$ is continuous for $P\otimes\mu$-almost all $(\omega, t)\in \Omega\times [a,b]$, where $\mu$ is the Lebesgue measure on $[a,b]$.
  \item $\int_{a}^{b}(\sup\limits_{|x|\le r}|F(\omega, t, x)|)^2dt<\infty$ a.s.
\end{itemize}
Indeed, in this case the random map $f: \ \Omega\times X\to Y$, given by $(f(\omega, x(\cdot))(t)=F(\omega, t, x(t))$, $t\in [a,b]$, is Carath\'{e}odory and due to the last condition maps $\mathcal P(X)$ to $\mathcal P(Y),$ see e.g. \cite{Appell}. Moreover, it maps $\mathcal Pa(X)$ to $\mathcal Pa(Y)$ due to the second and third assumption, so that the claim follows from Examples  \ref{ex_Carath} and \ref{ex_stoch_int_local}.

If the last of the above conditions on $F$ is replaced by the condition
\begin{itemize}
  \item $|F(\omega, t, x)|\le A(\omega, t)+C|x|^{p/q},$
\end{itemize}
where $p\ge 2,$ $1\le q \le\infty$, $A$ is a measurable stochastic process with $L^q$-trajectories and $C\ge 0$ is a constant, then $h$ acts from $\mathcal Pa(X)$ to $\mathcal Pa(Y)$, where $X=L^r[a,b]$ and  $Y=L^q[a,b]$ if $1\le q <\infty$ and $Y=C[a,b]$ if $q =\infty$. This follows from the continuity properties of the superposition operator in $L^r$-spaces.

Finally, the function $F$ can be replaced by a random continuous Volterra operator $V_\omega : X\to Y$ such that
$V^t_\omega(x)$ is $\mathcal F_t$-measurable for any $x\in X$, where $V^t_\omega$ is the restriction of $V_\omega$ on the subspace $C[a,t]$, as in this case $h_V$ acts from $\mathcal Pa(X)$ to $\mathcal Pa(Y)$. Here $X$ and $Y$ are again one of the above functional spaces.
\end{example}

\begin{example}
  More general stochastic integrals are also LC operators as long as they can be defined as limits in probability of finite dimensional approximations. However, in this case the domain and the range may be more complicated, see \cite{Pon-3}.
\end{example}

\begin{example}\label{ex-evolution}
  More nontrivial examples of local operators are given by the evolution operators $U^t_a$ constructed for finite or infinite dimensional stochastic differential equations with the existence and uniqueness property on some interval $[a,b]$, see e.g.
   \cite[Prop. 5.1, 5.5]{Pon-4}.

Indeed, suppose that
$x_0|_A=y_0|_A$ a.s. for some $A$. As $x_0$ and $y_0$ are $\mathcal F_a$-measurable, we may assume that $ A\in \mathcal F_{a}$. Put $x(t)=U^t_ax_0$, $y(t)=U^t_ay_0$,
$z(t)=x(t)I_A+y(t)I_{\Omega - A}$ ($t\ge a$) and observe that due to the locality of stochastic integrals (see Example \ref{ex_stoch_int_local}), $z(t)$ is a
unique solution of the underlying equation, which satisfies
$z(a)=x_0I_A+y_0I_{\Omega - A}=y_0$. By the uniqueness property, $z(t)=y(t)$ a.s. for all $a\le t \le b$.
In particular,
$x(t)|_A=y(t)|_A$ a.s. This yields locality of the evolution
operator $U_{a}^{t}$ for $a\le t\le b$.

This example shows that the evolution operators are always LC operators, and this property is a simple consequence of the well-posedness of the initial value problem for the underlying stochastic equation. In this respect, it is important to remark that evolution operators are not always generated by Carath\'{e}odory functions. For instance, the so-called "singular" delay differential equations do not produce Carath\'{e}odory evolution operators  \cite{Moh}. Another example is described in \cite[Ex. 6.2]{Pon-4}.
\end{example}

\begin{example}\label{ex-local-diff}
Differentiation is also an example of a local operator which cannot be represented by a  Carath\'{e}odory function.
\end{example}

%  This example is inspired by the paper \cite{Hasen}.
%
%Let $\mathcal B$ be a stochastic basis on $T=[a,b]$, $H$ be a separable Hilbert space, $V$ be a separable Banach space, $X$ be a separable Banach function space consisting of (some) functions $\eta: T\to V$,  and $W(t)$ ($t\in T$) be a Wiener process in $H$ with a nonnegative, symmetric trace class covariance operator in $H$.

%The conjugate Banach space $X'$ is compactly embedded into $X$.
%
%   Consider the equation
%\begin{equation}\label{eq_Partial_parameter}
%dx(t) = f(t,x(t),u(t))dt+F(t,x(t),u(t))dW(t),\ \ \  t\in [a,b],
%\end{equation}
%where the solutions $x$ and the controls $u$ belong to the same Banach space $\mathcal M$ of $\mathcal B$-adapted stochastic processes with trajectories from the space $X$ (up to a $P$-zero set).
%Assume that the mappings $f: T\times X \times X\to X$,
%$F: T\times X\times X \to L_{HS}(H,V)$, where $L_{HS}(H,V)$ is the space of Hilbert-Schmidt operators from $H$ to $V$, are chosen to guarantee that for any $u\in\mathcal M$ and $\mathcal F_a$-measurable $x_0: \Omega\to V$ there is a unique solution $x\in \mathcal M$ of Eq. (\ref{eq_Partial_parameter}).

\subsection{Examples of tight operators}\label{sec_ex_tight}

\begin{example}\label{ex-linear-comb-tight}
Some general properties:
\begin{itemize}
  \item Any finite linear combination of local and tight (resp. tight-range) operators is again local and tight (resp. tight-range).
  \item For operator superpositions check Proposition \ref{prop-tightness-superposition}.
  \item For uniform limits of sequences of tight operators check Proposition \ref{prop-sequence-tight-op}.
\end{itemize}

\end{example}

\begin{example}\label{ex-uniform-caratheodory}
  For any separable Banach spaces $X$ and $Y$ and any Carath\'{e}odory map $f: \Omega\times X\to Y$, the superposition operator $h_f: \ \mathcal{P} (X, \mathcal S)\to \mathcal{P} (Y, \mathcal S)$ is local and uniformly continuous on tight subsets of $\mathcal{P} (X)$.
%\end{proposition}
%\begin{proof}

We only have to prove the property of uniform continuity.
Let $\mathcal{P} (X)\equiv \mathcal{P} (X, \mathcal S)$ and $\mathcal{P} (Y)\equiv \mathcal{P} (Y, \mathcal S)$ and $\mathcal K$ be an arbitrary tight subset of $\mathcal P(X)$ and $\sigma>0$, $\varepsilon >0$. Choose a
compact and convex set $K\subset X$ such that
$$
P\{x\notin X\}<\varepsilon \ \ \ \mbox{for any} \ \ \ x\in\mathcal K
$$
and put
\[
\theta _\delta (\omega )\equiv \sup \left\{ \Vert f(\omega ,x)-f(\omega
,y)\Vert_Y ,\hskip 0.5truecmx,y\in K,\;\Vert x-y\Vert_X \le \delta \right\} .
\]
Since $K$ is compact, $\delta \,\to \,0$ implies that $\theta _\delta $
goes to zero a.s. and hence in probability. Thus, for any positive $\sigma
,\varepsilon $ there is $\delta >0$ such that ${ P}\{\theta _\delta
(\omega )\ge \sigma \}<\varepsilon $. Pick two arbitrary random points $x_1$
and $x_2$ from $\mathcal K$ satisfying ${P}\{\Vert x_1-x_2\Vert_X
\ge \delta \}<\varepsilon $ and put $\hat x_i=\pi (x_i)$ ($i=1,2$) where $\pi: \ X\to K$ is a continuous projection ($K$ is convex, closed and bounded). Let $$\hat\Omega=\{\omega\in\Omega: \ x_1(\omega)\in K \ \& \ x_2(\omega)\in K \}.$$
Then $x_i(\omega)= \hat x_i(\omega)$ ($i=1,2$),  and therefore $(h_fx_i)(\omega)= (h_f\hat x_i)(\omega)$ ($i=1,2$), if
$\omega \in \hat\Omega$. On the other hand, $P\hat\Omega\ge 1-2\varepsilon$ as $P\{x_i\in K\}<\varepsilon$ ($i=1,2$). Therefore,
\[
\begin{array}{c}
  P\{\Vert h_fx_1-h_fx_2\Vert_Y \ge \sigma \}\le {P}\{\Vert \hat x_1-\hat x_2\Vert_X
\ge \delta \}+{P}\{\theta _\delta (\omega )
\ge \sigma \}\\ +P\{h_fx_1\ne h_f\hat x_1\} + P\{h_fx_2\ne h_f\hat x_2\}<4\varepsilon,
\end{array}
\]
which yields the uniform continuity of $h$ on $\mathcal K$.
\end{example}

\begin{example}\label{ex-Carat-tight}
 If a Carath\'{e}odory map $f: \Omega\times V\to Y$ is an almost surely compact (resp. compact-range) operator from $V\subset X$ to $Y$, then the superposition operator $h_f: \ \mathcal{P} (V)\to \mathcal{P} (Y)$ is tight (resp. tight-range).
%\item The linear operators $I^1$ and $J^1$ are local and uniformly continuous on bounded subsets of their domains as long as they are continuous.
%\end{proposition}
%\begin{proof}

Consider an arbitrary tight subset $\mathcal K\in\mathcal P(V)$ and arbitrary positive numbers $\varepsilon>0, $ $\sigma >0$.  Pick $r>0$ for which
$P\{x\notin B_r\cap V \}<\varepsilon$ for all $x\in \mathcal K$, where $B_r=\{z\in X\,:\,\Vert z\Vert_X \le r\}$. Let also fix a countable set $%
\{z_i, \ i\in N\}$ which is dense in $B_r\cap V$. For each
$\omega $ the set $H(\omega )\equiv \{f(\omega ,z_i), \ i\in N\}$ is
precompact. Therefore, the measurable function
\[
k_n(\omega )\equiv \sup_{v\in H_{}(\omega )}\inf_{u\in H_n(\omega )}\Vert
v-u\Vert_Y,
\]
where $H_n(\omega )\equiv \{f(\omega
,z_i),\;1\le i\le n\},$
tends to zero a.s. and hence in probability. Geometrically, it means that there exists a number $m\in N$ and a subset $\Omega^1_\varepsilon\in \mathcal F$, $P\Omega^1_\varepsilon\ge 1-\varepsilon$, such that the set $%
H(\omega )$ is contained in the $\sigma $-neighborhood of the finite set $%
H_m(\omega )$ if $\omega\in\Omega^1_\varepsilon$. Let $K$ be a compact for which $x_i(\omega)\in K$ ($i=1,...,m$) if $\omega\in\Omega^2_\varepsilon$ and $P\Omega^2_\varepsilon\ge 1-\varepsilon$. Therefore,
for each random point $z$ taking values in $\{z_i\} $, $i\in %
\mbox{N}$, one has
\begin{equation}\label{eq-1}
(h_fz)(\omega)=f(\omega, z(\omega))\in K_\sigma \ \ \ \mbox{if} \ \ \ \omega\in \Omega_\varepsilon\equiv \Omega^1_\varepsilon\cap\Omega^2_\varepsilon,
\end{equation}
where $P\Omega_\varepsilon\ge 1-2\varepsilon.$ The set of all such $z$ is dense in $\mathcal P(B_r\cap V)$, so that (\ref{eq-1})
holds true for all $z\in\mathcal P(B_r\cap V)$. Defining for any $x\in\mathcal K$ the random point $z\in\mathcal P(B_r\cap V)$ by the formula $z(\omega)=x(\omega)$ if $x(\omega)\in B_r$ and $z(\omega)=0$ otherwise, we get $P\{x\ne z\}<\varepsilon$, so that
$$
{P}\{h_fx\notin K_\sigma\}\le P\{h_fz\notin K_\sigma\}+ P\{x\ne z\}
< 3\varepsilon .
$$
By Remark \ref{rem_def_tight_sets}, the set $h_f(\mathcal K)$ is tight.
%Applying again Theorem \ref{T} implies the
%tightness of $h$.
\end{example}

Deterministic integrals define compact operators in typical functional spaces. The next example shows that stochastic integrals define tight operators in typical spaces of stochastic processes. For the sake of simplicity we only consider It\^{o} integrals. However, more general stochastic integrals give rise to tight operators as well, see e.g. \cite{Pon-3}.

\begin{example}\label{ex-integral-tight}
  Let the projective system $\mathcal X$ be as in Example \ref{ex-Pi-for-C} and let $K(t,s)$ be a continuous (determinsitic) function on $[a,b]\times [a,b]$. Consider the It\^{o} integral operator
$$(Ju)(s)=\int_{a}^{t}K(t,s)u(s)dW(s)$$ as a LC operator acting from the space $\mathcal Pa(\mathcal X,\mathcal B)$ to the space $\mathcal Pa(\mathcal Y,\mathcal B)$ consisting of adapted stochastic processes with the continuous or $p$-integrable trajectories. We claim that the operator $J$ is tight if one of the following conditions is fulfilled:
\begin{enumerate}
  \item $X=C[a,b]$ or $X=L^r[a,b]$ ($2\le r <\infty$) and $Y=L^q[a,b]$ ($1\le q <\infty$);
  \item $X=C[a,b]$ or $L^r[a,b]$ ($2< r <\infty$) and $Y=C[a,b]$.
\end{enumerate}
To simplify the presentation we assume that $[a,b]=[0,1]$ and $K(t,s)\equiv 1$.
Let us first consider the case $X=L^r[a,b]$, where $r=2$. Notice that the
imbedding $L^2[a,b]$ in $L^q[a,b]$ is a continuous map if
$1\le q<2.$ It is sufficient, therefore, to consider the case
$2\le q<\infty $. Put
\[
g_t^n\equiv \sum_{k=0}^{n-1}{\frac kn}I_{[{\frac kn},{\frac{{k+1}}n})}(t)
\]
where $I_A$ is the indicator of the set $A$. Clearly, $\,g_t\le t$. The
standard estimates for stochastic integrals yield
\[
E\Bigl|\int_0^1I_{[g_t^n,t]}(s)u(s)dW(s)\Bigr|^q\le \hbox{const}\,%
E\left( \int_0^1I_{[g_t^n,t]}(s)u(s)^2ds\right) ^{\frac q2}.
\]
Therefore,
\[
E\int_0^1\Bigl|\int_0^1I_{[g_t^n,t]}(s)u(s)dW(s)\Bigr|^qdt=%
\int_0^1dt\;E\Bigl|\int_0^1I_{[g_t^n,t]}(s)u(s)dW(s)\Bigr|^q\le
\]
\[
\le \hbox{const}\,\int_0^1dt\;E\int_0^1I_{[g_t^n,t]}(s)u(s)^2ds^{%
\frac q2}=\hbox{const}\,E\int_0^1dt\left(
\int_0^1I_{[g_t^n,t]}(s)u(s)^2ds\right) ^{\frac q2}\le
\]
\[
\le \hbox{const}\,E\left( \int_0^1u(s)^2ds\left(
\int_0^1I_{[g_t^n,t]}(s)dt\right) ^{\frac 2q}\right) ^{\frac q2}\le
\]
\[
\le \hbox{const}\,E\left( \int_0^1u(s)^2ds\right) ^{\frac
q2}\times \sup_{0\le s\le 1}\int_0^1I_{[g_t^n,t]}(s)dt\le \hbox{const}\,{%
\frac 1n}E\Vert u\Vert _{L^2}^q
\]
due to the generalized H\"{o}lder inequality. Therefore,
\[
E\Bigl\Vert \int_0^{\cdot }u(s)dW(s)-\int_0^{g_n(\cdot
)}u(s)dW(s)\;\Bigr\Vert _{L^q}^q\;\le \;{\frac{\hbox{const}\,}n}E%
\Vert u\Vert _{L^2}^q,
\]
which means that a sequence of linear random finite dimensional (and therefore tight) operators
converges to $J$ uniformly on the sets $\{u\in {\mathcal Pa(X)}\,:\;\Vert u\Vert
_{L_2}\le r \ \mbox{a.s.}\}$ Applying Proposition \ref{prop-sequence-tight-op} completes the consideration of the case
$r=2\;$, $1\le q<\infty $.

Assume now that $r>2$ and $Y=C[a,b]$. Then (see e.g. \cite{Oks})
\[
E\left( \sup_{0\le s\le 1}\Bigl|\int_0^tu(s)dW(s)\Bigr|^2\right)
\le 4E\int_0^1u(s)^2ds\le \hbox{const}\,E\Vert u\Vert
_{L^{p}}^{r},
\]
and
\[
E\Bigl|\int_0^tu(s)dW(s)-\int_0^uu(s)dW(s)\Bigr|^2=\mbox{\bf
E}\Bigl|\int_u^tu(s)dW(s)\Bigr|^2=\mbox{\bf
E}\int_u^tu(s)^2ds\le
\]
\[
\le E\left( \int_u^tds\right) ^{\frac{1}{r'}}\left(
\int_u^tu(s)^{r}ds\right) ^{\frac{2}{r}}\le |t-t'|^{\frac 1{{r'}}}\Vert u\Vert
_{L^{r}}^2
\]
where $r'={\frac r{r-2}}$. By Kolmogorov's criterion $J$ maps
subsets $\{u\in \mathcal Pa(X): \ ||u||_{L^{r}}\}$ of the space into tight subsets of the space $\mathcal Pa(Y)$. By
Theorem \ref{prop-tight-operators}, the operator $J$ is tight.

All other cases follow from the two considered, as the space $C[a,b]$ is continuously imbedded in any space $L^r[a,b]$.
\end{example}

\begin{example}\label{ex-composition-tight}
The composition of the It\^{o} integral $J$ with any of the superposition operators $h_f$ and $h_V$ from Example \ref{ex-local-comp} is a tight local operator acting from $\mathcal Pa(X)$ to $\mathcal Pa(Y)$, where $X=C[a,b]$ or $X=L^r[a,b]$ $(2< r<\infty)$ and $Y=C[a,b]$ or $Y=L^q[a,b]$ ($1\le q<\infty$). This follows from the tightness properties of the operator $J$, the properties of superposition operators from Example \ref{ex-Carat-tight} and Proposition \ref{prop-tightness-superposition}.
\end{example}

\begin{example}\label{ex-evolution-1}
  The evolution operators $U(t)$ for stochastic differential equations with bounded delays are local and tight for sufficiently large  $t$, see \cite{Pon-4} for the details.
\end{example}

\subsection{Examples of Young expansions}\label{sec_ex_Young_ext}

\begin{example}\label{ex-Young-stoch-basis}
  Suppose that $\mathcal B^*=(\Omega^*,{\mathcal F}^*,{\mathcal F}_t^*,P^*)$ is an expansion of the stochastic basis $\mathcal B =(\Omega, {\mathcal F},{\mathcal F}_t, P)$ where the measure $P^*$ is generated by a random Dirac measure  $P^*=P\alpha^{-1}$ for some $\alpha\in \mathcal Pa (Z) \equiv\mathcal{P}a(\mathcal Z, \mathcal B)$, i.e. $P^*(A)=P\{\omega\in\Omega: \ \alpha(\omega)\in A(\omega)\}$. By Definition \ref{def-Young-extension}, this is a Young expansion of $\mathcal B$. We claim that
  the measure preserving map $\omega\mapsto (\omega,\alpha(\omega))$ generates a linear topological isomorphism between the spaces $\mathcal Pa (X)$ and $\mathcal Pa^* (X)$ defined by $\alpha_X: \ x\mapsto x\circ\alpha$.

  %is measure preserving, so that  $\alpha_X: \ x\mapsto x\circ\alpha$ is the linear topological isomorphism between the spaces $\mathcal P (X)$ and $\mathcal P^* (X)$.
    % We claim that the map $\alpha_X$ is, in fact, a linear topological isomorphism between the spaces $\mathcal Pa (X)$ and $\mathcal Pa^* (X)$.

  To see this, let us first check that $x\in \mathcal Pa^* (X)\equiv\mathcal{P}a(\mathcal X, \mathcal B^*)$ implies $x\circ \alpha\in \mathcal Pa(X)$. Below we will write $\Delta(\omega)$ for the set $\{z\in U: \ (\omega, u)\in \Delta\}$ where $\Delta\subset \Omega\times U$.
   Let $t\in T$ and $B\in \mbox{Bor}\,(Z_t)$, so that   $B^-\equiv (p^tx)^{-1}(B)\in F^*_t$. Then there exist $B_1, B_2\in \mathcal F_t^{*,0}$ such that $B_1\subset B^-\subset B_2$ and $P^*(B_1)=P^*(B_2)=P^*(B^-)$, which by the definition of $\mathcal F_t^{*,0}$ means that the set $\{(\omega, q^t(B_i(\omega))): \ \omega\in\Omega\}\in \mathcal F_t\otimes \mbox{Bor}\,(Z_t)$ ($i=1,2$). By the theorem of measurability of projections \cite{Wagner}, the sets $C_i\equiv \{\omega\in\Omega: q^t\alpha(\omega)\in q^t(B_i(\omega))\}$ belong to $\mathcal F_t$. In addition,
 $$
 \begin{array}{c}
   (p^t(x\circ\alpha))^{-1}(B)=\{\omega\in\Omega: \ p^t(x(\alpha(\omega))\in B\}=\{\omega\in\Omega: \ \alpha(\omega)\in B^-(\omega)\} \\
 =\{\omega\in\Omega: \ \alpha(\omega)\in B_2(\omega)\}-\Omega_0=\{\omega\in\Omega: \ q^t(\alpha(\omega))\in q^t(B_2(\omega))\}-\Omega_0,
 \end{array}
  $$
  where $\Omega_0=\{\omega\in\Omega:\ \alpha(\omega)\in B_2-B^-\}$. The latter is of measure $0$, because
  $$
  P\{\alpha\in (B_2-B^-)\}\le P\{\alpha\in (B_2-B_1)\}=P^*(B_2-B_1)=0
  $$
  and $B_2-B_1\in \mathcal F^{*,0}.$
  Therefore, $C_2-\Omega_0\in \mathcal F_t$, so that $x\circ \alpha\in \mathcal Pa(X)$. Thus, the correspondence $\alpha_X: x\mapsto x\circ\alpha$ is a linear isomorphism between $\mathcal Pa (X)$ and $\mathcal Pa^* (X)$ and since the map $\omega\mapsto (\omega,\alpha(\omega))$ is measure preserving, this correspondence is also a topological isomorphism.
\end{example}

\begin{example}\label{ex-iterated-stoch-bases}
  A  Young expansion $\mathcal B^{**} =(\Omega^{**} ,{\mathcal F^{**}},({\mathcal F_t^{**}})_{t\in {T}}, P^{**})$ of any  Young expansion $\mathcal B^* =(\Omega^* ,{\mathcal F^*},({\mathcal F_t^*}t)_{t\in {T}}, P^*)$ of a given stochastic basis $\mathcal B =(\Omega ,{\mathcal F},({\mathcal F}_t)_{t\in {T}}, P)$ is again a Young expansion of this basis.

  To see this, let us assume that $\Omega^*=\Omega\times Z_1$, $\Omega^{**}=\Omega^*\times Z_2$, where $Z_1$ and $Z_2$ are separable Frech\'{e}t spaces, and   $P_\omega^{**}$ is the limit (in the narrow topology) of a sequence of random Dirac measures. Pick arbitrary $\delta>0$ and arbitrary bounded random functions $f_i:\Omega\times Z_1\times Z_2 \to R$ that are continuous in $(z_1,z_2)$. This defines the neighborhood $U^*_{f_1,...,f_m,\delta}$ of the random measure $P^{**}_{(\omega, z_1)}$ in the space $P_{\Omega^{*}}(Z_2)$, see Subsection \ref{sec-extensions-bases}. Then there exists a random Dirac measure   ${\delta}_{\beta(\omega,z_1)}$, where $\beta\in \mathcal Pa(Z_2, \mathcal B^*)$, belonging to this neighborhood. By Theorem \ref{th-approx-simple-points}, the random points $\beta: \Omega^*\to Z_2$ can be assumed, without loss of generality, to be $P^{*}$-a.s. continuous in $z_1\in Z_1.$ This means that
  $$
  \left|\int\limits_{\Omega^{**}}f_idP^{**}-\int\limits_{\Omega^*}f_i(\omega, z_1, \beta(\omega, z_1))dP^*\right|<\delta, \ \ \ i=1,...,m.
  $$
  On the other hand, $P^*_\omega$ is the limit (in the narrow sense) of a sequence of random Dirac measures $\{\delta_{\alpha_n(\omega)}\}$, where $\alpha_n\in \mathcal Pa(Z_1, \mathcal B)$. As the functions $f_i(\omega, z_1, \beta(\omega, z_1))$ are $P^*$-a.s. continuous in $z_1$, Remark \ref{rem-narrow-topology} ensures that
  $$
  \lim\limits_{n\to\infty}\int\limits_\Omega f_i(\omega, \alpha_n(\omega), \beta(\omega, \alpha_n(\omega))dP=\int\limits_{\Omega^*} f_i(\omega, z_1, \beta(\omega, z_1)dP^*.
  $$ Therefore
$$
\left|\int\limits_\Omega f_i(\omega, \alpha_n(\omega), \beta(\omega, \alpha_N(\omega)))dP-\int\limits_{\Omega^{**}}f_idP^{**}\right|<2\delta
$$
for sufficiently large $n$. The random point $\gamma: \Omega\to Z_1\times Z_2$, defined as $$\gamma(\omega)=(\alpha_n(\omega), \beta(\omega, \alpha_n(\omega))),$$ is easy to see to be $\mathcal B$-adapted. By construction, it belongs to the neighborhood $U_{f_1,...,f_m,\delta}$ of the random measure $P^{**}_{\omega}$ in the space $P_{\Omega}(Z_1\times Z_2).$ Thus, $\mathcal B^{**}$ is a Young expansion of the stochastic basis $\mathcal B.$

Evidently, this construction can be iterated, so that finitely many consecutive Young expansions are all Young expansions of the original stochastic basis.
\end{example}

\begin{example}\label{ex-limit-stoch-bases}
The previous example can be extended to the case of countably many iterations. More precisely, let
  $$\mathcal B^*_\nu = (\Omega^\nu, \mathcal F^\nu, (\mathcal F^\nu_t)_{t\in T}, P^\nu), \ \ \ \nu\in N\cup \{0\}$$ be a sequence of stochastic bases, where $\mathcal B^*_0=\mathcal B$ and  $\mathcal B^*_\nu$ is a Young expansion of $\mathcal B^*_{\nu-1}$ for any $\nu\in N$. In particular,
  $$\Omega^\nu=\Omega\times \left(\prod\limits_{j=1}^\nu Z^j\right), \ \ \ \nu\in N,$$ where $Z^j\equiv Z^j_b$ are separable Frech\'{e}t spaces coming from the respective projective families $\mathcal Z^j=(Z_t^j, q^{ut}_j, T)$ ($b=\max T$).

 The direct product $Z^\infty=\prod\limits_{j=1}^{\infty} Z^j$ is a separable Frech\'{e}t space as well, and it gives rise to the stochastic basis
 $$
 \mathcal B^\infty=(\Omega^\infty, \mathcal F^\infty, (\mathcal F_t^\infty)_{t\in T}, P^\infty),
 $$
 where $\Omega^\infty=\Omega\times Z^\infty$ and $P^\infty$ is defined to be the inverse (projective) limit of the sequence $\{P^\nu\}$, while $\mathcal F^\infty$ and $\mathcal F^\infty_t$ are constructed according to the recipes from Definition \ref{def-Young-extension}.

 Let $\mathcal P^{\nu}:\ Z^\infty\to \prod\limits_{j=1}^{\nu} Z^j$ be the natural projections ($\nu \in N$), take arbitrary $\delta>0$ and random functions $f_i:\Omega\times Z^\infty \to R$ that are continuous in the second variable and bounded by 1.
  %consider the neighborhood $U^\infty_{f_1,...,f_m,\delta}$ of the random measure $P^{\infty}_{\omega}$ in the space $P_{\Omega}(Z^\infty)$ defined as
% $$
% \{\mu\in P_{\Omega}(Z^\infty): \ |E\int\limits_{Z^\infty}f_kd\mu-E\int\limits_{Z^\infty}f_kdP_\omega^\infty|, \ k=1,...,m \}.
% $$
 Since $Z^\infty$ is a separable metric space, the measure $P_\omega^\infty$ is a random Radon measure \cite[Th. 3.1.10, p. 3056]{Bog}, so that there exists a compact $C \subset Z^\infty$ such that $P^\infty(\Omega\times C)\ge 1-\delta$ and
 $$
 \left|\int\limits_{\Omega^\infty}f_kdP^\infty-\int\limits_{\Omega\times C}f_kdP^\infty\right|<\delta, \ k=1,...,m.
 $$
 Therefore, there exists a number $\nu$ and random functions $f_k^\nu:\Omega\times \prod\limits_{j=1}^{\nu} Z^j \to R$ that are continuous in the second variable, bounded by 1 and satisfying
 \begin{equation}\label{eq_proj_limit_1}
 P\{\sup\limits_{C} \left|f_k-(f_k^\nu\circ \mathcal P^\nu) \right|\ge\delta\}<\delta, \ k=1,...,m,
 \end{equation}
  so that
 $$
 \left|\,\int\limits_{\Omega\times C}f_kdP^\infty-\int\limits_{\Omega\times C}(f_k^\nu\circ \mathcal P^\nu)dP^\infty\right|<\delta, \ k=1,...,m.
 $$
 Hence
 $$
 \left|\,\int\limits_{\Omega^\infty}f_kdP^\infty-\int\limits_{\Omega^\infty}(f_k^nu\circ \mathcal P^\nu)dP^\infty\right|<4\delta, \ k=1,...,m,
 $$
 as $|f_k|\le 1$ and $P^\infty(\Omega^\infty-(\Omega\times C))\le \delta.$ By the definition of the inverse product of probability measures \cite{Bog},  $$\int\limits_{\Omega^\infty}(f_k^\nu\circ \mathcal P^\nu)dP^\infty=\int\limits_{\Omega^\nu}f_k^\nu dP^\nu, \ k=1,...,m.$$

 Applying the proposition from Example \ref{ex-iterated-stoch-bases} yields a random Dirac measure $\delta_{\gamma(\omega)}$, $\gamma\in \mathcal Pa(Z_\nu)$, which satisfies
 $$
 \left|\,\int\limits_{Z^\nu}f_k^\nu dP^\nu-E(f_k^\nu\circ\gamma) \right|<\delta.
 $$
 Let $\tilde\gamma (w)=(\gamma(\omega), 0)\in \prod\limits_{j=1}^\nu Z^j\times \prod\limits_{j=\nu+1}^\infty Z^j$.
 Then $\tilde\gamma\in \mathcal Pa(Z^\infty, \mathcal B)$ and
   $$|E(f_k^\nu\circ\gamma)-E(f_k\circ\tilde\gamma)|=|E(f_k^\nu\circ \mathcal P^\nu\circ\tilde\gamma)-E(f_k\circ\tilde\gamma)|<2\delta $$
   due to (\ref{eq_proj_limit_1}).

 Summarizing we obtain
 $$
 \left|\,\int\limits_{\Omega^\infty}f_kdP^\infty-E(f_k\circ\tilde\gamma)\right|<6\delta.
 $$
 As $\delta>0$ and random continuous functions $f_i:\Omega\times Z^\infty \to R$ bounded by 1 were arbitrary, we have proven that $P^\infty_\omega$ can be represented as the limit (in the narrow topology) of a sequence of random Dirac measures $\{\delta_{\tilde\gamma_n(\omega)}\}$, where $\tilde\gamma_n\in  \mathcal Pa(Z^\infty, \mathcal B)$. Thus, $\mathcal B^\infty$ is a Young expansion of the stochastic basis $\mathcal B_0=\mathcal B$ in the sense of Definition \ref{def-Young-extension}.

This construction can be generalized to the case of a countable projective family $(\mathcal B^*_\lambda, \mathcal P^{\mu\lambda}, \Lambda)$ of (partially ordered) Young expansions. The inverse limit of this family will be again a Young expansion of the stochastic basis $\mathcal B =(\Omega ,{\mathcal F},({\mathcal F}_t)_{t\in {T}}, P)$.
\end{example}

\subsection{Examples of LC extensions}\label{sec_ex_extensions}

\begin{example}\label{ex_superposition-extension}
%Let us consider two important examples.

Let $(\Omega^*,{\mathcal F}^*, P^*)$ be a probability space and $c:\ \Omega^*\to \Omega$ be a $({\mathcal F}^*,{\mathcal F})$- measurable surjection such that $P^*c^{-1}=P$. Denote by $\mathcal P^* (X)$ the set of all (equivalence classes of) random points $x: \ \Omega^*\to X$.

If $f(\cdot, \cdot): \Omega\times X\to Y$ is a Carath\'{e}odory function, then the associated (local and continuous) superposition  operator $h_f: \ \mathcal P (X)\to \mathcal P(Y)$,  $(h_fu)(\omega)=f(\omega, u(\omega))$ admits a unique LC extension $h_f^*: \ \mathcal P^* (X)\to \mathcal P^*(Y)$ given by $(h_f^*u)(\omega^*)=f(\omega, u(\omega^*))$. In other words, the extension $h_f^*$ of $h_f$ is the superposition operator generated by the same mapping $F$, but  naturally extended to a bigger probability space. In this case, the expansion does not need to be a Young expansion.

%The integral LC operator $(Ju)(s)=\int_{a}^{t}u(s)dW(s)$, defined on $\mathcal{P}a(X, \mathcal B)$, where $X=C({T})$ or $X=L^p({T})$ ($1\le p < \infty$), admits an LC extension $J^*$ only if the expansion  $\mathcal B^*$ of the underlying stochastic basis is acceptable. In this case,
%$
%(J^{*}u)(s)=\int_{a}^{t}u(s)dW^*(s),
%$
%where $W^*(t)=W(t)\circ c$ is the standard Wiener process on the stochastic basis $\mathcal B^*$.
\end{example}

In the next example we need
\begin{lemma}\label{lem-martingale-property}
Let $\mathcal B^*=(\Omega^*, \mathcal F^*, \mathcal F^*_t, P^*)$ be a Young expansion of the stochastic basis $\mathcal B$ in the sense of Definition \ref{def-Young-extension} and $M(t)$, $t\in [a,b]$, is a martingale on  $\mathcal B$ (see e.g. \cite{Oks}). Then the stochastic process $M^*(t, \omega^*)=M^*(t, \omega, z)\equiv M(t, \omega))$ is a martingale on the stochastic basis $\mathcal B^*$.
\end{lemma}
\begin{proof}
First of all, we notice that $M^*(t)$ is $\mathcal F_t\otimes (q^t)^{-1}\mbox{Bor}\,(Z_t)$-measurable. Indeed, for any $t\in [a,b]$ and $B\in \mbox{Bor}\, (R)$, the set $$\{M^*(t)\in B\}=\{M(t)\in B\}\times Z=\{M(t)\in B\}\times (q^t)^{-1}Z_t\in \mathcal F_t\otimes (q^t)^{-1}\mbox{Bor}\,(Z_t).$$
It remains, therefore, to prove the equality
\begin{equation}\label{eq-lemma-martingale-1}
  E^*(M^*(t)u)=E^*(M^*(s)u)
\end{equation}
for any $s$, $a\le s \le t$ and any $\mathcal F_t\otimes (q^t)^{-1}\mbox{Bor}\,(Z_t)$-measurable and bounded random variable $u: \Omega^*\to R$.
In fact, it is sufficient to check this equality for $u=I_{D}$, $A\in \mathcal D$, where $\mathcal D$ generates the $\sigma$-algebra $\mathcal F_t\otimes (q^t)^{-1}\mbox{Bor}\,(Z_t)$, in particular, for $P^*$-continuity sets of the form $D=A\times C$, where $A\in \mathcal F_s$ and $C=(q^s)^{-1}(C_0)$, $C_0\in \mbox{Bor}\,(Z_t)$. In this case, the function $M^*u$ becomes $P^*$-a.s. continuous, which gives us the opportunity to assume, without loss of generality, that $P^*$ is generated by a random Dirac measure $P^*_\omega=\delta_{\alpha(\omega)}$, $\alpha\in \mathcal Pa(Z)$, because  $P^*$ is a Young measure.

Under the above simplifications Eq. (\ref{eq-lemma-martingale-1}) becomes
\begin{equation}\label{eq-lemma-martingale-2}
\int_{\Omega}M(t,\omega)I_{A}(\omega)I_C(\alpha(\omega))dP=\int_{\Omega}M(s,\omega)I_{A}(\omega)I_C(\alpha(\omega))dP.
\end{equation}
Notice that $I_C\circ\alpha: \omega\to I_C(\alpha(\omega))$ is $\mathcal F_s$-measurable, because
$I_C\circ\alpha=I_{(q^s)^{-1}(C_0)}\circ\alpha=I_{C_0}\circ q^s\circ \alpha$ is the composition of the $\mbox{Bor}\,(Z_t)$-measurable function $I_{C_0}$ and the $\mathcal F_s$-measurable random point $q^s\circ \alpha: \Omega\to Z_s$. Thus, (\ref{eq-lemma-martingale-2}) follows from the assumption that $M(t)$ is a martingale on the stochastic basis $\mathcal B$. Therefore, the equality (\ref{eq-lemma-martingale-1}) is fulfilled as well, which means that $M^*(t)$ is a martingale on $\mathcal B^*$.
%$$\{\omega\in\Omega: \, I_C(\alpha(\omega))\in B\}=$$
\end{proof}

\begin{example}\label{ex-stoch-integrals}
Consider the LC operator $J: \mathcal{P}a(X, \mathcal B) \to \mathcal{P}a(X, \mathcal B)$ given by
$$
(Jx)(t)=\int\limits_{a}^tx(s)dW(s),
$$
where $W(t)$ is the standard scalar Wiener process on $[a,b]$ and $X$ is either $C[a,b]$ or $L^r[a,b]$. The operator $J$ is linear and, therefore, uniformly continuous on its domain (adapted stochastic processes with square integrable trajectories). By Theorem \ref{th-extension-operators-general}, $J$ admits a unique LC extension $J^*$ for  any Young expansion $\mathcal B^*$ of $\mathcal B$.

Let us check that  $W^*(t,\omega^*)=W^*(t,\omega, z)=W(t, \omega)$ remains the standard Wiener process on $\mathcal B^*$. Indeed, $W^*$ is sample continuous and by Lemma \ref{lem-martingale-property} it is a martingale with the zero mean (which coincides with $W^*(a)=W(a)=0$), and $(W^*)^2-t$ is a martingale as well by the same lemma. Thus, we have verified L\'{e}vy's characterization of the standard scalar Wiener process.

By this, the well-defined LC operator $\int\limits_{a}^tx(s)dW^*(s)$ extends the operator $J$. Applying the uniqueness property proven in Theorem \ref{prop-unique-local-extension} yields
$$
(J^*x)(t)=\int\limits_{a}^bx(s)dW^*(s).
$$
A similar argument can be used for an arbitrary stochastic integral defined on an appropriate domain described e.g. in \cite{Pon-3}.

Let us also remark that the operator $J$ cannot be extended to arbitrary expansions of $\mathcal B$, because $W^*(t)$ has at least to be a semimartingale in order that stochastic integration is properly defined.
\end{example}

\begin{example}\label{ex-extention compositions}
Combining Examples \ref{ex_superposition-extension} and \ref{ex-stoch-integrals} we get the formula for the (unique) LC extension
$$
(h^*x)(t)=\int\limits_a^tF(s,x(s))dW^*(s)
$$
of the nonlinear integral operator
$$
(hx)(t)=\int\limits_a^tF(s,x(s))dW(s),
$$
which is valid for any Young expansion of the underlying stochastic basis and a function $F: \ \Omega\times [a,b]\times R^n\to R^n$ satisfying the conditions from Example \ref{ex-local-comp}. The function $F$ can be again replaced by a Volterra operator described in the latter example.
%\begin{enumerate}
%  \item $F(\cdot, \cdot, x)$ is $\mathcal F\otimes \mbox{Bor}{R^n}$-measurable for all $x\in R^n$
%    \item $F(\cdot, t, x)$ is $\mathcal F_t$-adapted for any $t\in [a,b]$ and $x\in R^n$;
%  \item $F(\omega, t, \cdot)$ is continuous for $P\otimes\mu$-almost all $(\omega, t)\in \Omega\times [a,b]$, where $\mu$ is the Lebesgue measure on $[a,b]$.
%  \item $E\int_{a}^{b}\left(\sup\limits_{|x|\le r}|F(\omega, t, x)|\right)^2<\infty.$
%\end{enumerate}
%Indeed, in this case
%
%
%
%The function $f$ can be replaced by a random continuous Volterra operator $V(\omega)$, which for any $u\in C_T$ are $\mathcal F_t$-adapted stochastic processes
\end{example}
\subsection{Weak solutions of stochastic equations}\label{sec_ex_weak_sol}
\begin{example}\label{ex-weak-sol-1}
  Consider the initial value problem for an ordinary stochastic differential equation with random coefficients
  \begin{equation}\label{eq_init_value_problem-0}
    dx(t)=f^0(t,x(t))dt+\sum\limits_{j=1}^mf^j(t,x(t))dW_j(t) \ (t\in [a,b]) \ \ \ \mbox{and} \ \ \ x(a)=x_0,
  \end{equation}
   where $f^j$ satisfies the conditions that are similar to those listed in Example  \ref{ex-local-comp}:
\begin{enumerate}
  \item $f^j(\cdot, \cdot, x)$ is $\mathcal F\otimes \mbox{Bor}([a,b])$-measurable for all $x\in R^n, j=0,..,m$;
    \item $f^j(\cdot, t, x)$ is $\mathcal F_t$-adapted for any $t\in [a,b]$ and $x\in R^n, j=0,..,m$;
  \item $f^j(\omega, t, \cdot)$ (j=0,..,m) is continuous for $P\otimes\mu$-almost all $(\omega, t)\in \Omega\times [a,b]$, where $\mu$ is the Lebesgue measure on $[a,b]$.
  \item $|f^j(\omega, t, x)|\le C_j(\omega, t)$ $P\otimes\mu$-almost everywhere, where $C_0(\omega, \cdot)\in L^{r_0}[a,b]$ a.s. and $C_0(\omega, \cdot)\in L^{2r_j}[a,b]$ a.s. ($j=1,...,m$) for some $r_j>1$ ($j=0,...,m$).
  %$\int_{a}^{b}(\sup\limits_{|x|\le r}|f^j(\omega, t, x)|)^{r_jdt<\infty$ a.s.
\end{enumerate}
and $W_j$ are standard scalar Wiener processes (not necessarily independent) on the stochastic basis (\ref{eq-stoch-basis}).

The claims are that under the above assumptions on $f^j$ the initial value problem (\ref{eq_init_value_problem-0}) has at least one weak solution $x$ on the interval $[a,b]$ for any $\mathcal F_a$-measurable random point $x_0$ and that this solution has continuous paths on $[a,b]$. If it is a priori known that the problem (\ref{eq_init_value_problem-0}) has at most one weak solution for any Young expansion of the stochastic basis (\ref{eq-stoch-basis}), then $x$ is, in fact, strong, i.e. it is defined on the stochastic basis (\ref{eq-stoch-basis}) for all $t\in[a,b].$

To prove these claims let us consider the operator
\begin{equation}\label{eq-operator-h-ex-1}
(hx)(t)=x_0+\int_{a}^{t}f^0(s,x(s))ds+\sum\limits_{j=1}^m\int_{a}^{t}f^j(s,x(s))dW_j(s)
\end{equation}
in the space $\mathcal Pa(X)$, where $X=C[a,b]$, and check the assumptions of Corollary \ref{cor-fixed-point-general} (or Theorem \ref{th-fixed-point}, a particular case of this corollary).

Using the information from the examples of this section we obtain that
\begin{itemize}
  \item the corresponding projective system, generated by the space $X=C[a,b]$ satisfies {Property ($\Pi$)}, see Example \ref{ex-Pi-for-C}.
\item The integral superposition operator $(I_0x)(t)\equiv \int_{a}^{t}f^0(s,x(s))ds$ is an LC operator in the space $\mathcal P(X)$ to $\mathcal P(X)$, see Example \ref{ex_Carath}; as $f^0(\cdot, \cdot, x)$ is adapted for each $x\in R^n$, then $I_0$ maps $\mathcal Pa(X)$ into itself;
  \item the superposition operator $I_0: \mathcal Pa(X)\to \mathcal Pa(X)$ is tight-range, as the integral operator generating $I_0$ is compact-range in the space $C[a,b]$ due to assumption (4), see Example \ref{ex-Carat-tight};
  \item the integral operators $(I_jx)(t)\equiv\int_{a}^{t}f^j(s,x(s))dW_j(s)$ are LC operators in the space $\mathcal Pa(X)$ to $\mathcal Pa(X)$, see Examples \ref{ex-local-comp};
    \item the operators $I_j: \mathcal Pa(X)\to \mathcal Pa(X)$ are tight-range, because $I_j(\mathcal Pa(X))=I_j(\mathcal A)$, where the set $$\mathcal A=\{x\in\mathcal Pa(X): \ \|x(\omega)\|_{L^{2r_j}}\le \|C_j(\omega)\| \ \mbox{a.s.}\}$$ is bounded in the space $\mathcal Pa(L^{2r_j})$ and $r_j>1$, see Example \ref{ex-integral-tight};
  \item the operator $h: \mathcal Pa(X)\to \mathcal Pa(X)$ is a local and tight-range operator as a sum of such operators, see Example \ref{ex-linear-comb-tight}.
  \end{itemize}
Therefore, the operator $h$ has at least one weak fixed point $x^*$ in the space $\mathcal Pa(X)=\mathcal Pa(C[a,b])$. This fixed point will be a weak, path-continuous solution of the initial value problem (\ref{eq_init_value_problem-0}) on the interval $[a,b]$. If, in addition, this initial value problem is known to have at most one local solution on the interval $[a,b]$ for any Young expansion of the stochastic basis (\ref{eq-stoch-basis}), then the operator $h$ has at most one weak fixed point in the space $\mathcal Pa(X)$. Applying Corollary \ref{cor-fixed-point-general}, we get a unique strong solution of the problem (\ref{eq_init_value_problem-0}) on the interval $[a,b]$.
\end{example}

If the right-hand sides of the equation in (\ref{eq_init_value_problem-0}) are not bounded, then the solutions of it may not be defined on the entire interval $[a,b]$. In this case, we will need a notion of a local solution, i.e. a solution defined on some random subinterval. Such local solutions may be than extended either to the interval $[a,b]$, or they explode within a finite random interval. The proof is based on the iterated application of the fixed-point principle and, therefore, on an infinitely repeated Young expansions of the original stochastic basis, as it is constructed in Example \ref{ex-limit-stoch-bases}. Below we illustrate this procedure by using more general stochastic equations and the space $L^1[a,b]$ instead of $C[a,b]$, which allows to relax assumptions on the right-hand sides (because the tightness conditions are weaker in $L^1[a,b]$, see Example \ref{ex-integral-tight}).

\begin{example}\label{ex-weak-sol-2}
  Consider the initial value problem
  \begin{equation}\label{eq_init_value_problem}
    dx(t)=(V^0x)(t)dt+\sum\limits_{j=1}^m(V^j)(tx)dW_j(t) \ (t\in [a,b]) \ \ \ \mbox{and} \ \ \ x(a)=x_0,
  \end{equation}
  where $W_j$ are the same as in the previous example, $V^j$ are the superposition operators generated by random, continuous Volterra operators $V_\omega^j: \ \mathcal Pa(X)\to\mathcal Pa(Y^j)$ ($j=0,...,m$), which  satisfy the measurability conditions with respect to the filtration $(\mathcal F_t)$ from Example \ref{ex-local-comp}, and
  $X=C[a,b]$, $Y^0=L^{1}[a,b]$ and $Y^j=L^2[a,b]$ ($j=1.,,,.m$.)
  %In particular, $V^j$ may be defined as $f^j$ in Example  \ref{ex-weak-sol-1}.

  Then we have the following statements:
    \begin{enumerate}
    \item  the initial value problem (\ref{eq_init_value_problem}) has at least one weak local, path-continuous solution  for any $\mathcal F_{a}$-measurable random point $x_0$ in $R^n$, i. e. a solution defined on some Young expansion of the stochastic basis $\mathcal B$ and some random subinterval;
    \item if the absolute values of all weak local solutions of (\ref{eq_init_value_problem}) are known to be bounded in probability, then these solutions are defined for all $a\le t\le b$, i.e. $\tau=b$ a. s.;
     \item if for any Young expansion $\mathcal B^*$ of the stochastic basis $\mathcal B$ the initial value problem (\ref{eq_init_value_problem}) has at most one weak solution on $[a,b]$, then any such a is strong, i.e. defined on the original stochastic basis $\mathcal B$ for all $t\in [a,b]$.
  \end{enumerate}

  The proof of statements (1)-(3) is based again on Corollary \ref{cor-fixed-point-general} (or Theorem \ref{th-fixed-point}, a particular case of this corollary). Let us start with the first statement. To define a tight-range LC operator in the space $\mathcal Pa(L^1[a,b])$ we first define the random, $\mathcal F_a$-measurable in $\omega$ and continuous projections $\kappa_\omega^1$ of the space $R^n$ onto the ball $B_1$ of radius $1$ centered at $x_0(\omega)$ and define $V_\omega^{j,1}x=V_\omega^j(x(\omega)\circ\kappa_\omega^1)$. By construction, the operator $V_\omega^{j,1}$ is random continuous Volterra operator acting from $\mathcal Pa(L^1[a,b])$ to $\mathcal Pa(L^1[a,b])$ ($j=0$) and $\mathcal Pa(L^2[a,b])$ ($j=1,..,m$), respectively, and satisfying the same measurability conditions with respect to the filtration $\mathcal F_t$ as the operators  $V_\omega^j$. Defining $h^1$ by
  $$
  (h^1x)(t)=x_0+\int_{a}^{t}({V^{0,1}}x)(s)ds+\sum\limits_{j=1}^m\int_{a}^{t}({V^{j,1}}x)(s)dW_j(s)
  $$
  and using the tightness property of the It\^{o} integral from Example \ref{ex-integral-tight}, the compactness of the Lebesgue integral as an operator in $L^1[a,b]$, together with Example \ref{ex-Carat-tight}, we see that $h^1$ is a tight LC operator in the space $\mathcal Pa(L^1[a,b])$. Moreover, it is tight-range,  as it maps the space $\mathcal Pa(L^1[a,b])$ onto the set $h^1(\mathcal A)$, where $\mathcal A=\{x\in \mathcal Pa(L^1[a,b]): |x(\omega,t)-x_0(\omega)|\le 1 \ \mbox{a.s.}\}$, which is bounded in the space $\mathcal Pa(L^1[a,b])$.
  By Theorem \ref{th-fixed-point}, there exists a Young expansion
  $
  \mathcal B^1=(\Omega^1, \mathcal F^1, \mathcal F_t^1, P^1)
  $
  of the stochastic basis $\mathcal B,$ where $\Omega^1=\Omega\times Z$ and $Z=L^1[a,b]$ and a weak fixed point $x_1\in \mathcal Pa(L^1[a,b], \mathcal B^1)$ of the operator $h^1$. Notice that $|x_1-x_0|\le 1$ $P^1$-a.s. by construction and that  this solution, in fact, has continuous trajectories. Hence the stopping time $\tau_1(t)=\inf\{t: |x_1(t)-x_0|>1\}$ is well-defined and $\tau_1>a$ a.s., so that the restriction of $x^1$ to the random interval $[a,\tau_1]$ solves the initial value problem  (\ref{eq_init_value_problem}) on this interval.  For the sake of simplicity, we may still denote this solution by $x_1$.
 This proves the first part of the theorem.

  To prove the second statement, we iterate the above procedure by induction. If $\nu\ge 2$ and $x_{\nu-1}$ is an already constructed weak solution defined on a Young expansion
  $$
  \mathcal B^{\nu-1}=(\Omega^{\nu-1}, \mathcal F^{\nu-1}, \mathcal F_t^{\nu-1}, P^{\nu-1})
  $$
  for all $t\in [a, \tau_{\nu-1}]$ and satisfying $|x_{\nu-1}-x_0|\le \nu-1$ $P^{\nu-1}$-a.s. Here $\tau_{\nu-1}$ is some stopping time on $\mathcal B^{\nu-1}$ and $\Omega^{\nu-1}$ is the direct product of $\Omega$ and ${\nu-1}$ copies of the space $Z=L^1[a,b].$ Put
  $$
  \tilde x(t)=\left\{
\begin{array}{ll}
  x(t) & (t\ge \tau_{\nu-1}) \\
  x_{\nu-1}(t) & (t<\tau_{\nu-1})
\end{array}
\right.
  $$
  and
  define the LC operator
  %I_{\{s\ge \tau_{\nu-1}\}}(s)
    $$
  (h^\nu x)(t)=x_0+\int_{a}^{t}({V^{0,\nu}}x)(s)ds+\sum\limits_{j=1}^m\int_{a}^{t}({V^{j,\nu}}x)(s)dW^{\nu-1}_j(s),
  $$
  where $W_j^{\nu-1}$ are the standard Wiener processes on $\mathcal B^{\nu-1}$ and
  $V_\omega^{j,\nu}x$ are random continuous Volterra operators given by
  $$
  (V_\omega^{j,\nu}x)(t)=\left\{
\begin{array}{ll}
  (V_\omega^j(\tilde x(\omega)\circ\kappa_\omega^\nu))(t) & (t\ge \tau_{\nu-1}) \\
  (V_\omega^j(x_{\nu-1}(\omega)))(t)  & (t<\tau_{\nu-1})
\end{array}
\right.
  $$
  the random continuous projections $\kappa_\omega^\nu$ of the space $R^n$ onto the ball $B_\nu$ of radius $\nu$ centered at $x_0(\omega)$. By construction, the operators $V_\omega^{j,\nu}$ satisfy the same measurability conditions with respect to the filtration $(\mathcal F_t^{\nu-1})$ as the operators  $V_\omega^j$ do for the filtration $(\mathcal F_t)$. Therefore, $V^{0,\nu}: \mathcal Pa(L^1[a,b], \mathcal B^{\nu-1})\to \mathcal Pa(L^1[a,b], \mathcal B^{\nu-1})$ and $V^{j,\nu}: \mathcal Pa(L^1[a,b], \mathcal B^{\nu-1})\to\mathcal Pa(L^2[a,b], \mathcal B^{\nu-1})$ ($j=1,..,m$).

  The LC operator $h^\nu$ is tight-range exactly by the same reasons as the operator $h^1$, so that it has a weak fixed point $x_\nu$ defined on a Young expansion $\mathcal B^{\nu}$ of the stochastic basis $\mathcal B^{\nu-1}$. As before,  $x_\nu$ has continuous paths, so that $\tau_\nu=\inf\{|x_\nu-x_0|>\nu\}$ is well-defined and satisfies $\tau_\nu>\tau_{\nu-1}$ a.s. Therefore, it gives rise to a local solution defined on $\mathcal B^{\nu}$ for all $t\in [t_0, \tau_\nu]$. We denote this solution by $x_\nu$ as well. By construction it a.s. coincides with $x_{\nu-1}$ on the random interval $[a, \tau_{\nu-1}]$ and satisfies $|x_\nu-x_0|\le \nu$ $P^\nu$-a.s. The induction argument is completed.

  By letting $\nu\to\infty$ we obtain the stochastic basis
  $$
  \mathcal B^*=(\Omega^*, \mathcal F^*, \mathcal F^*_t, P^*)
  $$
  to be the limit of the sequence of the stochastic bases $\mathcal B^\nu$ in the sense of Example \ref{ex-limit-stoch-bases}. Clearly,  all $\tau_\nu$ remain stopping times on $\mathcal B^*$. Let us, therefore, put  $x^*(t)=x_\nu(t)$ if
    $t\in (\tau_{\nu-1}, \tau_\nu]$, $\nu \in N$. Evidently, this stochastic process is sample continuous on the random interval $[t_0, \tau)$, where $\tau=\sup\limits_{\nu\in N}\tau_\nu$ is a stopping time on $\mathcal B^*$ and satisfies the initial value problem (\ref{eq_init_value_problem}), where $W_j$ are replaced by the standard Wiener processes $W_j^*$ on $\mathcal B^*.$
        Moreover, by construction $\{\tau= b \}$ if and only if $|x^*|<\infty$ $P^*$-a.s. This means that if the $\sup$-norm of all local weak solutions of the problem (\ref{eq_init_value_problem}) on the interval $[a,b]$ is a priori known to be bounded in probability, then $x^*(t)$ is a.s. defined for all $a\le t \le b$.

  Finally, we prove the third statement. To this end, let us assume that the problem (\ref{eq_init_value_problem}) admits at most one weak solution on the interval $[a,b]$ for any Young expansion $\mathcal B^*$ of the stochastic basis $\mathcal B$. This means that the integral operator
    $$
(hx)(t)=x_0+\int_{a}^{t}({V^{0}}x)(s)ds+\sum\limits_{j=1}^m\int_{a}^{t}({V^{j}}x)(s)dW_j(s)
    $$
    which is local and uniformly continuous on tight subsets of the space $\mathcal Pa(L^1[a,b])$, has at most one weak fixed point in this space for  any acceptable expansion $\mathcal B^*$ of the stochastic basis $\mathcal B$ as well, then by Theorem  \ref{th-fixed-point} any weak fixed point of this operator must be strong. This fixed point will be a strong solution of the initial value problem (\ref{eq_init_value_problem}) defined for all $a\le t \le b$.
\end{example}

%\begin{example}\label{ex-PDE}}
%    This example is inspired by the paper \cite{Hasen}.
%
%Let $\mathcal B$ be a stochastic basis on $T=[a,b]$, $H$ be a separable Hilbert space, $V$ be a separable Banach space, $X$ be a separable Banach function space consisting of some c\`{a}dl\`{a}g functions $\eta: T\to V$,  and $W(t)$ ($t\in T$) be a Wiener process in $H$ with a nonnegative, symmetric trace class covariance operator in $H$. The conjugate Banach space $X'$ is assumed to be compactly embedded into $X$.
%
%
%
%   Consider the equation
%\begin{equation}\label{eq_Partial_parameter}
%dx(t) = f(t,x(t),u(t))dt+F(t,x(t),u(t))dW(t),\ \ \  t\in [a,b],
%\end{equation}
%where the solutions $x$ and the controls $u$ belong to the same Banach space $\mathcal M$ of $\mathcal B$-adapted stochastic processes with trajectories from the space $X$ (up to a $P$-zero set).
%Assume that the mappings $f: T\times X \times X\to X$,
%$F: T\times X\times X \to L_{HS}(H,V)$, where $L_{HS}(H,V)$ is the space of Hilbert-Schmidt operators from $H$ to $V$, are chosen to guarantee that for any $u\in\mathcal M$ and $\mathcal F_a$-measurable $x_0: \Omega\to V$ there is a unique solution $x\in \mathcal M$ of Eq. (\ref{eq_Partial_parameter}).
%
%
%
%\end{example}

\subsection{Counterexamples}\label{sec_counterex}

\begin{example}
  There exists a complete probability space $S$, a closed, convex, bounded and nonempty subset $\Xi$ of the space $\mathcal P(R^2, \mathcal S)$ and an LC operator $h: \ \Xi \to\Xi$ such that the equation $hx=x$ has no solutions in $\Xi$. This explains why we need additional assumptions on the invariant subset $\Xi$ in the finite dimensional fixed-point theorem \ref{th-finite-dim}. For the proof of this result see  \cite{Pon-2}.
\end{example}

\begin{example}
There exists a tight-range LC operator $h: \mathcal Pa(C[a,b])\to \mathcal Pa(C[a,b])$ with no strong fixed points in the space $\mathcal Pa(C[a,b])$. This justifies the notion of a weak solution, which always exists in this case (see Corollary \ref{cor-fixed-point-general}). The existence of such $h$ follows from the results of the paper \cite{Barlow}, where a stochastic ordinary differential equation with non-Lipschiz yet continuous coefficients and no strong solutions is constructed.
\end{example}

\end{document}